\documentclass[reqno,11pt,a4paper]{amsart}
\usepackage{pifont}
\usepackage[T1]{fontenc}
\usepackage{enumerate,enumitem}
\usepackage{textcomp}
\usepackage{verbatim}

\usepackage{pgffor}

\usepackage{esint}
\usepackage{pdfsync}
\usepackage[colorlinks=true,linkcolor=blue]{hyperref}%
\usepackage{graphicx}
\usepackage{epstopdf}
\usepackage[T1]{fontenc}
\usepackage{pifont}
\usepackage[english]{babel}

\usepackage{ulem}
\usepackage{amsmath,esint,color}
\usepackage{amsthm,latexsym,epsfig,graphicx,marvosym, mathrsfs,bigints,stmaryrd,textgreek,upgreek}
\usepackage{amsmath,stmaryrd}
\usepackage{relsize}
\usepackage{amsfonts}
\usepackage{amssymb}
\usepackage[headings]{fullpage}

\usepackage{tikz}
\usetikzlibrary{decorations.pathreplacing}

\usepackage{booktabs}

\allowdisplaybreaks

\usepackage{amsmath}
\usepackage{amsfonts}
\usepackage{amssymb}
\usepackage{amsthm}

\newcommand{\ZZ}{\mathbb{Z}}
\newcommand{\CC}{\mathbb{C}}
\newcommand{\NN}{\mathbb{N}}
\newcommand{\TT}{\mathbb{T}}

\newcommand{\RR}{\mathbb{R}}
\newcommand{\DD}{\mathbb{D}}

\newtheorem{theorem}{Theorem}
\newtheorem{proposition}{Proposition}
\newtheorem{lemma}{Lemma}
\newtheorem{remark}{Remark}

\newtheorem{definition}{Definition}

\newcommand{\dd}{\mathrm{d}}
\newcommand{\vertiii}[1]{{\vert\kern-0.25ex\vert\kern-0.25ex\vert #1
\vert\kern-0.25ex\vert\kern-0.25ex\vert}}

\begin{document}
\title{Doubly Connected V-States in Geophysical Models: A General Framework}
\author{Taoufik Hmidi}
\address{NYUAD Research Institute, New York University Abu Dhabi, PO Box 129188, Abu Dhabi, United Arab Emirates.}
\email{th2644@nyu.edu}
\author{Liutang Xue}
\address{School of Mathematical Sciences, Laboratory of Mathematics and Complex Systems (MOE), Beijing Normal University, Beijing 100875, P.R. China}
\email{xuelt@bnu.edu.cn}
\author{Zhilong Xue}
\address{School of Mathematical Sciences, Laboratory of Mathematics and Complex Systems (MOE), Beijing Normal University, Beijing 100875, P.R. China}
\email{zhilongxue@mail.bnu.edu.cn}

\date{\today}
{\thanks{L. Xue and Z. Xue have been partially supported by National Key Research and Development Program of China (No. 2020YFA0712900)
and National Natural Science Foundation of China (No. 12271045).
T. Hmidi has been supported by Tamkeen under the NYU Abu Dhabi Research Institute grant.}
}


\begin{abstract}
  In this paper, we prove the existence of doubly connected V-states (rotating patches)
  close to an annulus for active scalar equations with completely monotone kernels.
  This provides a unified framework for various results related to geophysical flows.
  This allows us to recover existing results on this topic while also extending to new models,
  such as the gSQG and QGSW equations in radial domains and 2D Euler equation in annular domains.
\end{abstract}

\maketitle

\tableofcontents

\section{Introduction}
We consider the two-dimensional (abbr. 2D) active scalar equation
\begin{align}\label{eq:ASE}
\begin{cases}
  \partial_t\omega+(u\cdot \nabla)\omega=0,\quad & (t,\mathbf{x})\in (0,\infty)\times \mathbf{D}\\
  u=\nabla^{\perp}\psi,\quad & (t,\mathbf{x})\in (0,\infty)\times \mathbf{D},\\
  \omega(\mathbf{x},0)=\omega_0(\mathbf{x}),\quad & \mathbf{x}\in \mathbf{D},
\end{cases}
\end{align}
where $\mathbf{D}$ is either the whole space $\RR^2$ or some radially connected domain (e.g. disc or annulus),
$\nabla^{\perp}=(\partial_2,-\partial_1)$, the vector $u=(u_1,u_2)$ is the velocity field,
$\omega$ is a scalar field understood as vorticity or temperature or buoyancy of the fluid,
and the stream function $\psi$ is prescribed through the following relation
\begin{align}\label{eq:stream_function}
  \psi(t,\mathbf{x})=\int_{\mathbf{D}}K(\mathbf{x},\mathbf{y})\omega(t,\mathbf{y})\dd \mathbf{y}.
\end{align}
Throughout our paper, we assume that the kernel can be decomposed as follows
\begin{align*}
  K(\mathbf{x},\mathbf{y})=K_0(|\mathbf{x}-\mathbf{y}|)+K_1(\mathbf{x},\mathbf{y})
\end{align*}
and satisfies the following conditions:
\begin{enumerate}[label=$(\mathbf{A}\theenumi)$,ref=$\mathbf{A}\theenumi$]
\item\label{A1}
{\it Complete monotonicity}:
$t\in (0,\infty)\mapsto -K'_0(t)$ is a nonzero completely monotone function (see Definition \ref{def:cmf} below), equivalently,
there exists a non-negative measure $\mu$ on $[0, \infty)$ such that $\mu\not\equiv 0$ and
\begin{align}\label{eq:K0prim}
  -K'_0(t)=\int_0^{\infty}e^{-tx }\dd \mu(x),\quad \forall t >0.
\end{align}
\item\label{A2}
{\it Integrability assumption}:
there exists
some $\alpha\in (0,1)$ such that\footnote{The upper limit 1 in \eqref{cond:K0} can be changed to any fixed number
$a_0\in (0,\infty)$, and they are indeed equivalent in view of Remark \ref{rmk:assumA2} below.}
\begin{equation}\label{cond:K0}
\begin{aligned}
  \int_0^1 |K_0(t)| t^{-\alpha+\alpha^2}\dd t < \infty.
\end{aligned}
\end{equation}
\item\label{A3} {\it{Regularity assumption}}: $K_1\in C^k_{\mathrm{loc}}({\mathbf{D}}^2)$ for some $k\geqslant 4$.\vskip1mm
\item\label{A4} {\it{Symmetry  assumption}}: we assume that for any $ \mathbf{x},\mathbf{y}\in\mathbf{D},$
\begin{align}\label{eq:symm-assum}
  K_1(\mathbf{x},\mathbf{y})= K_1(\mathbf{y},\mathbf{x}),
  \,K_1(\overline{\mathbf{x}},\overline{\mathbf{y}})=K_1(\mathbf{x},\mathbf{y}),\,
  K_1(e^{i\theta}\mathbf{x},e^{i\theta}\mathbf{y}) = K_1(\mathbf{x},\mathbf{y}), \forall \theta\in \RR,
\end{align}
where $\overline{\mathbf{x}} = (x_1, -x_2)$ is the reflection of $\mathbf{x} =(x_1,x_2)$.
\end{enumerate}
We shall frequently identify the complex plane $\mathbb{C}$ with $\mathbb{R}^2$.
It is important to observe that the symmetry assumption \eqref{A4} guarantees that the radial functions serve as stationary solutions for equation \eqref{eq:ASE}-\eqref{eq:stream_function}. 
\\[1mm]
\noindent
The authors in \cite{HXX23b} initially introduced this general model \eqref{eq:ASE}-\eqref{eq:stream_function}
under the assumptions \eqref{A1}-\eqref{A4},
which notably includes a variety of significant hydrodynamic models as particular examples.
\begin{itemize}
\item \textit{2D Euler equation in the whole space}. In this case, $\mathbf{D}=\mathbb{R}^2$, $\omega(t,\mathbf{x})$ is the vorticity of the fluid and
\begin{align}\label{eq:K-euler}
  \psi (t,\mathbf{x}) = (- \Delta)^{-1} \omega(t,\mathbf{x}),\quad K(\mathbf{x},\mathbf{y}) = K_0(|\mathbf{x}-\mathbf{y}|) = -\frac{1}{2\pi} \log |\mathbf{x} -\mathbf{y}|.
\end{align}
The 2D Euler equation serves as a core model in fluid mechanics,
describing the motion of an inviscid incompressible fluid.
\item
\textit{The generalized surface quasi-geostrophic (abbr. gSQG) equation in the whole plane}.
In this case, $\mathbf{D}=\mathbb{R}^2$, 
and
\begin{align}\label{eq:K-gSQG}
   \psi(t,\mathbf{x})= (-\Delta)^{-1 + \frac{\beta}{2}}\omega(t,\mathbf{x}),\quad K(\mathbf{x}, \mathbf{y}) = K_0(|\mathbf{x}-\mathbf{y}|)
  = c_{\beta}|\mathbf{x}- \mathbf{y}|^{-\beta},\quad \beta\in (0,1),
\end{align}
with $c_{\beta}=\tfrac{\Gamma(\frac{\beta}{2})}{\pi 2^{2-\beta}\Gamma(1-\frac{\beta}{2})}$.
The $\beta=1$ case corresponds to the surface quasi-geostrophic
(abbr. SQG) equation, which is a simplified model to track the atmospheric circulation near the tropopause \cite{HPGS}
and the ocean dynamics in the upper layers \cite{LK06}.
It is important to mention that the SQG equation exhibits a formal deep analogy with the 3D incompressible Euler equations \cite{CMT94}.
The gSQG equation, introduced by C\'ordoba, Fontelos, Mancho, Rodrigo in \cite{CFMR}, serves as an interpolation model between 2D Euler and SQG equations.
\item
\textit{The quasi-geostrophic shallow-water (abbr. QGSW) equation in the whole space}. We have that
$\mathbf{D}= \mathbb{R}^2$, 
and
\begin{align}\label{eq:K-QGSW}
  \psi(t,\mathbf{x})=(-\Delta + \varepsilon^2)^{-1} \omega(t,\mathbf{x}),\quad
  K(\mathbf{x},\mathbf{y})=K_0(|\mathbf{x}-\mathbf{y}|)=\frac{1}{2\pi }\mathbf{K}_0(\varepsilon|\mathbf{x}-\mathbf{y}|),
\end{align}
with $\varepsilon>0$, $\mathbf{K}_0(\cdot)$ the zero-th modified Bessel function (see \eqref{def:Kn-z}).
The QGSW equation is derived asymptotically from the rotating shallow water equations in the limit of fast rotation
and small variation of free surface \cite{Val08}.
\item \textit{The 2D Euler equation in the radial domain $R \mathbb{D}$}. Here, $\mathbb{D}\triangleq B(\mathbf{0},1)$,
$R\mathbb{D} = B(\mathbf{0},R)$.
We impose the slip (non-penetration) boundary condition $u\cdot \mathbf{n} =0$ on $\partial B(\mathbf{0},R)$,
where $\mathbf{n}$ denotes the outward unit normal vector.
It holds that $\mathbf{D} = R \mathbb{D}$, the stream function $\psi$ solves the  Dirichlet boundary value problem of Poisson equation
\begin{align*}
  -\Delta\psi = \omega\quad \text{in}\,\,\mathbf{D},\qquad \psi|_{\partial \mathbf{D}}=0,
\end{align*}
and the kernel function $K(\mathbf{x},\mathbf{y})$ corresponds to the Green function given by
\begin{align}\label{eq:Green_Euler}
  K(\mathbf{x},\mathbf{y})=-\frac{1}{2\pi}\log |\mathbf{x}-\mathbf{y}|+\frac{1}{2\pi}\log |R-R^{-1}\mathbf{x}\overline{\mathbf{y}}|,
  \quad \mathbf{x},\mathbf{y}\in \mathbf{D},
\end{align}
with $\overline{\mathbf{y}} = (y_1,-y_2)$ the conjugate of $\mathbf{y}$.
\item \textit{The gSQG equation and QGSW equation in the radial domain $R \mathbb{D}$, and the 2D Euler equation in the annular domain
$R_2 \mathbb{D}\setminus \overline{R_1\mathbb{D}}$}.
The slip (non-penetration) boundary condition is also imposed.
The formulas for $K(\mathbf{x},\mathbf{y})$
in these three models are complex, but all satisfy assumptions \eqref{A1}-\eqref{A4}
(see Section \ref{sec:applications} for more details).
\end{itemize}

\noindent
Intense efforts have been dedicated in recent decades to the mathematical analysis of the active scalar equation
\eqref{eq:ASE}-\eqref{eq:stream_function}, with particular focus on the 2D Euler equation and the gSQG equation.
The 2D Euler equation is globally well-posed both in the whole space $\mathbb{R}^2$
and in domains that have sufficiently smooth boundaries.
For the smooth solution, one can refer to \cite{MB02,MP94}. For the  patch solution,
which is a weak solution of equation \eqref{eq:ASE}-\eqref{eq:stream_function} with
$\omega_0(\mathbf{x}) = \mathbf{1}_D(\mathbf{x})$ for some bounded smooth domain $D\subset \mathbb{R}^2$,
one can see \cite{Yud63} for the global existence and uniqueness result.
See also \cite{Chemin93,BerC93,KRYZ,Kiselev19}
for solving the contour dynamics problem on the global persistence of
$C^{k,\gamma}$-patch boundary regularity with $k\geqslant 1$ and $\gamma\in (0,1)$.
The ill-posedness results in various function spaces for the 2D Euler equation
can be found in \cite{BouL15,BouL15b,ElgM20,ElgJ17} (regular solution) and \cite{KL23a} (patch solution).
The gSQG equation with $\beta\in (0,2)$ is more subtle in the global well-posedness issue.
The local well-posedness results in the Sobolev framework were explored by \cite{CCCG12,ConN18a}
for the regular solution and by \cite{Rod05,Gan08,CCCG12,KYZ17,GanP21,Ber-Gan}
for the patch solution.
Considering the patch solution composed of two patches with different sign, the
finite-time singularity of such a locally well-posed solution
was proved for the gSQG equation with $\beta\in (0, \frac{1}{2}]$ in the half-plane with slip boundary condition
\cite{KRYZ,GanP21,Zlat23}.
This singularity result also extends to the smooth solution framework \cite{JKY25,Zlat23}.
However, so far the global well-posedness issue for the gSQG equation with any
$\beta\in (0,2)$ in whole space $\mathbb{R}^2$ or torus $\mathbb{T}^2$ remains an open problem. For the gSQG equation, one can also refer to \cite{Resnick,Mar08,ConN18b,NHQ18} for the global existence of $L^2$-weak solution,
as well as \cite{Cor-Zo21,Cor-Zo24,KL23} for the ill-posedness results
in the H\"older or Sobolev spaces.
\vskip1mm
\noindent Here, we mainly focus on a specific class of patch solutions,
usually known as \textit{V-states}, or \textit{relative equilibria} or
\textit{rotating patches}. Their shape is not altered during the motion
and can be described via a rigid body transformation.
By assuming the center of rotation is the origin, the V-states take the form
\begin{align}\label{eq:Vstates}
  \omega(t,\mathbf{x}) = \mathbf{1}_{D_t}(\mathbf{x}),
  \quad \textrm{with}\;\; D_t = e^{i \Omega t} D,
\end{align}
where $D \subset \RR^2$ is a bounded smooth domain and compactly embedded in $\mathbf{D}$.
The real number $\Omega$ is called the \textit{angular velocity} of the rotating domain.
The V-states is referred to as \textit{simply connected} or \textit{doubly connected}
if the rotating domain $D$ is simply or doubly connected, respectively.
We also call the V-states \textit{$m$-fold symmetric} if $e^{i\frac{2\pi }{m}} D = D$,
$m\in\mathbb{N}^\star$.

\vskip0.5mm
\noindent
The study of V-states for the planar Euler equation (i.e. \eqref{eq:ASE}-\eqref{eq:stream_function} and \eqref{eq:K-euler})
goes back to Kirchhoff \cite{Kirch} in 1874,
where he proved that any ellipse with semi-axis $a$ and $b$ rotates uniformly with the angular velocity $\Omega=\frac{ab}{a^2+b^2}$.
Approximately a century later, Deem and Zabusky \cite{DZ78} performed numerical experiments showcasing
the existence of V-states with more $m$-fold symmetry.
Burbea \cite{Burbea82} analytically justified the existence of simply connected V-states with
$m$-fold ($m\geqslant 2$) symmetry. His original approach is to apply local bifurcation tools to
construct local curves of solutions bifurcating from Rankine vortices (i.e. $\mathbf{1}_{\mathbb{D}}(\mathbf{x})$)
at the angular velocity $\Omega=\frac{m-1}{2m}$.
Hmidi, Mateu, and Verdera \cite{HMV13} revisited this construction with more rigor
and also showed that the boundaries of such implicit V-states are $C^\infty$ smooth and convex near the equilibrium.
The boundary analyticity of the constructed V-states was explored by Castro, C\'ordoba, and G\'omez-Serrano \cite{CCG16b},
and its global bifurcation version was discussed by Hassainia, Masmoudi, and Wheeler \cite{HMW20}.
As for the doubly connected V-states of 2D Euler equation, De la Hoz, Hmidi, Mateu, and Verdera \cite{FHMV16}
proved the existence of V-states with $m$-fold ($m\geqslant 3$) symmetry bifurcating from annuli
(i.e. $\mathbf{1}_{\mathbb{D}\setminus \overline{b\mathbb{D}}}(\mathbf{x})$, $0<b<1$) at simple eigenvalues.
The existence of nontrivial simply and doubly connected V-states
for the 2D Euler equation on the disc
(with slip boundary condition) was addressed by De la Hoz, Hassainia, Hmidi, and Mateu \cite{DHHM}.
For other related research about simply or doubly connected V-states for the planar Euler equation,
we refer to \cite{ABCD+,BHM23,CCG19,CJS24,Garcia20,Garcia21,GHJ20,GHM23,GPSY,HHM23+,HHR23,HHR24,HM16,HM16b,HM17,HW22,HR22,HT24,WXZ22,WZZ23,Vishik18a,Vishik18b} and references therein.
\\[0.5mm]
\noindent
The construction of simply or doubly connected $m$-fold symmetric V-states for the gSQG equation
(i.e.  \eqref{eq:ASE}-\eqref{eq:stream_function} and \eqref{eq:K-gSQG})
with $\beta\in (0,1)$ was given by Hassainia, Hmidi \cite{HH15} and De la Hoz, Hassainia, Hmidi
\cite{FHH16}, respectively. These are global-in-time time-periodic patch solutions for gSQG equation,
although the global regularity issue of patch solutions generally remains open.
The method is also to conduct the Crandall-Rabinowitz theorem, but the spectral analysis is more involved and complicated compared with the 2D Euler case.
Later, Castro, C\'ordoba, and G\'omez-Serrano \cite{CCG16}
proved the existence of simply connected convex smooth V-states for the gSQG equation with $\beta\in [1,2)$
(see \cite{CCG16b} for the boundary analyticity).
By finding a decomposition of the Green function $K(\mathbf{x},\mathbf{y})$
for the spectral fractional Laplacian in bounded domains, 
the authors in \cite{HXX23a} showed the existence of $m$-fold symmetric V-states bifurcating from
$\mathbf{1}_{b\mathbb{D}}(\mathbf{x})$ for the gSQG equation in the unit disc $\mathbb{D}$, either for $m\in \mathbb{N}^\star$ large
or $b\in(0,1)$ small or $\beta\in(0,1)$ small.
One can refer to \cite{CCG20,CFMS25,Garcia20,Garcia21,GSIP,GPSY,HHM23,HM17,HW22} for other related works about V-states for the gSQG equation.
\\[0.5mm]
\noindent
The mathematical study of V-states for the QGSW equation (i.e.  \eqref{eq:ASE}-\eqref{eq:stream_function} and \eqref{eq:K-QGSW})
was initiated by Dritschel, Hmidi, and Renault \cite{DHR19}.
They investigated the analytical and numerical aspects of the bifurcation diagram of simply connected rotating vortex patch equilibria
and proved in particular  the existence of $m$-fold symmetric simply connected V-states in the whole space.
This existence result of V-states for sufficiently large $m\in \mathbb{N}^\star$ also holds for the QGSW equation in radial domain,
see \cite[Theorem 1.2]{HXX23b}.
Roulley \cite{Rou23a} established the existence of doubly connected V-states in the planar case provided that the symmetry is large enough.
One can refer to \cite{HR21} for the construction of quasi-periodic simply-connected V-states for the QGSW model using KAM tools.
\\[0.5mm]
\noindent
Note that all the previous work on the V-states of the 2D Euler, gSQG and QGSW equations relies on the
explicit formula of kernel $K$, and a crucial step in verifying the assumptions of
Crandall-Rabinowitz's theorem requires the case-by-case study of the monotonicity of spectral sequence:
for example in the whole space,
\begin{align*}
  \lambda_n^{\mathrm{Euler}} = \frac{n-1}{2n},\quad
  \lambda_n^{\mathrm{gSQG}} = \frac{\Gamma(1-\beta)}{2^{1-\beta}\Gamma^2(1-\frac{\beta}{2})}
  \bigg(\frac{\Gamma(1+\frac{\beta}{2})}{\Gamma(2-\frac{\beta}{2})}
  -\frac{\Gamma(n+\frac{\beta}{2})}{\Gamma(n+1-\frac{\beta}{2})}\bigg),
\end{align*}
\begin{align*}
  \lambda_n^{\mathrm{QGSW}} = \mathbf{I}_1(\varepsilon)\mathbf{K}_1(\varepsilon)-\mathbf{I}_{n}(\varepsilon)\mathbf{K}_{n}(\varepsilon),
\end{align*}
with $\Gamma(\cdot)$ the Gamma function, and $\mathbf{I}_n(\cdot)$ and $\mathbf{K}_n(\cdot)$ the modified Bessel functions of the first and second kind (see \eqref{def:In-z}-\eqref{def:Kn-z}).
In \cite{HXX23b}, the authors were the first to investigate the V-states for the active scalar equation \eqref{eq:ASE}-\eqref{eq:stream_function} under general assumptions on $K$.
They proved the existence of $m$-fold symmetric simply connected rotating patches bifurcating from
$\mathbf{1}_{b\mathbb{D}}(\mathbf{x})$ ($b\mathbb{D}\subset\subset \mathbf{D}$)
for the convolution-type kernels assuming \eqref{A1}-\eqref{A2}.
They also extended their result for general kernels under assumptions \eqref{A1}-\eqref{A4} but with a large symmetry.
A key novelty of that work lies in factorizing the spectrum $\{\lambda_{n,b}\}$ for the convolution kernel case through the introduction of a universal function $\phi_n$, as follows
\begin{align}\label{def:lam-n-b-intro}
  \lambda_{n,b}   =  \int_0^{\infty}\phi_{n}(bx)\tfrac{\dd \mu(x)}{x},\quad \textrm{with}\quad
  \phi_{n}(x)  = 2\int_0^{\pi}e^{-2x \sin \eta} e^{i2n\eta}\dd \eta.
\end{align}
Then, observing that $\phi_n$ satisfies a second-order differential equation (see \eqref{eq:phi_n_ode})
and employing the comparison principle together with the Hankel transform,
we established the positivity and monotonicity of the mapping $n \mapsto \phi_n(x)$ for all $x>0$,
and also derived quantitative versions for these properties; see \eqref{ineq:est-phi-n}-\eqref{ineq:est-phi-n-diff}.
\vskip1mm
\noindent
In this paper, we extend our previous work \cite{HXX23b} by investigating the existence of doubly connected V-states for the active scalar equation
\eqref{eq:ASE}-\eqref{eq:stream_function}, under the general assumptions \eqref{A1}-\eqref{A4}.
Our main goal is to construct doubly connected V-states \eqref{eq:Vstates} bifurcating from the annulus
$\mathbf{1}_{\mathbb{D}\setminus \overline{b\mathbb{D}}}(\mathbf{x}), b\in (0,1)$ by performing bifurcation theory.
For consistency, we assume that
\begin{align*}
  \mathbb{D}\setminus \overline{b\mathbb{D}} \subset\subset \mathbf{D}.
\end{align*}
We aim to construct doubly connected V-states \eqref{eq:Vstates} bifurcating from this annular patch.
Specifically, we consider domains enclosed by two simple curves: the inner boundary parameterized by
\begin{align*}
  \theta \in \mathbb{R} \mapsto R_1(\theta)=\sqrt{b^2 + 2r_1(\theta)}\, e^{i\theta},
\end{align*}
and the outer boundary parameterized by
\begin{align*}
  \theta \in \mathbb{R} \mapsto R_2(\theta)=\sqrt{1 + 2r_2(\theta)}\, e^{i\theta}.
\end{align*}
Then, the contour dynamics equation of V-states with angular velocity $\Omega$ can be written as
\begin{align*}
  F_j(\Omega,\mathbf{r})\triangleq\Omega\, r'_j(\theta) + \partial_\theta \bigg(\int_0^{2\pi}
  \int_{R_1(\eta)}^{R_2(\eta)} K\big(R_j(\theta)e^{i\theta},\rho e^{i\eta}\big) \rho\, \dd \rho \dd \eta\bigg)=0, \quad j=1,2,\quad \mathbf{r}=(r_1,r_2).
\end{align*}
Since $\mathbf{1}_{\mathbb{D}\setminus \overline{b\mathbb{D}}}$ is a stationary solution,
the vectorial function $F\triangleq (F_1,F_2)$
satisfies $F(\Omega,\mathbf{0})=\mathbf{0}$ for any $\Omega\in \RR$.
In order to show the existence of non-trivial V-states,
it remains to check the hypotheses of the Crandall-Rabinowitz theorem
in some suitable Banach spaces.
\\
Compared to the analysis of simply connected V-states carried out in \cite{HXX23b}, the study of doubly connected V-states leads to a system composed of two coupled scalar equations. This significantly complicates the spectral analysis of the linearized operator $\partial_{\mathbf{r}}F(\Omega,\mathbf{r})$ at the equilibrium state $\mathbf{r}=\mathbf{0}$. In this case, the  spectrum  solves a dispersion equation given by the following quadratic equation
\begin{align}\label{eq:Omega-eq-intro}
  \Omega^2 - (A_{n,b} + B_{n,b})\, \Omega +  A_{n,b} B_{n,b} +  (\widetilde{\lambda}_{n,b} + \widetilde{p}_{n,b})^2 = 0,
\end{align}
where
\begin{align*}
  A_{n,b} = - V_b^1[\mathbf{0}] + \lambda_{n,b} + p_{n,b},\quad
  B_{n,b} = - V_b^2[\mathbf{0}] -\lambda_{n,1} + p_{n,1} ,
\end{align*}
the coefficient $\lambda_{n,b}$ is defined in \eqref{def:lam-n-b-intro} and
\begin{align}\label{def:tild-lamb-intro}
  \widetilde{\lambda}_{n,b}  = \int_0^{\infty}\phi_{n,b}(x) \frac{\dd \mu(x)}{x}, \quad \textrm{with} \quad
  \phi_{n,b}(x) \triangleq   \int_0^{2\pi} e^{-x \sqrt{1+b^2 - 2b\cos \eta}} e^{i n\eta} \dd \eta.
\end{align}
The coefficients
 $p_{n,b}$, $\widetilde{p}_{n,b}$ are defined from $K_1$ as \mbox{in \eqref{def:Pnb}.} However, $V_b^{1}[\mathbf{0}]$ and $V_b^{2}[\mathbf{0}]$ are given by
\begin{equation}\label{def:Vb1-intro}
\begin{split}
  V_b^{1}[\mathbf{0}] & = -\frac{1}{b}\widetilde{\lambda}_{1,b}
  + \lambda_{1,b}  + \frac{1}{b} \int_0^{2\pi} \int_b^1 \nabla_{\mathbf{x}}K_1(b,\rho e^{i\eta})\cdot (1,0)^T \,\rho\dd \rho \dd \eta,
\end{split}
\end{equation}
and
\begin{equation}\label{def:Vb2-intro}
\begin{split}
  V_b^2[\mathbf{0}]  &  = - \lambda_{1,1}  + b \widetilde{\lambda}_{1,b}
  + \int_0^{2\pi} \int_b^1 \nabla_{\mathbf{x}}K_1(1,\rho e^{i\eta}) \cdot (1,0)^T \,\rho\dd \rho \dd \eta.
\end{split}
\end{equation}
To obtain two distinct real roots from equation \eqref{eq:Omega-eq-intro},
which is the minimal requirement for applying the Crandall--Rabinowitz theorem,
its discriminant must be positive (see Section \ref{subsec:stable-mode} for more discussion). Specifically, we require
\begin{align}\label{eq:discriminant_intro}
  \mathbf{\Delta}_{n,b} \triangleq (A_{n,b}-B_{n,b})^2 - 4(\widetilde{\lambda}_{n,b} + \widetilde{p}_{n,b})^2>0.
\end{align}
Under this latter condition, the solutions $\Omega^+_{n,b}$ and $\Omega^-_{n,b}$
take the form
\begin{align}\label{eq:Omega_n_pm-intr}
  \Omega^+_{n,b} \triangleq   \frac{A_{n,b} + B_{n,b}}{2} + \frac{\sqrt{ \mathbf{\Delta}_{n,b}}}{2},\quad
  \Omega^-_{n,b} \triangleq  \frac{A_{n,b} + B_{n,b}}{2}- \frac{\sqrt{ \mathbf{\Delta}_{n,b}}}{2}.
\end{align}
Proving the monotonicity with respect to $n$ of the spectral sequences $(\Omega^\pm_{n,b})_{n\in\mathbb{N}^\star}$ is highly nontrivial, even in the simplified case when $K_1 \equiv 0$. Nevertheless, since the term $\widetilde{\lambda}_{n,b} + \widetilde{p}_{n,b}$ decays rapidly with respect to $n$, equation~\eqref{eq:Omega-eq-intro} can be approximated by the simpler form
\[
(\Omega - A_{n,b})(\Omega - B_{n,b}) = 0
\]
for sufficiently large $n$. Using perturbative arguments, the monotonicity property of the spectrum sequences can thus be reduced to establishing the monotonicity of $A_{n,b}$ and $B_{n,b}$ for large enough $n$, a task that is more tractable. Notice that the limit of the discriminant can be computed explicitly,
 \begin{align}\label{eq:basic_assum}
  \mathbf{\Delta}_{\infty,b} \triangleq \lim_{n\to \infty}\mathbf{\Delta}_{n,b} = \big( V_b^1[\mathbf{0}] - V_b^2[\mathbf{0}] \big)^2.
 \end{align}
\vskip1mm

Our main result states as follows.
\begin{theorem}\label{thm:general_kernel}
Suppose that $K(\mathbf{x},\mathbf{y})=K_0(|\mathbf{x}-\mathbf{y}|)+K_1(\mathbf{x},\mathbf{y})$ meets the conditions \eqref{A1}-\eqref{A4}.
Set
\begin{align}\label{def:S-set}
  \mathcal{S} \triangleq \Big\{b\in (0,1)\,:\, V_b^1[\mathbf{0}] - V_b^2[\mathbf{0}] \neq 0\Big\}\cap \mathcal{S}_{\mathrm{max}}, \quad
  \mathcal{S}_{\mathrm{max}} \triangleq \Big\{ b\in (0,1)\,:\,\mathbb{D}\setminus \overline{b \mathbb{D}} \subset\subset \mathbf{D} \Big\}.
\end{align}
Then, for any $b\in \mathcal{S}$,
there exists a sufficiently large number $m_0\in \mathbb{N}^\star$,
such that for any $m\geqslant  m_0$, the active scalar equation \eqref{eq:ASE}-\eqref{eq:stream_function} admits two families of $m$-fold
symmetric doubly connected V-states with $C^{2-\alpha}$ boundary bifurcating from the trivial solution
$\mathbf{1}_{\mathbb{D}\setminus \overline{b\mathbb{D}}}(\mathbf{x})$,
at two different angular velocities $\Omega^\pm_{m,b}$ given by \eqref{eq:Omega_n_pm-intr}.
In addition, if $K(\mathbf{x},\mathbf{y})=K_0(|\mathbf{x}-\mathbf{y}|)$, then the set $\mathcal{S}_{\mathrm{max}} \setminus \mathcal{S}$ is at most countable.
\end{theorem}
\noindent
This theorem enables to recover previous results obtained for various geophysical models, including Euler, gSQG and QGSW equations.
This connection will be explored in detail in Section \ref{sec:applications}.
Furthermore, it provides new results regarding the construction of doubly connected V-states for geophysical flows in different settings.
Notably, we shall consider the gSQG and QGSW equations in the specific case where the domain is a disc:
$\mathbf{D} = R\mathbb{D} = B(\mathbf{0},R)$, $R>1$.
\begin{theorem}\label{Cor:rad-domain}
  Consider the gSQG equation and the QGSW equation in the radial domain $\mathbf{D}= R\mathbb{D}$ with the slip boundary condition.
Then the set $\mathcal{S}= \mathcal{S}_{\mathrm{max}} =(0,1)$, and for any $b\in(0,1)$,
the existence result of $m$-fold symmetric doubly connected V-states as in Theorem \ref{thm:general_kernel}
holds.
\end{theorem}
\noindent The second new result is related to active scalar equations in annular domains,
and we illustrate this discussion through 2D Euler equation evolving in the domain
\begin{align}\label{Annular-d}
  \mathbf{D}= R_2 \mathbb{D}\setminus \overline{R_1 \mathbb{D}}, \quad 0< R_1<1<R_2\leqslant +\infty,
\end{align}
where the case $R_2 = +\infty$ should be understood as the exterior domain of a disc. In particular, by careful spectral analysis,
the integer $m_0$ in Theorem \ref{thm:general_kernel} in this case may be not large.
\begin{theorem}\label{Cor:Euler}
  Consider the Euler equation in the annular domain \eqref{Annular-d} with the slip boundary condition.
Then the set $\mathcal{S} =  \mathcal{S}_{\mathrm{max}} = (R_1,1)$,
and for every $b\in (R_1,1)$ and for every $m\geqslant 2$ satisfying $\mathbf{\Delta}_{m,b}>0$, equivalently,
\begin{align}\label{eq:Delta>0equa-intro}
  m > \tfrac{b^2}{(1-b^2)(b^2 + 2\mathfrak{C}_b)} \tfrac{1}{1- (\frac{R_1}{R_2})^{2m}}
  \Big( (1-R_1^{2m})(1-\tfrac{1}{R_2^{2m}}) + (1-\tfrac{R_1^{2m}}{b^{2m}}) \big(1-\tfrac{b^{2m}}{R_2^{2m}}
  +  2b^m(1- \tfrac{1}{R_2^{2m}})  \big) \Big),
\end{align}
with $\mathfrak{C}_b$ given by \eqref{def:Cb},
there exist two curves of $m$-fold symmetric doubly connected V-states with $C^{2-\alpha}$
boundary bifurcating from $\mathbf{1}_{\mathbb{D}\setminus \overline{b\mathbb{D}}}(\mathbf{x})$
at two different angular velocities $\Omega_{m,b}^\pm$.
\end{theorem}
Some remarks are listed below.
\begin{remark}
\begin{enumerate}
\item[(1)]
In the convolution case where  $K_1\equiv 0$, one gets from \eqref{def:lam-n-b-intro} and
\eqref{def:tild-lamb-intro},
\begin{align}\label{eq:good_decomposition}
  V_b^1[\mathbf{0}] - V_b^2[\mathbf{0}] = \int_0^\infty \tfrac{\Psi_b(x)}{x} \dd \mu(x),\quad
  \textrm{with}\;\; \Psi_b(x) =  \phi_1(x)+\phi_1(bx)
  -(b+\tfrac{1}{b})\phi_{1,b}(x).
\end{align}
According to Lemma $\ref{lem:kernel_function_limit}$, for sufficiently small values of $b$, the function $x \in (0,1) \mapsto \Psi_b(x)$ changes sign. Thus, by appropriately choosing the measure $\mu$, we can guarantee that the set $\mathcal{S}$ defined in \eqref{def:S-set} is empty.
However, we can in fact exhibit a broad class of measures $\mu$ for which the set $\mathcal{S}$ is nonempty;
see Proposition $\ref{prop:measure_examples}$ for further discussion.
Moreover, for all geophysical flows on the whole plane or the disc or the annulus considered in Section $\ref{sec:applications}$,
the set $\mathcal{S}$ always equals $\mathcal{S}_{\mathrm{max}}$.
Performing arguments on the real analyticity of the involved functions, we can show the dichotomy:
either $\mathcal{S}_{\mathrm{max}}\setminus \mathcal{S}$ is countable or $\mathcal{S}$ is an empty set.
We believe that the latter scenario can only occur when the kernel is identically zero.
\item[(2)]
In Section $\ref{sec:conclusion_discuss}$, we further explore various possibilities for the sign of $V_b^1[\mathbf{0}]$, $V_b^2[\mathbf{0}]$
and $\Omega^{\pm}_{m,b}$ (for $m$ large enough).
For each case, we construct illustrative examples of convolution-type kernels that satisfy the required conditions.
\item[(3)] For the 2D Euler equation in the annular domain \eqref{Annular-d}, by letting $R_1\rightarrow 0$
and $R_2 = R$, and arguing as deriving \eqref{eq:K-extdom-der}, the kernel function $K(\mathbf{x},\mathbf{y})$ given by \eqref{eq:K-Euler-annulus}
becomes the Green function \eqref{eq:Green_Euler}.
Since the analysis in $R\mathbb{D}\setminus \{0\}$ in the setting of double connected V-states
is essentially identical to that in $R\mathbb{D}$,
Theorem \ref{Cor:Euler} about the annular domain is a natural generalization of \cite[Theorem 6]{DHHM}
concerning the disc domain and one can view the latter result as the limit version of Theorem \ref{Cor:Euler}.
\end{enumerate}
\end{remark}

In the proof of Theorem \ref{thm:general_kernel}, the first key step is to establish the monotonicity in \(n\) of the sequences \((\Omega^\pm_{n,b})_{n\in \mathbb{N}^\star}\),
defined in \eqref{eq:Omega_n_pm-intr}, at least for sufficiently large \(n\).
This property is crucial to ensure a one-dimensional kernel of the linearized operator at the equilibrium state,
\(\partial_{\mathbf{r}}F(\Omega^\pm_{n,b},\mathbf{0})\).
Regarding the terms \(\lambda_{n,b}\) and \(\widetilde{\lambda}_{n,b}\), which are determined by the kernel \(K_0\), both can be factorized as in \eqref{def:lam-n-b-intro} and \eqref{def:tild-lamb-intro} using the universal functions \(\phi_n(x)\) and \(\phi_{n,b}(x)\), respectively.
The function $\phi_n(x)$ has been thoroughly studied in \cite{HXX23a} and was found to
satisfy the quantitative estimates \eqref{ineq:est-phi-n}-\eqref{ineq:est-phi-n-diff} below.
Using these estimates and assumptions \eqref{A1}-\eqref{A2},
we derive the crucial lower/upper bounds of $\lambda_{n,b}$ and
$\lambda_{n,b} - \lambda_{n+1,b}$ for every $n\in \mathbb{N}^\star$ in Proposition \ref{prop:lambda_property}.
Although $\phi_n(x) = \phi_{n,1}(x)$, so far we do not find an ordinary differential equation satisfied by $\phi_{n,b}(x)$
(in contrast to ODE \eqref{eq:phi_n_ode} by $\phi_n(x)$), so the techniques used in studying $\phi_n(x)$ as in \cite{HXX23a} will not work for $\phi_{n,b}(x)$.
Instead, we develop a new approach based on Chebyshev's polynomials to
derive the new integral expression formula of $\phi_{n,b}(x)$
and $\phi_{n,b}(x) - \phi_{n+1,b}(x)$ as respectively in
\eqref{exp:phi-nb} and \eqref{eq:diff-phi-n-b-2},
and then we prove the positivity and monotonicity of $n\mapsto \phi_{n,b}(x)$
for all $b\in (0,1]$ and $x>0$.
Furthermore, we show that $\widetilde{\lambda}_{n,b}$ is always positive
and decays very rapidly with respect to $n$. For the spectral terms $p_{n,b}$
and $\widetilde{p}_{n,b}$ defined by the kernel $K_1$,
by using the assumption \eqref{A3} and integration by parts,
they also decay sufficiently rapidly about $n$.
Consequently, based on these asymptotics of the above spectral terms,
we can verify the convergence in \eqref{eq:basic_assum};
moreover, under the condition that $V_b^1[\mathbf{0}]- V_b^2[\mathbf{0}] \neq 0$,
and especially noting that $\lambda_{n,b}-\lambda_{n+1,b}$ decays
much slower than $\widetilde{\lambda}_{n,b}$, $p_{n,b}$ and $\widetilde{p}_{n,b}$,
we manage to prove that for $n$ large enough, the spectrum sequence
$(\Omega_{n,b}^\pm)_{n\in \mathbb{N}^\star}$ is monotonic with respect to $n$
(see Proposition \ref{prop:spectral}).

\vskip0.5mm
\noindent
In verifying the spectral assumptions of Crandall-Rabinowitz's theorem, it remains to establish that \(\partial_{\mathbf{r}}F(\Omega^\pm_{n,b},\mathbf{0})\)
has codimension one and satisfies the transversality condition.
To achieve this, we apply in part a Mikhlin-type multiplier theorem (see Lemma \ref{lem:multiplier-lemma}).

\vskip0.5mm
\noindent
For the regularity analysis required in the proof of Theorem \ref{thm:general_kernel},
we note that the singular terms arising from the induced effects have been precisely handled in \cite{HXX23b}.
Moreover, the interaction terms are non-singular, making their treatment straightforward.
\vskip0.5mm
\noindent The last delicate point in the proof of Theorem \ref{thm:general_kernel} is to prove that
$\mathcal{S}_{\mathrm{max}}\setminus\mathcal{S}$ given by \eqref{def:S-set}
contains at most countable points in the special case $K(\mathbf{x},\mathbf{y})=K_0(|\mathbf{x}-\mathbf{y}|)$.
To this end, by studying the real analyticity in $b$ of the universal functions $\phi_n(bx)$ and $\phi_{n,b}(x)$,
we prove that the spectral terms $\lambda_{n,b}$ and $\widetilde{\lambda}_{n,b}$ are both real analytic in $b\in (0,1)$
(see Lemma \ref{lem:analytic_lambda}), which combined with \eqref{def:Vb1-intro}-\eqref{def:Vb2-intro}
implies that $V_b^1[\mathbf{0}] - V_b^2[\mathbf{0}]$ is also real analytic in $b$.
On the other hand, since the universal function \(\Psi_b(x)\) in \eqref{eq:good_decomposition} is positive for \( b = 0.5 \)
and all \( x > 0 \) (see Lemma \ref{lem:kernel_function_limit}), there exists at least one value \( b = 0.5 \) such that
\( V_b^1[\mathbf{0}] - V_b^2[\mathbf{0}] \) is strictly positive.
Consequently, the desired result follows from the properties of real analytic functions.
\vskip0.5mm
\noindent
To prove Theorem \ref{Cor:rad-domain}, we analyze the gSQG and QGSW models in the disc.
Using summation identities such as Sneddon's formula or Lemma \ref{lem:sum_up_QGSW},
we derive explicit expressions for \( V_b^1[\mathbf{0}] \) and \( V_b^2[\mathbf{0}] \) in terms of the modified Bessel functions
\( \mathbf{I}_n \) and \( \mathbf{K}_n \). Furthermore, we establish that \( V_b^1[\mathbf{0}] - V_b^2[\mathbf{0}] > 0 \)
for all \( b \in (0,1) \), and in these special cases, the desired existence result follows directly from \mbox{Theorem \ref{thm:general_kernel}.}
\vskip0.5mm
\noindent
In the proof of Theorem \ref{Cor:Euler}, the main challenge lies in deriving an explicit expression for Green's function of the annulus.
Fortunately, we have obtained an infinite series representation of Green's function, which allows us to apply
Theorem \ref{thm:general_kernel} through straightforward computations.
Moreover, by some careful spectral analysis in this special case,
we present an explicit and natural condition \eqref{eq:Delta>0equa-intro} on $m$ (the value of $m$ may be not large)
to show the existence of $m$-fold symmetric doubly connected V-staes.

\vskip1mm
\noindent The remainder of this paper is organized as follows.
In Section \ref{sec:eq-Vstates-linear}, we derive the boundary equation governing the doubly connected V-states,
and explore the linearization around the equilibrium state (annulus), along with the associated spectrum (linear dispersion).
In Section \ref{sec:universal-func}, we analyze key qualitative properties of universal functions
$x\mapsto \phi_{n,b}(x)$ and $x\mapsto\phi_n(bx)$ that appear in spectral terms $\lambda_{n,b}$ and $\widetilde{\lambda}_{n,b}$.
Section \ref{sec:spectral-stud} is dedicated to crucial spectral analysis.
We first establish the lower/upper pointwise bounds of certain spectral terms.
We then investigate the positivity of the discriminant $\mathbf{\Delta}_{n,b}$
and the distribution of spectra $\Omega^\pm_{n,b}$ for $n$ large enough.
In Section \ref{sec:proof-thm1}, we present a detailed proof of Theorem \ref{thm:general_kernel},
and Theorem \ref{thm:general-statement},
dealing with a more general statement on the existence of doubly connected V-states.
As immediate applications of Theorem \ref{thm:general_kernel}
and Theorem \ref{thm:general-statement},
we revisit and reprove in Section \ref{sec:applications} the known existence results of doubly connected V-states for 2D Euler,
gSQG and QGSW equations in the whole space as well as for the 2D Euler equation in the radial domain. Furthermore,
we establish Theorem \ref{Cor:rad-domain} and Theorem \ref{Cor:Euler} with regard to some new results to these models.
In Section \ref{sec:conclusion_discuss}, we further discuss the possible signs of \( V_b^1[\mathbf{0}] \), \( V_b^2[\mathbf{0}] \),
and \( \Omega^\pm_{n,b} \) for sufficiently large \( n \) as well as clarify the meaning of stable and unstable states.
Finally, some useful auxiliary lemmas are compiled in the Appendix.

\section{Boundary equation of doubly connected V-states and linearization}\label{sec:eq-Vstates-linear}
The primary goal is to derive the equations that govern the boundary motion of rotating doubly connected patches,
modeled by a nonlinear transport system that couples the two interfaces.
The secondary objective is to analyze the linearized operator around the annulus,
represented as a matrix Fourier multiplier, from which the dispersion relation can be obtained.

\subsection{Boundary equation}
Consider a vortex solution with doubly connected rotating domains, defined by \eqref{eq:Vstates}
and angular velocity $\Omega\in\mathbb{R}$. We assume that the boundary $\partial D_t$, which consists of two interfaces,
can be parameterized in polar coordinates as follows:
\begin{align*}
  \mathbf{z}(t,\cdot):\mathbb{T}&\mapsto \partial D_t, \\
  \theta&\mapsto \mathbf{z}(t,\theta)=\begin{pmatrix}
z_1(t,\theta) \\
z_2(t,\theta)
\end{pmatrix} = e^{i t \Omega}\mathbf{z}(\theta)\triangleq e^{it\Omega}
  \begin{pmatrix}
    \sqrt{b^2+2r_1(\theta)}e^{i\theta}\\
    \sqrt{1+2r_2(\theta)}e^{i\theta}
  \end{pmatrix},
\end{align*}
where $
\mathbf{z}(\theta)=\begin{pmatrix}
z_1(\theta) \\
z_2(\theta)
\end{pmatrix},$
 with
$z_1(\theta)\triangleq \sqrt{b^2+2r_1(\theta)}e^{i\theta}$ $(0<b<1)$ is the inner boundary of $D$
and $z_2(\theta)\triangleq \sqrt{1+2r_2(\theta)}e^{i\theta}$ represents the outer boundary of $D$.
It is a classical fact that, since the boundary is transported by the flow, the patch equations can be written as follows:
\begin{align*}
  \partial_tz_j(t,\theta) \cdot \mathbf{n}_j = u\big(t,z_j(t,\theta)\big) \cdot \mathbf{n}_j,\quad j=1,2,
\end{align*}
where $\mathbf{n}_j(t,z_j(t,\theta)) \triangleq i\,\partial_{\theta}z_j(t,\theta)$ ($j=1,2$)
denotes a normal vector of $\partial D_t$ at $z_j(t,\theta)$.
Note that
\begin{align*}
  u(t,z_j(t,\theta))\cdot \mathbf{n}_j = -\partial_\theta \big[\psi(t,z_j(t,\theta))\big],
\end{align*}
and
\begin{align*}
  \partial_tz_j(t,\theta)\cdot \mathbf{n}_j(t,z_j(t,\theta))
  =\text{Im}\Big(\partial_tz_j(t,\theta)\,\overline{\partial_{\theta}z_j(t,\theta)}\Big) = \Omega\, r'_j(\theta).
\end{align*}
Taking advantage of the symmetry assumption \eqref{eq:symm-assum}, we see that
\begin{align*}
  \psi(t,z_j(t,\theta)) & = \int_D K(e^{i t \Omega } z_j(\theta), e^{it \Omega} \mathbf{y})\, \dd \mathbf{y} \\
  & = \int_D K\big(z_j(\theta),\mathbf{y}\big) \,\dd \mathbf{y}.
\end{align*}
Hence, the rotating patch equations can be rewritten as
\begin{align}\label{eq:vortex-patch-eqs}
  \Omega\, r'_j(\theta) + \partial_\theta \bigg(\int_{D}K(z_j(\theta),\mathbf{y})\dd \mathbf{y}\bigg)=0, \quad j=1,2.
\end{align}
Using the polar coordinates yields
\begin{align*}
  \int_{D}K(z_j(\theta),\mathbf{y})\dd \mathbf{y}
  =\int_0^{2\pi}\int_{R_1(\eta)}^{R_2(\eta)} K\big(R_j(\theta)e^{i\theta},\rho e^{i\eta}\big) \rho\, \dd \rho \dd \eta,\quad j=1,2,
\end{align*}
with
\begin{align*}
  R_1(\theta)\triangleq \sqrt{b^2+2r_1(\theta)},\quad R_2(\theta)\triangleq \sqrt{1+2r_2(\theta)}.
\end{align*}
Therefore, by setting
\begin{equation}\label{def:F0-r}
\begin{split}
  \mathbf{r}(\theta) \triangleq
  \displaystyle{\begin{pmatrix}
	r_1(\theta)\\
	r_2(\theta)
  \end{pmatrix}},
  \quad F_0[\mathbf{r}](\theta) \triangleq
  \displaystyle{\begin{pmatrix}
	\int_0^{2\pi}\int_{R_1(\eta)}^{R_2(\eta)} K\big(R_1(\theta)e^{i\theta},\rho e^{i\eta}\big) \rho\, \dd \rho \dd \eta\\[1mm]
	\int_0^{2\pi}\int_{R_1(\eta)}^{R_2(\eta)} K\big(R_2(\theta)e^{i\theta},\rho e^{i\eta}\big) \rho\, \dd \rho \dd \eta
      \end{pmatrix}},
\end{split}
\end{equation}
the  equations \eqref{eq:vortex-patch-eqs} become
\begin{align}\label{eq:vortex-patch}
  F(\Omega,\mathbf{r})\triangleq \Omega\, \mathbf{r}'(\theta)+\partial_{\theta}F_0[\mathbf{r}](\theta)=0.
\end{align}
\vskip1mm
\noindent
Additionally, notice that Rankine vortices $\mathbf{1}_{\DD\setminus \overline{b\DD}}(x)$
are stationary solutions of the equation \eqref{eq:vortex-patch},
\begin{align}\label{eq:rankine_vorticity_double}
  F(\Omega,\mathbf{0})\equiv 0,\quad \forall \Omega\in \RR,
\end{align}
which follows from the symmetry assumption of $K(\mathbf{x},\mathbf{y})$.

\subsection{Linearization}
We now proceed to linearize the nonlinear functional equation \eqref{eq:vortex-patch}.
This can be accomplished in a straightforward manner,
following the approach used in previous works, see, for instance, \cite{FHMV16,HHR23}.
We have
\begin{align}\label{eq:linear-main}
  \partial_{\mathbf{r}}F(\Omega,\mathbf{r})\mathbf{h}(\theta)
  =\Omega \,\mathbf{h}'(\theta)+\partial_{\theta}
  \Big(V_b[\mathbf{r}](\theta)\, \mathbf{h}(\theta) + \mathcal{L}[\mathbf{r}](\mathbf{h})(\theta)\Big),
\end{align}
where
\begin{equation}\label{def:eq-Vb}
\begin{split}
  V_b[\mathbf{r}](\theta)=
  \begin{pmatrix}
    V_b^1[\mathbf{r}](\theta) & 0\\
    0 & V_b^2[\mathbf{r}](\theta)
  \end{pmatrix},
\end{split}
\end{equation}
with
\begin{equation}\label{def:Vb1-2}
\begin{split}
  V_b^1[\mathbf{r}](\theta)
  =&\frac{1}{R_1(\theta)}\int_0^{2\pi}\int_{R_1(\eta)}^{R_2(\eta)}
  \Big(\nabla_\mathbf{x}K(R_1(\theta)e^{i\theta},\rho e^{i\eta})\cdot e^{i\theta}\Big)\rho \, \dd \rho \dd \eta,\\
  V_b^2[\mathbf{r}](\theta) = &\frac{1}{R_2(\theta)} \int_0^{2\pi}\int_{R_1(\eta)}^{R_2(\eta)}
  \Big(\nabla_\mathbf{x}K(R_2(\theta)e^{i\theta},\rho e^{i\eta})\cdot e^{i\theta}\Big)\rho \, \dd \rho \dd \eta,
\end{split}
\end{equation}
and
\begin{equation}\label{def:eq-L[r]}
\begin{split}
  \mathcal{L}[\mathbf{r}](\mathbf{h}) (\theta) = \int_{\mathbb{T}}
  \begin{pmatrix}
    -K(R_1(\theta)e^{i\theta},R_1(\eta)e^{i\eta}) & K(R_1(\theta)e^{i\theta},R_2(\eta)e^{i\eta}) \\
    -K(R_2(\theta)e^{i\theta},R_1(\eta)e^{i\eta}) & K(R_2(\theta)e^{i\theta}, R_2(\eta)e^{i\eta})
  \end{pmatrix}
  \begin{pmatrix}
	h_1(\eta)\\
	h_2(\eta)
  \end{pmatrix}\dd \eta.
\end{split}
\end{equation}
	
Since
\begin{align*}
  K(\mathbf{x},\mathbf{y}) = K_0(|\mathbf{x}-\mathbf{y}|)+K_1(\mathbf{x},\mathbf{y}),
\end{align*}
by applying the Gauss-Green theorem, we infer that
\begin{align}\label{def:eq-Vb1}
  V_b^1[\mathbf{r}](\theta)=&-\frac{1}{R_1(\theta)}
  \iint_D \Big(\nabla_\mathbf{y} (K_0(|R_1(\theta)e^{i\theta}-\mathbf{y}|))\cdot e^{i\theta}\Big)\,\dd \mathbf{y}
  \nonumber \\
  &+\frac{1}{R_1(\theta)}\int_0^{2\pi}\int_{R_1(\eta)}^{R_2(\eta)}
  \Big(\nabla_\mathbf{x}K_1(R_1(\theta)e^{i\theta},\rho e^{i\eta})\cdot e^{i\theta}\Big)\rho \, \dd \rho \dd \eta
  \nonumber \\
  = & -\frac{1}{R_1(\theta)}\int_{\TT} K_0(|R_1(\theta)e^{i\theta} - R_2(\eta)e^{i\eta}|)
  \big(-i\partial_{\eta}(R_2(\eta)e^{i\eta})\big)\cdot e^{i\theta} \dd \eta \nonumber \\
  &-\frac{1}{R_1(\theta)} \int_{\TT} K_0(|R_1(\theta)e^{i\theta} - R_1(\eta)e^{i\eta}|)
  \big(i\partial_{\eta}(R_1(\eta)e^{i\eta})\big) \cdot e^{i\theta} \dd \eta \nonumber \\
  &+\frac{1}{R_1(\theta)}\int_0^{2\pi}\int_{R_1(\eta)}^{R_2(\eta)}
  \Big(\nabla_\mathbf{x}K_1(R_1(\theta)e^{i\theta},\rho e^{i\eta})\cdot e^{i\theta}\Big)\rho \, \dd \rho \dd \eta,
\end{align}
and similarly,
\begin{align}\label{def:eq-Vb2}
  V_b^2[\mathbf{r}](\theta) = & -\frac{1}{R_2(\theta)}
  \int_{\TT} K_0(|R_2(\theta)e^{i\theta}-R_2(\eta)e^{i\eta}|)
  \big(-i\partial_{\eta}(R_2(\eta)e^{i\eta})\big)\cdot e^{i\theta} \dd \eta \nonumber \\
  & - \frac{1}{R_2(\theta)}\int_{\TT} K_0(|R_2(\theta)e^{i\theta} - R_1(\eta)e^{i\eta}|)
  \big(i\partial_{\eta}(R_1(\eta)e^{i\eta})\big) \cdot e^{i\theta} \dd \eta \nonumber \\
  & + \frac{1}{R_2(\theta)} \int_0^{2\pi}\int_{R_1(\eta)}^{R_2(\eta)}
  \Big(\nabla_\mathbf{x}K_1(R_2(\theta)e^{i\theta},\rho e^{i\eta})\cdot e^{i\theta}\Big)\rho \, \dd \rho \dd \eta.
\end{align}

Now, we collect some useful facts in the polar coordinates.
Denote by
\begin{align}\label{def:G}
  G_1(\rho_1,\theta,\rho_2,\eta) \triangleq K_1(\rho_1e^{i\theta},\rho_2e^{i\eta}),
  \quad G(\rho_1,\theta,\rho_2,\eta) \triangleq K(\rho_1e^{i\theta},\rho_2e^{i\eta}),
\end{align}
then thanks to assumption \eqref{A4} we have
\begin{equation}\label{eq:G1-prop1}
\begin{split}
  G_1(\rho_1,-\theta,\rho_2,-\eta) & = G_1(\rho_1,\theta,\rho_2,\eta),\quad \\
  G_1(\rho_1,\theta+\theta',\rho_2,\eta+\theta') & = G_1(\rho_1,\theta,\rho_2,\eta),\quad
  \forall \theta'\in \mathbb{R}.
\end{split}
\end{equation}
Hence, we get in particular
$G_1(\rho_1,0,\rho_2,-\eta)=G_1(\rho_1,0,\rho_2,\eta)$,
and by differentiating in $\theta^\prime$ at $\theta^\prime=0$, the second identity  in \eqref{eq:G1-prop1} yields
\begin{equation}\label{eq:G-deriv}
\begin{split}
  \partial_\theta G_1(\rho_1,\theta,\rho_2,\eta) =-\partial_\eta G_1(\rho_1,\theta,\rho_2,\eta), \quad
  \quad \partial_\theta G(\rho_1,\theta,\rho_2,\eta) =-\partial_\eta G(\rho_1,\theta,\rho_2,\eta).
\end{split}
\end{equation}
By letting $\mathbf{x}= \rho_1 e^{i\theta}$ and $\mathbf{y} = \rho_2 e^{i\eta}$,
we get from straightforward computations:
\begin{align}\label{eq:K-der}
  \nabla_{\mathbf{x}} K(\mathbf{x},\mathbf{y})=\begin{pmatrix}
  \partial_{\rho_1}G\cos \theta - \partial_{\theta}G \frac{\sin \theta}{\rho_1}\\
  \partial_{\rho_1}G\sin \theta + \partial_{\theta}G  \frac{\cos \theta}{\rho_1}
\end{pmatrix}, \quad
  \nabla_{\mathbf{y}} K (\mathbf{x},\mathbf{y})=\begin{pmatrix}
  \partial_{\rho_2}G \cos \eta - \partial_{\eta}G \frac{\sin \eta}{\rho_2}\\
  \partial_{\rho_2}G\sin \eta+\partial_{\eta}G  \frac{\cos \eta}{\rho_2}
\end{pmatrix}.
\end{align}
Note that the  above identities  hold true when  $(K,G)$ is replaced by $(K_1,G_1)$.
\vskip1mm

\noindent Next, we consider the linearization around the trivial solution $\mathbf{r}(\theta) \equiv \mathbf{0}$.
For $V_b^1[\mathbf{0}]$ and $V_b^2[\mathbf{0}]$, the formula \eqref{def:Vb1-2} gives
\begin{equation}\label{eq:Vb1-2[0]}
\begin{split}
  V_b^1[\mathbf{0}](\theta)
  & =\frac{1}{b} \int_0^{2\pi}\int_b^1
  \Big(\nabla_\mathbf{x}K(b e^{i\theta},\rho e^{i\eta})\cdot e^{i\theta}\Big)\rho \, \dd \rho \dd \eta,\\
  V_b^2[\mathbf{0}](\theta) & =  \int_0^{2\pi}\int_b^1
  \Big(\nabla_\mathbf{x}K( e^{i\theta},\rho e^{i\eta})\cdot e^{i\theta}\Big)\rho \, \dd \rho \dd \eta,
\end{split}
\end{equation}
and by virtue of \eqref{eq:G1-prop1}, we also have
\begin{align*}
  V_b^1[\mathbf{0}](\theta)=&-\frac{1}{b}\int_{\TT}K_0(|be^{i\theta}-e^{i\eta}|)
  (-i\partial_{\eta}(e^{i\eta}))\cdot e^{i\theta}\dd \eta\\
  & -\frac{1}{b}\int_{\TT}K_0(b|e^{i\theta}-e^{i\eta}|)(i\partial_{\eta}(be^{i\eta}))\cdot e^{i\theta} \,\dd \eta \\
  &+\frac{1}{b}\int_0^{2\pi}\int_{b}^{1}
  \Big(\nabla_{\mathbf{x}}K_1(be^{i\theta},\rho e^{i\eta})\cdot e^{i\theta}\Big)\rho \, \dd \rho \dd \eta \\
  = & -\frac{1}{b}\int_{\TT}K_0(|b-e^{i\eta}|)\cos \eta\, \dd \eta + \int_{\TT}K_0(b|1-e^{i\eta}|) \cos \eta \,\dd \eta \\
  & +\frac{1}{b}\int_0^{2\pi}\int_{b}^{1}
  \partial_{\rho_1}G_1(b,0,\rho,\eta)\rho \, \dd \rho \dd \eta,
\end{align*}
and similarly,
\begin{align*}
  V_b^2[\mathbf{0}](\theta)=&-\int_{\TT}K_0(|1-e^{i\eta}|)\cos \eta \,\dd \eta
  + b\int_{\TT}K_0(|1-be^{i\eta}|) \cos \eta \,\dd \eta\\
  & + \int_0^{2\pi}\int_{b}^{1}
  \partial_{\rho_1}G_1(1,0,\rho,\eta)\rho \, \dd \rho \dd \eta.
\end{align*}
Thus, we find that $V_b[\mathbf{0}]$ is indeed a constant matrix independent of $\theta$.
For $\mathcal{L}[\mathbf{0}](\mathbf{h})$, we get
\begin{align*}
  \mathcal{L}[\mathbf{0}](\mathbf{h}) (\theta) = &\int_{\mathbb{T}}
  \begin{pmatrix}
    - K(be^{i\theta},be^{i\eta}) & K(be^{i\theta},e^{i\eta}) \\
    - K(e^{i\theta},be^{i\eta}) & K(e^{i\theta}, e^{i\eta})
  \end{pmatrix}
  \begin{pmatrix}
    h_1(\eta)\\
    h_2(\eta)
  \end{pmatrix}\dd \eta\\
  = & \int_{\mathbb{T}}
  \begin{pmatrix}
    -K(b,be^{i\eta}) & K(b,e^{i\eta}) \\
    -K(1,be^{i\eta}) & K(1, e^{i\eta})
  \end{pmatrix}
  \begin{pmatrix}
    h_1(\eta+\theta)\\
    h_2(\eta+\theta)
  \end{pmatrix}\dd \eta.
\end{align*}
For every smooth function
$\mathbf{h}(\theta)=\sum_{n\in \ZZ}
\begin{pmatrix}
  h_{1,n}\\
  h_{2,n}
\end{pmatrix} e^{in\theta}$,
we have
\begin{align*}
  \mathcal{L}[\mathbf{0}](\mathbf{h})=\sum_{n\in \ZZ}\mathbf{\Lambda}_{n,b}
  \begin{pmatrix}
	h_{1,n}\\
	h_{2,n}
  \end{pmatrix} e^{in\theta},
\end{align*}
where 
\begin{align}\label{def:Lambda-nb}
  \mathbf{\Lambda}_{n,b} \triangleq \int_{\mathbb{T}}
  \begin{pmatrix}
	-K(b,be^{i\eta}) & K(b,e^{i\eta}) \\
	-K(1,be^{i\eta}) & K(1, e^{i\eta})
  \end{pmatrix} e^{in\eta}\dd \eta.
\end{align}
Notice that $\mathbf{\Lambda}_{n,b}=\mathbf{\Lambda}_{-n,b}$ and the operator $\mathcal{L}[\mathbf{0}]$ is a Fourier multiplier operator
with the discrete spectrum set $\mathrm{sp}(\mathcal{L}[\mathbf{0}]) = \{\mathbf{\Lambda}_{n,b}\,,\, n\in \mathbb{N} \}$.
Direct computations also give that
\begin{align*}
  \partial_{\theta}\Big(V_b[\mathbf{r}](\theta)\mathbf{h}(\theta)\Big)\Big|_{\mathbf{r}=\mathbf{0}}
  =\partial_{\theta}\Big(V_b[\mathbf{0}](\theta)\mathbf{h}(\theta)\Big)=V_b[\mathbf{0}]\mathbf{h}'(\theta),
\end{align*}
and
\begin{align*}
  \partial_{\theta}\Big(\mathcal{L}[\mathbf{r}](\mathbf{h})(\theta)\Big)\Big|_{\mathbf{r} = \mathbf{0}}
  =\partial_{\theta}\Big(\mathcal{L}[\mathbf{0}](\mathbf{h})(\theta)\Big)=\mathcal{L}[\mathbf{0}](\mathbf{h}')(\theta).
\end{align*}
Hence, we find
\begin{align}\label{eq:linear-0}
  \partial_{\mathbf{r}}F(\Omega,\mathbf{0})\mathbf{h}(\theta)
  =\sum_{n\in \mathbb{Z}\setminus \{0\}}i\,n\Big(\Omega\, \mathbf{Id} + V_b[\mathbf{0}] + \mathbf{\Lambda}_{n,b}\Big)
  \begin{pmatrix}
	h_{1,n}\\
	h_{2,n}
  \end{pmatrix} e^{in\theta},
\end{align}
where $\mathbf{Id}$ is the unitary matrix of $\mathbb{R}^{2\times 2}$.

\subsection{Spectrum}
In the following, our aim is to describe the structure of the linearized operator in the annular patch. This operator takes the form of a Fourier matrix multiplier. The key objective is to derive the dispersion relation using a factorization formula.
From the decomposition
$$K(\mathbf{x},\mathbf{y}) = K_0(|\mathbf{x}-\mathbf{y}|) + K_1(\mathbf{x},\mathbf{y}),$$
we rewrite $\mathbf{\Lambda}_{n,b}$, given by \eqref{def:Lambda-nb}, as
\begin{align}\label{eq:Lambda-n-b}
  \mathbf{\Lambda}_{n,b}=
  \begin{pmatrix}
	-\lambda_{n,b} & \widetilde{\lambda}_{n,b} \\
	-\widetilde{\lambda}_{n,b} & \lambda_{n,1}
  \end{pmatrix}
  + \mathbf{P}_{n,b},
  \quad
  \textrm{with}
  \quad
  \mathbf{P}_{n,b} \triangleq \begin{pmatrix}
	-p_{n,b} & \widetilde{p}_{n,b} \\
	-\widetilde{p}_{n,b} & p_{n,1}
  \end{pmatrix}
\end{align}
and (we may assume $0<b\leqslant 1$ below to give some unified definitions)
\begin{equation}\label{def:lambda-nb}
\begin{aligned}
  \lambda_{n,b} \triangleq \int_{\mathbb{T}} K_0(2b|\sin \tfrac{\eta}{2}|) e^{in\eta}\dd \eta,\quad
  \widetilde{\lambda}_{n,b}  \triangleq \int_{\mathbb{T}}
  K_0(|b-e^{i\eta}|) e^{in\eta}\dd \eta,
\end{aligned}
\end{equation}
with
\begin{align}\label{def:Pnb}
  p_{n,b} \triangleq \int_{\mathbb{T}}
  K_1(b,be^{i\eta}) e^{in\eta}\dd \eta,\quad
  \widetilde{p}_{n,b} \triangleq \int_{\mathbb{T}} K_1(b,e^{i\eta}) e^{in\eta}\dd \eta.
\end{align}
In the above, we also have used the facts that
\begin{align}\label{eq:fact-a1}
  |b-e^{i\eta}|=|1-be^{-i\eta}|=|1-be^{i\eta}|,
\end{align}
and
\begin{align*}
  K_1(b,e^{i\eta}) = K_1(b e^{-i\eta},1) = K_1(be^{i\eta},1)=K_1(1,be^{i\eta}).
\end{align*}
Denoting by
\begin{align}\label{c-1}
  \mathsf{c}_b\triangleq\frac{1}{b}\int_0^{2\pi}\int_{b}^{1}
  \partial_{\rho_1}G_1(b,0,\rho,\eta)\rho \, \dd \rho \dd \eta
  = \frac{1}{b} \int_0^{2\pi} \int_b^1 \nabla_{\mathbf{x}}K_1(b,\rho e^{i\eta})\cdot (1,0)^T \,\rho\dd \rho \dd \eta,
\end{align}
and
\begin{align}\label{c-2}
  \widetilde{\mathsf{c}}_b \triangleq\int_0^{2\pi}\int_{b}^{1}
  \partial_{\rho_1}G_1(1,0,\rho,\eta)\rho \, \dd \rho \dd \eta
  =  \int_0^{2\pi} \int_b^1 \nabla_{\mathbf{x}}K_1(1,\rho e^{i\eta}) \cdot (1,0)^T \,\rho\dd \rho \dd \eta,
\end{align}
we also recast $V_b[\mathbf{0}]$ as
\begin{align}\label{eq:V-0}
  V_b[\mathbf{0}]=
    \begin{pmatrix}
    V_b^1[\mathbf{0}] & 0\\
    0 & V_b^2[\mathbf{0}]
  \end{pmatrix}
  =
  \begin{pmatrix}
	\lambda_{1,b} -\frac{1}{b} \widetilde{\lambda}_{1,b} + \mathsf{c}_b &	0\\
	0	&  - \lambda_{1,1} + b \widetilde{\lambda}_{1,b} +  \widetilde{\mathsf{c}}_b
  \end{pmatrix}.
\end{align}
Next under the assumption \eqref{A1}, we factorize the coefficients $\lambda_{n,b}$ and $\widetilde{\lambda}_{n,b}$
by using some universal functions.
In view of \eqref{eq:K0prim},
we deduce by Fubini's theorem that for any $C>0$,
\begin{align*}
  K_0(t) & = K_0(C) - \int_C^t \int_0^\infty e^{-\tau x} \dd \mu(x)
  \dd \tau  \\
  & = K_{0}(C)+\int_0^{\infty}\frac{e^{-tx}-e^{-Cx}}{x}\dd \mu(x).
\end{align*}
Thanks to Fubini's theorem again, we have that for every $n\in \NN^\star$ and $0<b\leqslant 1$,
\begin{align}\label{def:lam-n-1}
  \lambda_{n,b} = \int_0^{2\pi}K_0(2b\sin \tfrac{\eta}{2})e^{in\eta}\dd \eta
  = & \int_0^{2\pi}\int_0^{\infty}\frac{e^{-2bx\sin \frac{\eta}{2}}-e^{-C x}}{x} e^{in\eta} \,\dd \mu(x)\dd \eta \nonumber \\
  = & \int_0^{\infty}\int_0^{2\pi}e^{-2bx\sin \frac{\eta}{2}} e^{in\eta}\dd \eta \frac{\dd \mu(x)}{x} \nonumber \\
  = & \int_0^{\infty}\phi_{n}(bx)\frac{\dd \mu(x)}{x},
\end{align}
where\footnote{Indeed, the universal function $\phi_n(x)$ defined by \eqref{eq:phi_n_ori} is two times of the same function in \cite{HXX23b}.}
\begin{align}\label{eq:phi_n_ori}
  \phi_{n}(x) \triangleq \int_0^{2\pi}e^{-2x \sin \frac{\eta}{2}} e^{in\eta}\dd \eta
  = 2\int_0^{\pi}e^{-2x \sin \eta} e^{i2n\eta}\dd \eta.
\end{align}
In addition, in light of the relation \eqref{eq:fact-a1}, we find
\begin{align}\label{def:lam-n-3}
  \widetilde{\lambda}_{n,b}  & = \int_0^{2\pi}K_0(|1-be^{i\eta}|)e^{in\eta}\dd \eta
  = \int_0^{2\pi}K_0(|b-e^{i\eta}|)e^{in\eta}\dd \eta \nonumber \\
  & = \int_0^{2\pi} \int_0^{\infty} \frac{e^{-|1-be^{i\eta}|x} - e^{-C x}}{x} e^{in\eta}\, \dd \mu(x) \dd \eta   \nonumber \\
  & = \int_0^\infty \int_0^{2\pi}e^{-|1-be^{i\eta}|x}e^{in\eta}\dd \eta \frac{\dd \mu(x)}{x} \nonumber \\
  & = \int_0^{\infty}\phi_{n,b}(x) \frac{\dd \mu(x)}{x},
\end{align}
where
\begin{align}\label{def:phi-n-b}
  \phi_{n,b}(x) \triangleq \int_0^{2\pi} e^{-|1-be^{i\eta}|x } e^{in\eta} \dd \eta
  = \int_0^{2\pi} e^{-x \sqrt{1+b^2 - 2b\cos \eta}} e^{i n\eta} \dd \eta.
\end{align}


In order to apply the local bifurcation tools, i.e. Crandall-Rabinowitz's theorem, we first need to analyze the lack of invertibility of
the linear operator $\partial_{\mathbf{r}}F(\Omega,\mathbf{0})$. Namely,  those values of $\Omega$ such that
\begin{align}\label{eq:Omega}
  \det\big(\Omega\, \mathbf{Id} + V_b[\mathbf{0}]+\mathbf{\Lambda}_{n,b}\big) = 0,
\end{align}
equivalently,
\begin{align*}
  \mathrm{det}
  \begin{pmatrix}
    \Omega + V_b^1[\mathbf{0}] - \lambda_{n,b} - p_{n,b}  & \widetilde{\lambda}_{n,b} + \widetilde{p}_{n,b} \\
    -\widetilde{\lambda}_{n,b} - \widetilde{p}_{n,b} & \Omega + V_b^2[\mathbf{0}]
    + \lambda_{n,1} + p_{n,1}
  \end{pmatrix}
  = 0.
\end{align*}
Denote by
\begin{align}\label{eq:def-A-n}
  A_{n,b} \triangleq - V_b^1[\mathbf{0}] + \lambda_{n,b} + p_{n,b}
  = - \lambda_{1,b} + \tfrac{1}{b} \widetilde{\lambda}_{1,b}  - \mathsf{c}_b
  + \lambda_{n,b} + p_{n,b}
\end{align}
and
\begin{align}\label{eq:def-B-n}
  B_{n,b} \triangleq - V_b^2[\mathbf{0}] -\lambda_{n,1} - p_{n,1}
  = \lambda_{1,1} - b \widetilde{\lambda}_{1,b} - \widetilde{\mathsf{c}}_b- \lambda_{n,1} - p_{n,1} .
\end{align}
Then we can rewrite \eqref{eq:Omega} as
\begin{align}\label{eq:Omega-eq}
  \Omega^2 - (A_{n,b} + B_{n,b})\, \Omega +  A_{n,b} B_{n,b} +  (\widetilde{\lambda}_{n,b} + \widetilde{p}_{n,b})^2 = 0.
\end{align}
The discriminant of the quadratic polynomial \eqref{eq:Omega-eq} is as follows
\begin{equation}\label{discriminant}
\begin{split}
 \mathbf{\Delta}_{n,b} &\triangleq  (A_{n,b} + B_{n,b})^2
 - 4 \big(A_{n,b} B_{n,b} +  (\widetilde{\lambda}_{n,b} + \widetilde{p}_{n,b})^2 \big) \\
 &=(A_{n,b}-B_{n,b})^2 - 4(\widetilde{\lambda}_{n,b} + \widetilde{p}_{n,b})^2.
\end{split}
\end{equation}
It will be used later in the discussion on stable and unstable modes.

\section{Analysis of the universal functions}\label{sec:universal-func}
In this section, we collect several properties of the universal functions $\phi_{n,b}(x)$ and $\phi_n(b x)$
that appear in \eqref{def:lam-n-1} and \eqref{def:lam-n-3},
which play an important role in the spectral study below.

\subsection{Positivity and Monotonicity of $\phi_{n,b}(x)$}
For the universal function $\phi_n(x)=\phi_{n,1}(x)$ given by \eqref{eq:phi_n_ori},
according to \cite[Eq. (2.15)]{HXX23b} and \cite[Eq. (3.3)]{HXX23b},
it satisfies an ordinary differential equation (ODE)
\begin{align}\label{eq:phi_n_ode}
  \phi''_{n}(x) + x^{-1}\phi_n'(x) - 4 \Big(1+\frac{n^2}{x^2}\Big) \phi_n(x) = -\frac{8}{x},\quad x>0,
\end{align}
with the boundary value
$\phi_n(0)=0$ and $\lim\limits_{x\to \infty}\phi_n(x)=0$.
Based on the study of equation \eqref{eq:phi_n_ode},
the authors \cite{HXX23b} prove the positivity and monotonicity of the map $n\mapsto\phi_n(x)$,
and moreover show the following quantitative estimates (see \cite[Lem. 3.2]{HXX23b} and \cite[Prop. 3.3]{HXX23b})
\begin{align}\label{ineq:est-phi-n}
  \frac{8n^2}{4n^2+1}\frac{x}{n^2+x^2}\leqslant \phi_n(x) \leqslant \frac{8n^2}{4n^2-1}\frac{x}{n^2+x^2},
\end{align}
and
\begin{align}\label{ineq:est-phi-n-diff}
  \frac{(2n+1)x}{(n^2+x^2)((n+1)^2+x^2)} \leqslant \phi_{n}(x)-\phi_{n+1}(x)
  \leqslant \frac{8(2n+1)x}{(n^2+x^2)((n+1)^2+x^2)}\cdot
\end{align}
However,  the deformation in $b$ is so strong, and we do not succeed to find an  ordinary differential equation for $\phi_{n,b}(x)$ with $b\in (0,1)$.
Here, using Chebyshev's polynomials, we find a new approach to show the positivity and monotonicity of the function
$n\mapsto\phi_n(x)$, which also applies to $n\mapsto \phi_{n,b}(x)$ for all $b\in (0,1)$.
\\[1mm]
\noindent
Before proceeding forward, we recall the definition and some properties of Chebyshev's polynomials (see e.g. \cite[Sec. 8.94]{GR15}).
The Chebyshev polynomials of the first kind $T_n$, is defined by
\begin{align}\label{eq:chebyshev-1}
  \forall x\in [-1,1],\quad T_n( x) \triangleq \cos (n\arccos x), \quad \forall n\in \mathbb{N},
\end{align}
and the Chebyshev polynomials of the second kind $U_n(x)$, $x\in [-1,1]$ is given by
\begin{align}\label{eq:chebyshev-2}
  U_n( x)  \triangleq \frac{\sin [(n+1)\arccos x]}{\sin (\arccos x)}, \quad \forall n\in \NN.
\end{align}
The Rodrigues formulas for Chebyshev polynomials of the first and second kinds respectively read that
\begin{align}\label{eq:T-n-formula}
  T_n(x)=(-1)^n \frac{\sqrt{\pi}}{2^n \Gamma(n+ \frac{1}{2})}
  (1-x^2)^{\frac{1}{2}} \frac{\dd^n }{\dd x^n}\big[(1-x^2)^{n-\frac{1}{2}}\big]
\end{align}
and
\begin{align}\label{eq:U-n-formula}
  U_n(x)=(-1)^n \frac{\sqrt{\pi} (n+1)}{2^{n+1} \Gamma(n+ \frac{3}{2})} (1-x^2)^{-\frac{1}{2}}
  \frac{\dd^n}{\dd x^n}\big[(1-x^2)^{n+\frac{1}{2}}\big].
\end{align}



First, we give a direct proof for the positivity of $\phi_{n,b}(x)$ with any $0<b\leqslant 1$.
\begin{lemma}\label{lem:phi_n_b_1}
Let $n\in \mathbb{N}^\star$ and $0<b \leqslant 1$.
Then, the function $\phi_{n,b}(x)$ is strictly  positive for any $x\in (0,+\infty)$ and satisfies
\begin{align}\label{eq:bound_uniform}
\forall x>0,\quad   0 < \phi_{n,b}(x) \leqslant 2\pi e^{-(1-b)x}.
\end{align}
\end{lemma}


\begin{proof}[Proof of Lemma \ref{lem:phi_n_b_1}]
By changing of variables and using \eqref{eq:chebyshev-1}, we write
\begin{align*}
  \phi_{n,b}(x) & =\int_0^{2\pi}e^{-|1-be^{i\eta}|x}e^{in\eta}\dd \eta=2\int_0^{\pi}e^{-|1-be^{i\eta}|x}\cos (n\eta)\dd \eta\\
  & = 2\int_0^{\pi}e^{- x \sqrt{1+ b^2 -2b\cos \eta}}\,T_n(\cos \eta)\dd \eta\\
  & = 2\int_{-1}^{1}e^{-x\sqrt{1+b^2-2by}}T_n(y)(1-y^2)^{-\frac{1}{2}}\dd y.
\end{align*}
Next, we apply the Rodrigues formula \eqref{eq:T-n-formula} to the above identity and we find that
\begin{align*}
  \phi_{n,b}(x)=\frac{\sqrt{\pi}}{2^{n-1} \Gamma(n+ \frac{1}{2})}
  \int_{-1}^{1}e^{-x\sqrt{1+b^2-2by}}(-1)^n\frac{\dd^n}{\dd y^n}[(1-y^2)^{n-\frac{1}{2}}]\dd y.
\end{align*}
Noting that for any $k=0,1,\cdots,n-1$,
\begin{align*}
  \Big(\frac{\dd^k}{\dd y^k} \big[(1-y^2)^{n-\frac{1}{2}}\big]\Big)\Big|_{y=\pm 1}=0,
\end{align*}
through integration by parts, we conclude that
\begin{align}\label{exp:phi-nb}
  \phi_{n,b}(x)=\frac{\sqrt{\pi}}{2^{n-1} \Gamma(n+ \frac{1}{2})} \int_{-1}^{1}\frac{\dd^n}{\dd y^n}
  \Big(e^{-x\sqrt{1+b^2-2by}}\Big)(1-y^2)^{n-\frac{1}{2}}\dd y.
\end{align}
We claim that for each $x>0$ and $b\in(0,1]$ the function
\begin{align}\label{Comp-Mom}
  y\in (-1,1)\mapsto e^{-x\sqrt{1+b^2-2by}}\triangleq e^{W_b(y)},\quad \textrm{with}\quad
  W_b(y) \triangleq -x\sqrt{1+b^2-2by},
\end{align}
is absolutely monotonic, that is, for any $n\in \NN$,
\begin{align}\label{eq:positive-derivative-2}
  \frac{\dd^n}{\dd y^n}\Big(e^{-x\sqrt{1+b^2-2by}}\Big)>0,\quad \forall y\in (-1,1).
\end{align}
This ensures that $\phi_{n,b}$ is strictly positive in $(0,\infty)$.
Let us now move to the proof of \eqref{eq:positive-derivative-2}.
First, we observe that $W_b^\prime$ is absolutely monotonic, since
\begin{align*}
  W_b^\prime(y)=bx \big(1+b^2-2by\big)^{-\frac12}>0;
\end{align*}
and after successive  differentiation  we see that
\begin{align*}
  \forall n\in \NN,\quad \forall y\in(-1,1),\quad W_b^{(1+n)}(y)>0.
\end{align*}
Now, since $t\in\RR\mapsto e^t$ is absolutely monotonic, by composition we find that
$e^{W_b}$ is absolutely monotonic. This last property follows from Fa\`a di Bruno's formula:
\begin{align*}
  \frac{\dd^n}{\dd x^n} e^{f(x)} = e^{f(x)} \cdot \sum_{k=0}^n \mathbf{B}_{n,k} \big(f'(x), f''(x), \dots, f^{(n-k+1)}(x)\big),
\end{align*}
where $\mathbf{B}_{n,k}(x_1,x_2,\cdots,x_{n-k+1})$ are the Bell polynomials,
and their coefficients are all positive numbers.
\\[1mm]
\noindent As to the estimate of $\phi_{n,b}$, we note that
\begin{align}\label{eq:basic_n_b}
  |1-b e^{i\eta}| = \sqrt{1+b^2-2b\cos \eta} = \sqrt{(1-b)^2+4b\sin^2 \tfrac{\eta}{2}}.
\end{align}
Therefore, we deduce that for any $x>0$ and $b\in (0,1)$,
\begin{align*}
  |\phi_{n,b}(x)|\leqslant \int_0^{2\pi}e^{-x \sqrt{1+b^2 - 2b\cos \eta}}\dd \eta
  \leqslant  2 \pi e^{- (1-b)x}.
\end{align*}
\end{proof}

The next result is about the strict monotonicity of $n\mapsto\phi_{n,b}(x)$.
\begin{lemma}\label{lem:phi-n-b-mono}
Let $n\in \mathbb{N}^\star$, $x\in (0,+\infty)$, $b\in (0,1]$. Then the map $n\mapsto \phi_{n,b}(x)$ is strictly decreasing.
\end{lemma}

\begin{remark}
By Riemann-Lebesegue theorem, applied with \eqref{def:phi-n-b}, we get
$\lim\limits_{n\to +\infty} \phi_{n,b}(x)=0$. Thus, the strict decreasing property of
$n\mapsto \phi_{n,b}(x)$  implies the strict positivity of $\phi_{n,b}(x)$  as shown before in Lemma \ref{lem:phi_n_b_1}.
\end{remark}

\begin{proof}[Proof of Lemma \ref{lem:phi-n-b-mono}]
From \eqref{def:phi-n-b} and $\phi_{n,b}\in \RR$, we rewrite the difference $\phi_{n+1,b}-\phi_{n,b}$ as
\begin{align*}
  \phi_{n+1,b}(x)-\phi_{n,b}(x)=&\int_0^{2\pi} e^{-x \sqrt{1+b^2 - 2b\cos \eta}} (e^{i (n+1)\eta}-e^{in\eta}) \dd \eta\\
  = & \int_0^{2\pi} e^{-x \sqrt{1+b^2 - 2b\cos \eta}} (e^{i \frac{1}{2}\eta}-e^{-i\frac{1}{2}\eta})
  e^{i(n+\frac{1}{2})\eta}\dd \eta\\
  = & - 2\int_0^{2\pi} e^{-x \sqrt{1+b^2 - 2b\cos \eta}}\sin \tfrac{\eta}{2} \sin \big((n+\tfrac{1}{2})\eta\big) \, \dd \eta.
\end{align*}
According to the definition of Chebyshev polynomials of the second kind $U_m$ in \eqref{eq:chebyshev-2}, it holds that
\begin{align*}
  U_m(\cos \theta)\sin \theta=\sin (m+1)\theta,\,\quad \forall m\in \NN^\star,\,\theta\in [0,\pi].
\end{align*}
Hence, we have
\begin{align*}
  \phi_{n+1,b}(x)-\phi_{n,b}(x) & =-2\int_0^{2\pi} e^{-x \sqrt{1+b^2 - 2b\cos \eta}}
  (\sin \tfrac{\eta}{2})^2 U_{2n}(\cos \tfrac{\eta}{2})  \dd \eta \\
  & = - 2\int_0^{2\pi} e^{-x \sqrt{(1+b)^2-4b\cos^2 \frac{\eta}{2}}}(1-\cos^2 \tfrac{\eta}{2})U_{2n}(\cos \tfrac{\eta}{2})  \dd \eta\\
  & = -4\int_0^{\pi} e^{-x \sqrt{(1+b)^2-4b\cos^2 \eta}}(1-\cos^2 \eta)U_{2n}(\cos \eta)  \dd \eta.
\end{align*}
By changing of the variable $y=\cos \eta$, we find
\begin{align*}
  \phi_{n+1,b}(x)-\phi_{n,b}(x)=-4\int_{-1}^{1} e^{-x \sqrt{(1+b)^2-4by^2}}U_{2n}(y) \sqrt{1-y^2} \dd y.
\end{align*}
Using the Rodrigues formula \eqref{eq:U-n-formula} gives that
\begin{align*}
  \phi_{n+1,b}(x)-\phi_{n,b}(x)=-\frac{\sqrt{\pi} (n+1)}{2^{n-1} \Gamma(n+ \frac{3}{2})}
  \int_{-1}^{1} e^{-x \sqrt{(1+b)^2-4by^2}}
  \frac{\dd^{2n}}{\dd y^{2n}} \big[(1-y^2)^{2n+\frac{1}{2}}\big] \dd y.
\end{align*}
Noting that
\begin{align*}
  \Big(\frac{\dd^{k}}{\dd y^{k}}\big[(1-y^2)^{2n+\frac{1}{2}}\big]\Big)\Big|_{y=\pm 1}=0,\quad \forall k=0,1,\cdots, 2n,
\end{align*}
we integrate by parts to deduce that
\begin{align}\label{eq:diff-phi-n-b-2}
  \phi_{n+1,b}(x)-\phi_{n,b}(x) = - \frac{\sqrt{\pi} (n+1)}{2^{n-1} \Gamma(n+ \frac{3}{2})}
  \int_{-1}^{1}\frac{\dd^{2n}}{\dd y^{2n}}
  \Big(e^{-x \sqrt{(1+b)^2-4by^2}}\Big)(1-y^2)^{2n+\frac{1}{2}}\dd y.
\end{align}
We will show that
\begin{align}\label{eq:ind-2n-der}
\frac{\dd^{2n}}{\dd y^{2n}}
  \Big(e^{-x \sqrt{(1+b)^2-4by^2}}\Big)>0.
\end{align}
First we write in view of the notation \eqref{Comp-Mom},
\begin{align*}
  e^{-x \sqrt{(1+b)^2-4by^2}}=e^{\widetilde{W}_{b}(y^2)} \triangleq H(y^2),\quad \textrm{with}\quad
  \widetilde{W}_b(y) \triangleq -x\sqrt{(1+b)^2-4by}.
\end{align*}
According to the same argument as the proof of \eqref{eq:positive-derivative-2},
the function $y\in [0,1]\mapsto H(y)$ is absolutely monotonic.
Using Schloemilch's identity \cite{Sch},
based on Fa\`a di Bruno's formula, we infer that
\begin{align*}
  \frac{\dd^{2n}}{\dd y^{2n}} H(y^2)=
  \sum_{k=0}^{n}\frac{(2n)!}{(2n-2k)! k!}(2y)^{2n-2k} H^{(2n-k)}(y^2)>0.
\end{align*}
Therefore, the function $y\in(-1,1)\mapsto H(y^2)$ is absolutely monotonic. \\
As a result, \eqref{eq:diff-phi-n-b-2} and \eqref{eq:ind-2n-der}
imply the required monotonicity property, that is,
$\phi_{n+1,b}(x) - \phi_{n,b}(x) < 0$, $\forall n\in\mathbb{N}^\star$.
\end{proof}

\subsection{Analyticity of universal functions and related spectral terms}
In this subsection, we deal with the real analyticity in $b\in(0,1)$ of universal functions $\phi_n(bx)$, $\phi_{n,b}(x)$
and related spectral terms $\lambda_{n,b}$, $\widetilde{\lambda}_{n,b}$.
Recall that a  real function $x\mapsto f(x)$ is real analytic on an open set $D$
is equivalent to that $f$ is smooth and for every compact set $D_0\subset\subset D$,
there exists a constant $C>0$ such that for every $x\in D_0$ and every $k\in\mathbb{N}^\star$ the following bound holds
\begin{align}\label{eq:analytic_lem}
  \Big|\frac{\dd^k f}{\dd x^k}(x)\Big|\leqslant  C^{k+1}k!.
\end{align}
In fact, the inequality \eqref{eq:analytic_lem} implies the convergence of
$\sum\limits_{k=0}^\infty \frac{f^{(k)}(x)}{k!}(y-x)^k$ for any $y\in (x-\tfrac{1}{C},x-\tfrac{1}{C})$.

\begin{lemma}\label{lem:analytic_lambda}
  Let assumptions \eqref{A1}-\eqref{A2} be satisfied with some $\alpha\in (0,1)$.
\begin{enumerate}
\item
For any $n\in\mathbb{N}^\star$ and $x\in (0,+\infty)$, the functions $b\mapsto \phi_{n}(bx)$ and $b\mapsto\lambda_{n,b}$
are real analytic on $(0,1)$.
\item
For any $n\in \mathbb{N}^\star$ and $x\in (0,+\infty)$,
the functions $b\mapsto\phi_{n,b}(x)$ and $b \mapsto \widetilde{\lambda}_{n,b}$ are real analytic on $(0,1)$.
\end{enumerate}
\end{lemma}

\begin{proof}[Proof of Lemma \ref{lem:analytic_lambda}]
\textbf{(i)}
From \eqref{eq:phi_n_ori}, we see that
$\frac{\dd^k }{\dd b^k}(\phi_{n}(bx)) = x^k\phi^{(k)}_{n}(bx)$,
where
\begin{align*}
  \phi_n^{(k)}(x)=\int_0^{2\pi}e^{-2x \sin \tfrac{\eta}{2}}\Big(-2\sin \frac{\eta}{2}\Big)^{k} e^{in\eta}\dd \eta.
\end{align*}
Using the facts that for any $x>0$ and $\sigma > 0$,
\begin{align*}
  e^{-x} x^\sigma = \int_0^x (e^{-y}y^\sigma)'\dd y
  &  = \int_0^x \big(-e^{-y}y^\sigma + \sigma e^{-y} y^{\sigma - 1} \big) \dd y \\
  & \leqslant  \sigma \int_0^{+\infty} e^{-y}y^{\sigma-1}\dd y
  = \sigma \Gamma(\sigma),
\end{align*}
and for every $-1<s<0$,
\begin{align*}
  \int_0^{2\pi} \Big(\sin \frac{\eta}{2}\Big)^s  \dd \eta = 4 \int_0^{\frac{\pi}{2}} \big(\sin \eta\big)^s \dd \eta
  \leqslant 4 \int_0^{\frac{\pi}{2}} \Big(\frac{2}{\pi} \eta \Big)^s  \dd \eta
  \leqslant \frac{2\pi}{s+1},
\end{align*}
we infer that for every $x>0$ and $-1< s<k$,
\begin{align}\label{eq:phi-n-general-bound}
  |x^{k-s}\phi_n^{(k)}(bx)|  & =\Big|\int_0^{2\pi}e^{-2bx \sin \tfrac{\eta}{2}}
  \Big(-2\sin \frac{\eta}{2}\Big)^{k}x^{k-s} e^{in\eta}\dd \eta\Big|\notag\\
  & \leqslant 2^s \int_0^{2\pi} e^{-2bx \sin \tfrac{\eta}{2}} \Big(2\sin \frac{\eta}{2}\Big)^{k-s}x^{k-s} \Big(\sin \frac{\eta}{2}\Big)^s \dd \eta \notag \\
  & \leqslant  (k-s)\Gamma(k-s)2^sb^{-(k-s)} \int_0^{2\pi} \Big(\sin \frac{\eta}{2}\Big)^s
  \dd \eta \notag \\
  & \leqslant (k-s)\Gamma(k-s)\, 2^s b^{-(k-s)} \max\Big\{\frac{2\pi}{s+1},2\pi \big\}.
\end{align}
For the endpoint case $s=-1$, noting that
\begin{align*}
  \int_0^{2\pi}e^{-x\sin \tfrac{\eta}{2}}\dd \eta = 4 \int_0^{\frac{\pi}{2}} e^{-x \sin \eta} \dd \eta
  \leqslant 4\int_0^{\frac{\pi}{2}} e^{-\frac{2}{\pi} x \eta} \dd \eta \leqslant  \frac{2\pi }{x},
\end{align*}
we instead have that for every $x>0$,
\begin{align}\label{eq:phi-n-decay-bound}
    |x^{k+1}\phi_n^{(k)}(bx)| &=\Big|\int_0^{2\pi}e^{-2bx \sin \frac{\eta}{2}}
  \Big(-2\sin \frac{\eta}{2}\Big)^{k}x^{k+1} e^{in\eta}\dd \eta\Big|\notag\\
  &\leqslant 2^kb^{-k}k\Gamma(k)x\int_0^{2\pi}e^{-bx \sin \frac{\eta}{2}}
  \dd \eta\notag\\
  &\leqslant \pi 2^{k+1}b^{-(k+1)} k !.
\end{align}
Hence, the analyticity of $b\mapsto \phi_n(bx)$ for any $x>0$ and $b\in (0,1)$ is ensured by the case $s=0$ in
\eqref{eq:phi-n-general-bound}, indeed,
\begin{align*}
  \Big|\frac{\dd^k (\phi_{n}(bx))}{\dd b^k}\Big| \leqslant  2\pi\, k!\, b^{-k}.
\end{align*}
Choosing $s=1$ in \eqref{eq:phi-n-general-bound}, it follows that
\begin{align*}
  \Big|\frac{\dd^k \phi_{n}(bx)}{\dd b^k}\Big|=|x^k\phi^{(k)}_n(bx)|
   = |x||x^{k-1}\phi^{(k)}_n(bx)|\leqslant 4\pi b^{-k+1}(k-1)!x,
\end{align*}
which together with \eqref{eq:phi-n-decay-bound} leads to
\begin{align}\label{eq:uniform_bound_phi_nb}
  \Big|\frac{\dd^k \phi_{n}(bx)}{\dd b^k}\Big|=|x^k\phi^{(k)}_n(bx)|\leqslant  \frac{C^kk!}{b^k}\frac{x}{1+x^2},
\end{align}
with some absolute constant $C\geqslant 1$.
Next, for $\lambda_{n,b}$ given by \eqref{def:lam-n-1},
we see that for any $k\in \mathbb{N}$,
\begin{align}\label{eq:lamb-n1-diff}
  \frac{\dd^k \lambda_{n,b}}{\dd b^k}=\int_0^{+\infty}\frac{\dd^k \phi_{n}(bx)}{\dd b^k} \frac{\dd \mu(x)}{x}.
\end{align}
According to \eqref{cond:K0} and \eqref{es:f-weig}, we infer that for $\alpha\in (0,1)$,
\begin{align*}
  \int_0^1 |K_0'(t)|t^{1-\alpha+\alpha^2}\dd t
  \leqslant (1-\alpha +\alpha^2 ) \int_0^1 |K_0(t)| t^{-\alpha +\alpha^2} \dd t +  |K_0(1)| < +\infty.
\end{align*}
By virtue of \eqref{eq:K0prim} and Fubini's theorem, it follows that
\begin{align*}
  \int_0^1 |K_0'(t)|t^{1-\alpha+\alpha^2}\dd t
  & = \int_0^1 \int_0^{+\infty}e^{-tx} t^{1-\alpha+\alpha^2}\dd \mu(x) \dd t \\
  & = \int_0^{+\infty}\int_0^1 e^{-tx} t^{1-\alpha+\alpha^2}\dd t \,\dd \mu(x) \\
  & = \int_0^{+\infty}x^{-2+\alpha-\alpha^2} \int_0^x e^{-t}t^{1-\alpha+\alpha^2}\dd t \,\dd \mu(x) < +\infty.
\end{align*}
Noting that
\begin{align*}
  \lim_{x\to 0^+} \bigg( x^{-2+\alpha-\alpha^2} \int_0^x e^{-t}t^{1-\alpha+\alpha^2} \dd t \bigg)
  = \frac{1}{2-\alpha+\alpha^2} > 0,
\end{align*}
there exists some constant $c_{\alpha}>0$ such that
\begin{align*}
  \frac{c_{\alpha}}{1+x^{2-\alpha+\alpha^2}}\leqslant x^{-2+\alpha-\alpha^2} \int_0^x e^{-t}t^{1-\alpha+\alpha^2}\dd t,\quad \forall x>0,
\end{align*}
which directly implies that
\begin{align}\label{eq:integrablity_mu}
  \int_0^{+\infty}\frac{1}{1+x^2}\dd \mu(x)\leqslant 2\int_0^{+\infty}\frac{1}{1+x^{2-\alpha+\alpha^2}}\dd \mu(x)<+\infty.
\end{align}
Hence, from \eqref{eq:uniform_bound_phi_nb}, \eqref{eq:lamb-n1-diff}, and \eqref{eq:integrablity_mu},
we see that for every $k\in \mathbb{N}^\star$,
\begin{align*}
  \Big|\frac{\dd^k \lambda_{n,b}}{\dd b^k}\Big| \leqslant \frac{C^k k!}{b^k}
  \int_0^{+\infty} \frac{1}{1+x^2} \dd \mu(x) \leqslant  \frac{C^{k+1} k!}{b^k} ,
\end{align*}
which concludes the analyticity of $b \mapsto \lambda_{n,b}$ in the interval $(0,1)$.
\vskip1.5mm
	
\noindent \textbf{(ii)}
Recalling that
\begin{align*}
  \phi_{n,b}(x)=\int_0^{2\pi} e^{-x \sqrt{1+b^2 - 2b\cos \eta}} e^{i n\eta} \dd \eta
  = 2\int_0^{\pi}e^{-x \sqrt{1+b^2 - 2b\cos \eta}}\cos (n\eta) \,\dd \eta,
\end{align*}
we introduce the extension map:
\begin{align*}
  \mathbf{z} \in\mathbf{D}_1 \mapsto \Phi_{n,\mathbf{z}}(x) \triangleq
  2 \int_0^{\pi}e^{-x \sqrt{1+\mathbf{z}^2 - 2\,\mathbf{z}\,\cos \eta}} \cos (n\eta) \dd \eta,
\end{align*}
where the domain $\mathbf{D}_1$ is given by
\begin{align*}
  \mathbf{D}_1 \triangleq \big\{\mathbf{z} = z_1 + i z_2 \, : \,|z_2| < 1 - z_1,\,z_1 \in (0,1)\big\}.
\end{align*}
For $\mathbf{z} = b\in (0,1) \subset \mathbf{D}_1$, we see that $\Phi_{n,\mathbf{z}}(x) = \phi_{n, b}(x)$,
thus we intend to show that $\Phi_{n,\mathbf{z}}$ and the associated map
\begin{align}\label{eq:extend_map_lambda_n_3}
  \mathbf{z}\in \mathbf{D}_1\mapsto \int_{0}^{+\infty}\Phi_{n,\mathbf{z}}(x)\frac{\dd \mu(x)}{x}
\end{align}
are both analytic in $\mathbf{z}\in \mathbf{D}_1$.
By the definition of $\mathbf{D}_1$, it follows that for any $\mathbf{z}\in \mathbf{D}_1$ and
$\eta\in [0,\pi]$,
\begin{align*}
  \mathrm{Re}\,(1+\mathbf{z}^2-2\mathbf{z}\cos \eta) = 1+ z_1^2 -z_2^2 - 2 z_1 \cos \eta
  \geqslant (1-z_1)^2 - z_2^2 > 0.
\end{align*}
Since $\mathbf{z}\mapsto \sqrt{\mathbf{z}}$ has a unique analytic branch for $\mathrm{Re}\,\mathbf{z}>0$
and $1+\mathbf{z}^2-2 \mathbf{z}\cos \eta$
is analytic on $\mathbf{z}\in \mathbf{D}_1$, we know that for any fixed $\eta\in [0,\pi]$ and $x>0$,
the map $\mathbf{z}\in \mathbf{D}_1\mapsto e^{-x\sqrt{1+\mathbf{z}^2-2 \mathbf{z}\cos \eta}}$ is analytic.
Obviously, the map $\mathbf{z}\in \mathbf{D}_1\mapsto \Phi_{n,\mathbf{z}}(x)$ is continuous.
Furthermore, for every closed piecewise $C^{1}$ curve $\Gamma$ in $\mathbf{D}_1$,
the function $\mathbf{z}\in \Gamma \mapsto e^{-x\sqrt{1+\mathbf{z}^2-2\,\mathbf{z}\cos \eta}}$
is uniformly bounded for any $\eta\in [0,\pi]$.
Hence, we apply Fubini's theorem to get that
\begin{align*}
  \oint_{\Gamma}\Phi_{n,\mathbf{z}}(x) \dd \mathbf{z}
  = 2\int_{0}^{\pi}\oint_{\Gamma}e^{-x\sqrt{1+\mathbf{z}^2-2\,\mathbf{z}\cos \eta}}\dd \mathbf{z}\cos n\eta \, \dd \eta=0,
\end{align*}
where we have used the fact that $\mathbf{z}\in \mathbf{D}_1\mapsto e^{-x\sqrt{1+\mathbf{z}^2-2\,\mathbf{z}\cos \eta}}$
is analytic.
Consequently, by Morera's theorem (e.g. see \cite{SteinS03}), for any fixed $x>0$,
the map $\mathbf{z}\in \mathbf{D}_1\mapsto \Phi_{n,\mathbf{z}}(x)$ is analytic.
Next, we focus on the map given by \eqref{eq:extend_map_lambda_n_3}.
By virtue of Lebesgue's dominated convergence theorem and the fact $\Phi_{n,\mathbf{z}}(0)=0$,
one can check that this map \eqref{eq:extend_map_lambda_n_3} is continuous in $\mathbf{D}_1$,
and in addition, for every closed piecewise $C^{1}$ curve $\Gamma$ in $\mathbf{D}_1$,
there exists a positive constant $C_{\Gamma}>0$ such that
\begin{align*}
  |\Phi_{n,\mathbf{z}}(x)|\leqslant C_\Gamma xe^{-x/C_\Gamma }.
\end{align*}
Taking advantage Fubini's theorem and \eqref{eq:integrablity_mu},
we deduce that for every closed piecewise $C^{1}$ curve $\Gamma\subset \mathbf{D}_1$,
\begin{align*}
  \oint_{\Gamma}\int_{0}^{+\infty}\Phi_{n,\mathbf{z}}(x) \frac{\dd \mu(x)}{x}\dd \mathbf{z}
  =\int_0^{+\infty}\oint_{\Gamma}\Phi_{n,\mathbf{z}}(x) \,\dd \mathbf{z}\frac{\dd \mu(x)}{x}=0.
\end{align*}
As a consequence of Morera's theorem (\cite{SteinS03}), the map \eqref{eq:extend_map_lambda_n_3} is analytic and thus
$ b\in (0,1) \mapsto\widetilde{\lambda}_{n,b}$ is analytic.
\end{proof}

\begin{remark}\label{rmk:assumA2}
From the complete monotonicity of $-K_0'(t)$ and \eqref{eq:integrablity_mu}, we see that
\begin{align*}
  |K'_0(t)|\leqslant |K'_0(1)|=\int_0^{+\infty}e^{-x}\dd \mu(x)\leqslant C_0,\quad \forall t\geqslant 1,
\end{align*}
which implies that
\begin{align*}
  |K_0(t)|\leqslant |K_0(1)|+\int_1^t |K_0'(t)|\leqslant C_0(t+1), \quad \forall t\geqslant 1.
\end{align*}
Then it follows from \eqref{A2} that for any $a_0>0$,
\begin{align*}
  \int_0^{a_0}|K_0(t)|t^{-\alpha+\alpha^2}\dd t
  & \leqslant
  \int_0^1 |K_0(t)|t^{-\alpha+\alpha^2}\dd t+\int_1^{\max(1,a_0)}|K_0(t)|t^{-\alpha+\alpha^2}\dd t\\
  & \leqslant \int_0^1 |K_0(t)|t^{-\alpha+\alpha^2}\dd t+C_0\int_1^{\max(1,a_0)}(t+1)t^{-\alpha+\alpha^2}\dd t<+\infty.
\end{align*}
\end{remark}

\section{Spectral Study}\label{sec:spectral-stud}
In this section, we mainly study the "nonlinear eigenvalues" corresponding to the solutions in $\Omega$ of the dispersion equation
\eqref{eq:Omega-eq}. They are denoted by $\Omega^\pm_{n,b}$ as described in \eqref{eq:Omega_n_pm-intr}
provided that the discriminant $\mathbf{\Delta}_{n,b}$ defined in \eqref{discriminant} is strictly positive.
We first derive the pointwise lower/upper bounds of some spectral terms appearing in \eqref{eq:Omega_n_pm-intr}.
Next, we show the positivity of $\mathbf{\Delta}_{n,b}$ for sufficiently large $n$
under the condition \eqref{eq:assum-1}; and furthermore, for the special case $K=K_0(|\mathbf{x}-\mathbf{y}|)$
we show that condition \eqref{eq:assum-1} holds by removing at most a countable set of $b\in \mathcal{S}_{\mathrm{max}}$.
Finally, for $n$ large enough, we study the monotonicity property in $n$ of the real roots
$\Omega^\pm_{n,b}$.

\subsection{Pointwise bounds of several spectral terms}

In this subsection, we present some upper and lower bounds about the spectral terms
$\lambda_{n,b}$, $\widetilde{\lambda}_{n,b}$ and $p_{n,b}$, $\widetilde{p}_{n,b}$,
which are given by \eqref{def:lambda-nb} and \eqref{def:Pnb}, respectively.

\begin{proposition}\label{prop:lambda_property}
Let assumptions \eqref{A1}-\eqref{A4} be satisfied with some $\alpha\in (0,1)$.
For any $0<b\leqslant 1$ and $n \in \mathbb{N}^\star$, the following estimates hold:
\begin{align}\label{eq:est-upper}
  \frac{1}{C n^2}\leqslant  \lambda_{n,b}\leqslant  \frac{C}{n^{\alpha-\alpha^2}},
  \quad  \frac{1}{C n^3}\leqslant  \lambda_{n,b}-\lambda_{n+1,b}\leqslant  \frac{C}{n^{1+\alpha-\alpha^2}}
\end{align}
and
\begin{align}
  0<\widetilde{\lambda}_{n,b} & \leqslant  \frac{C_{k,b}}{n^k},
  \quad \forall k\in \mathbb{N}^\star, \label{eq:est-upper2} \\
  |p_{n,b}| + |\widetilde{p}_{n,b}| & \leqslant  \frac{C}{n^4},  \label{eq:est-upper3}
\end{align}
with some $C\geqslant 1$ and $C_{k,b}>0$.
\end{proposition}

\begin{proof}[Proof of Proposition \ref{prop:lambda_property}]
We begin with the estimation in \eqref{eq:est-upper} concerning $\lambda_{n,b}$ given by \eqref{def:lambda-nb}.
Via integrating by parts, we have
\begin{align}\label{eq:lambda-nb-ub1}
  0 < \lambda_{n,b} & =
  \Big|\int_0^{2\pi} K_0(2 b \sin \tfrac{\eta}{2}) \cos (n\eta)\, \dd \eta\Big| \nonumber \\
  & = \frac{b}{n}\Big|\int_0^{2\pi}K'_0(2b \sin \tfrac{\eta}{2})
  \sin (n\eta) \cos \tfrac{\eta}{2}\,\dd \eta\Big| \nonumber \\
  & \leqslant  \frac{C}{n^{\alpha-\alpha^2}}\int_0^{2\pi}|K'_0(2 b\sin \tfrac{\eta}{2})|
  |\sin \tfrac{\eta}{2}|^{1-\alpha+\alpha^2}\dd \eta,
\end{align}
where we have used that $\big|\frac{\sin (n\eta)}{\sin \frac{\eta}{2}}\big|\leqslant Cn$ and $|\sin n\eta|\leqslant 1$.
In addition, using the fact that
\begin{align*}
  \tfrac{2}{\pi}\eta\leqslant  \sin \eta \leqslant  \eta,\quad 0\leqslant  \eta\leqslant  \tfrac{\pi}{2},
\end{align*}
together with the decreasing property of $|K_0'(t)|$, Lemma \ref{lem:int} and \eqref{cond:K0}, we deduce that
\begin{align}\label{eq:lambda-nb-ub2}
  \int_0^{2\pi}|K'_0(2b\sin \tfrac{\eta}{2})| |\sin \tfrac{\eta}{2}|^{1-\alpha+\alpha^2}\dd \eta \nonumber
  & = 4\int_0^{\frac{\pi}{2}}|K'_0(2 b\sin \eta)| |\sin \eta|^{1-\alpha+\alpha^2}\dd \eta \nonumber \\
  & \leqslant  4\int_0^{\frac{\pi}{2}}|K_0'(\tfrac{4b}{\pi}\eta)|\eta^{1-\alpha+\alpha^2}\dd \eta \nonumber \\
  & \leqslant C \int_0^{2b} |K_0(\eta)| \eta^{-\alpha +\alpha^2} \dd \eta + C \leqslant C ,
\end{align}
which combined with \eqref{eq:lambda-nb-ub1} leads to the upper bound of $\lambda_{n,b}$ in \eqref{eq:est-upper}.
The lower bounds of $\lambda_{n,b}$ and $\lambda_{n,b}-\lambda_{n+1,b}$ follow from \eqref{def:lam-n-1},
\eqref{ineq:est-phi-n}-\eqref{ineq:est-phi-n-diff}
and the fact that the measure $\mu$ is a nonzero non-negative measure.
Indeed, there exist some $0<d<\infty$ and $c_\star>0$ such that $\mu([0,d])\geqslant c_\star$, thus
\begin{align*}
  \lambda_{n,b} & \geqslant  \int_0^\infty \frac{8n^2}{4n^2 +1} \frac{b}{n^2 + b^2 x^2} \dd \mu(x) \\
  & \geqslant \frac{8b}{5} \int_0^d\frac{1}{n^2 + b^2 x^2} \dd \mu(x) \geqslant \frac{8b}{5} \frac{c_\star}{n^2 + b^2 d^2}
  \geqslant \frac{1}{Cn^2}
\end{align*}
and
\begin{align*}
  \lambda_{n,b} -\lambda_{n+1,b} & \geqslant  \int_0^\infty \frac{(2n+1)b}{(n^2+b^2x^2)((n+1)^2+ b^2 x^2)} \dd \mu(x) \\
  & \geqslant \frac{c_\star (2n+1) b}{(n^2 + b^2 d^2)((n+1)^2 + b^2 d^2)}
  \geqslant \frac{1}{C n^3}.
\end{align*}
For the upper bound of $\lambda_{n,b}-\lambda_{n+1,b}$, from \eqref{def:lambda-nb} and applying  integration by parts,
we deduce that
\begin{align*}
  \lambda_{n,b} -\lambda_{n+1,b}
  & = \Big|\int_0^{2\pi}K_0(2 b \sin \tfrac{\eta}{2})\Big( \cos (n\eta)-\cos ((n+1)\eta)\Big)\, \dd \eta\Big| \\
  & = 2\Big|\int_0^{2\pi}K_0(2 b \sin \tfrac{\eta}{2}) \sin \tfrac{\eta}{2} \,\sin (n\eta+\tfrac{\eta}{2})\, \dd \eta\Big| \\
  & = \frac{2}{n+\tfrac{1}{2}}
  \Big|\int_0^{2\pi} \cos(n\eta + \tfrac{\eta}{2}) \frac{\dd}{\dd \eta} \Big(K_0(2b\sin\tfrac{\eta}{2}) \sin\tfrac{\eta}{2} \Big)
  \dd \eta \Big| \\
  & = \frac{2}{(n+\tfrac{1}{2})^2}\Big|\int_0^{2\pi}\sin (n \eta +\tfrac{\eta}{2})  \,\frac{\dd^2 }{\dd \eta^2}
  \Big(K_0(2 b \sin \tfrac{\eta}{2})\sin \tfrac{\eta}{2}\Big)\, \dd \eta\Big|,
\end{align*}
where in the third and fourth lines we have used the following estimates to eliminate the boundary contribution
(owing to assumptions \eqref{A1}-\eqref{A2} and Lemma \ref{lem:int})
\begin{equation}\label{eq:est-t-0-K0-0}
\begin{aligned}
  \lim_{t\to 0^{+}}|K_0(t)t| & \leqslant \lim_{t\to 0^{+}}t\int_{t}^{t_0} \big(-K_0'(\tau) \big)\dd \tau +\lim_{t\to 0^{+}}|K_0(t_0)|t \\
  & \leqslant \lim_{t\to 0^{+}} t^{\alpha-\alpha^2} \int_{t}^{t_0} \big(-K_0'(\tau)\tau^{1-\alpha+\alpha^2}\big) \dd \tau=0
\end{aligned}
\end{equation}
and
\begin{equation}\label{eq:est-t-0-K0-1}
\begin{aligned}
  \lim_{t\to 0^{+}} \big(-K'(t)t^2\big) & \leqslant
  \lim_{t\to 0^{+}}t^{\alpha-\alpha^2}\int_t^{t_0}K''_0(t)t^{2-\alpha+\alpha^2}\dd t
  + \lim_{t\to 0^{+}}t^2|K''(t_0)| \\
  & \leqslant \lim_{t\to 0^{+}}t^{\alpha-\alpha^2}\Big(\int_t^{t_0}K_0(t)t^{-\alpha+\alpha^2}\dd t
  + |K_0(t_0)|t_0^{1-\alpha+\alpha^2}\Big)=0.
\end{aligned}
\end{equation}
In view of the fact that
$|\frac{\sin (n+\frac{1}{2})\eta}{\sin \frac{\eta}{2}}|\leqslant C(2n+1)$ and $|\sin (n+\tfrac{1}{2})\eta|\leqslant 1$,
and using Lemma \ref{lem:int}, we can argue as \eqref{eq:lambda-nb-ub1}-\eqref{eq:lambda-nb-ub2} to deduce that
\begin{align*}
  \lambda_{n,b} -\lambda_{n+1,b} & \leqslant \frac{C}{(n+\tfrac{1}{2})^{1+\alpha-\alpha^2}}
  \int_0^{2\pi}(\sin\tfrac{\eta}{2})^{^{1-\alpha+\alpha^2}} \Big|\frac{\dd^2 }{\dd \eta^2}
  \Big(K_0(2 b \sin \tfrac{\eta}{2})\sin \tfrac{\eta}{2}\Big)\Big|\, \dd \eta \\
  & \leqslant \frac{C}{(n+\tfrac{1}{2})^{1+\alpha-\alpha^2}},
\end{align*}
as desired.
\vskip1mm
	
\noindent
Next, we consider the estimates \eqref{eq:est-upper2}-\eqref{eq:est-upper3}.
For $\phi_{n,b}(x)$ given by \eqref{def:phi-n-b}, through integrating by parts and using \eqref{eq:basic_n_b},
we directly obtain that for every $k\in \mathbb{N}^\star$,
\begin{align*}
  0 < \phi_{n,b}(x) & = \Big|\int_0^{2\pi} e^{-x \sqrt{1+b^2 - 2b\cos \eta}} \cos(n\eta) \dd \eta \Big| \\
  & = \frac{1}{n} \Big|\int_0^{2\pi} e^{-x \sqrt{1+b^2 -2b\cos\eta}} \frac{bx \sin\eta}{\sqrt{1+b^2-2b\cos\eta}}
  \sin(n\eta) \dd \eta\Big| \\
  & \leqslant \frac{1}{n^k} \int_0^{2\pi} \Big|\partial^{k-1}_\eta
  \Big(e^{-x \sqrt{1+b^2 -2b\cos\eta}} \frac{bx \sin\eta}{\sqrt{1+b^2-2b\cos\eta}} \Big)\Big|\dd \eta \\
  & \leqslant\frac{C_{k,b}}{n^{k}}e^{-(1-b)x} x\, \mathcal{P}_{k-1}(x),
\end{align*}
where $\mathcal{P}_k(x)$ is the $k$-th order polynomial about $x$.
Hence, for $\widetilde{\lambda}_{n,b}$ given by \eqref{def:lam-n-3},
by virtue of \eqref{eq:integrablity_mu}, we infer that for any $k\geqslant  0$,
\begin{align*}
  0 < \widetilde{\lambda}_{n,b} & \leqslant \frac{C_{k,b}}{n^k} \int_0^{+\infty} e^{-(1-b)x} \, \mathcal{P}_{k-1}(x) \dd \mu(x)
  \nonumber \\
  & \leqslant \frac{C_{k,b}}{n^k} \Big(\sup_{x\in [0,+\infty)} \frac{\mathcal{P}_{k-1}(x) (1+x^2)}{e^{(1-b)x}}\Big)
  \int_0^{+\infty} \frac{1}{1+x^2} \dd \mu(x)
  \leqslant \frac{\widetilde{C}_{k,b}}{n^{k}}.
\end{align*}
Recalling that $p_{n,b}$ and $\widetilde{p}_{n,b}$ are defined in \eqref{def:Pnb},
the corresponding estimates in \eqref{eq:est-upper3} also follow from integration by parts.
Indeed, taking $p_{n,b}$ as an example, thanks to \eqref{A3}, it follows from the periodicity of $\eta\in [0,2\pi]\mapsto K_1(b,be^{i\eta})$ that
\begin{align*}
  |p_{n,b}| & = \Big|\int_0^{2\pi} K_1(b, b e^{i\eta})e^{in\eta}\dd \eta\Big| \\
  & = \frac{1}{n^4}\Big|\int_0^{2\pi} K_1(b,b e^{i\eta}) \Big(\partial^4_\eta e^{in\eta} \Big) \dd \eta\Big| \\
  & \leqslant \frac{1}{n^4}\Big|\int_0^{2\pi} \partial_\eta^4 \Big(K_1(b,b e^{i\eta})\Big)e^{in\eta}\dd \eta \Big|
  \leqslant  \frac{C}{n^4}.
\end{align*}
This completes the proof of Proposition \ref{prop:lambda_property}.
\end{proof}

\subsection{Asymptotic structure of the discriminant}
The stability of the annular solution $\mathbf{1}_{\mathbb{D}\setminus \overline{b\mathbb{D}}}$
is related to the sign of the discriminant $\mathbf{\Delta}_{n,b}$ introduced in \eqref{discriminant},
see Section \ref{subsec:stable-mode} for more discussion.
It is important to note that bifurcation may occur when neutral stable modes with simple eigenvalues are present.
Analyzing the sign of the discriminant is quite challenging due to the complexity of its structure.
However, in this work, we will focus on its behavior for large modes by leveraging the asymptotic behavior of
the involved functions. This approach will provide valuable information about the stability criteria in the asymptotic regime.

\noindent We have the following result on the positivity of $\mathbf{\Delta}_{n,b}$.
\begin{proposition}\label{prop:Delta-analysis}
Let the kernel function $K$ satisfy the assumptions \eqref{A1}-\eqref{A4}.
Then the following statements are valid.
\begin{enumerate}
\item
We have
\begin{equation}\label{eq:Delta-inf-b}
\begin{split}
  \mathbf{\Delta}_{\infty,b} \triangleq \lim_{n\to\infty} \mathbf{\Delta}_{n,b}
  = \big(V_b^1[\mathbf{0}] - V_b^2[\mathbf{0}] \big)^2
  = \left( \int_0^\infty \tfrac{\Psi_b(x)}{x} \dd\mu(x)
  + \mathsf{c}_b - \widetilde{\mathsf{c}}_b \right)^2,
\end{split}
\end{equation}
with
\begin{align}\label{def:Psi-bx}
  \Psi_b(x)\triangleq \phi_1(x)+\phi_1(bx)
  -(b+\tfrac{1}{b})\phi_{1,b}(x),
\end{align}
and the coefficients $\mathsf{c}_b$ and $\widetilde{\mathsf{c}}_b$
are defined in \eqref{c-1} and \eqref{c-2}, respectively.
\item
Assume that $b\in \mathcal{S}_{\mathrm{max}}$ is such that
\begin{align}\label{eq:assum-1}
  V_b^1[\mathbf{0}] - V_b^2[\mathbf{0}] \ne 0.
\end{align}
Then there exists $N_b\in\mathbb{N}^\star$ so that for any $n\geqslant N_b$, $\mathbf{\Delta}_{n,b}$ is strictly positive.
\end{enumerate}
\end{proposition}

\begin{proof}[Proof of Proposition \ref{prop:Delta-analysis}]
(i)
From $A_{n,b}$ and $B_{n,b}$ given by \eqref{eq:def-A-n}-\eqref{eq:def-B-n}, and
in view of Proposition \ref{prop:lambda_property}, we see that
\begin{equation}\label{eq:Anb-conver}
\begin{split}
  \lim_{n\to \infty}A_{n,b} = - V_b^1[\mathbf{0}]
  & = \tfrac{1}{b}\widetilde{\lambda}_{1,b}-\lambda_{1,b} -\mathsf{c}_b \\
  & = \int_0^{\infty} \left(\frac1b \phi_{1,b}(x)-\phi_1(bx)\right)\frac{\dd \mu(x)}{x}-\mathsf{c}_b,
\end{split}
\end{equation}
and
\begin{equation}\label{eq:Bnb-conver}
\begin{split}
  \lim_{n\to \infty}B_{n,b} = - V_b^2[\mathbf{0}]
  & =-b\widetilde{\lambda}_{1,b}+\lambda_{1,1} -\widetilde{\mathsf{c}}_b\\
  & = -\int_0^{\infty} \Big(b\,\phi_{1,b}(x)-\phi_{1}(x)\Big)\frac{\dd\mu(x)}{x} -\widetilde{\mathsf{c}}_b,
\end{split}
\end{equation}
where we have used the expressions \eqref{def:lam-n-1} and \eqref{def:lam-n-3}.
Applying  \eqref{discriminant} and Proposition \ref{prop:lambda_property},
it follows that
\begin{align*}
  \lim_{n\rightarrow \infty} \mathbf{\Delta}_{n,b}
  = \lim_{n\rightarrow \infty} \Big(\big(A_{n,b} - B_{n,b}  \big)^2
  - 4 \big(\widetilde{\lambda}_{n,b} + \widetilde{p}_{n,b}\big)^2 \Big)
  = \big(V_b^1[\mathbf{0}] - V_b^2[\mathbf{0}] \big)^2.
\end{align*}
Then combined with \eqref{eq:Anb-conver}-\eqref{eq:Bnb-conver}, this leads to \eqref{eq:Delta-inf-b}, as desired.

\noindent (ii)
Condition \eqref{eq:assum-1} implies that
$\lim\limits_{n\to +\infty}\mathbf{\Delta}_{n,b}
  = ( V_b^1[\mathbf{0}] - V_b^2[\mathbf{0}])^2 > 0$.
By continuity, there exists a sufficiently large $N_b\in\mathbb{N}^\star$
depending on $b$ such that $\mathbf{\Delta}_{n,b}>0$ for any $n\geqslant N_b$.
\end{proof}

\begin{remark}\label{rem:b-small-discussion}
In the convolution case $K(\mathbf{x},\mathbf{y})=K_0(|\mathbf{x} -\mathbf{y}|)$,
one can consider the vanishing limit as $b\to 0$, which will transform the doubly connected case to the simply connected one.
We expect as in \cite{HXX23b} the existence of some $b^\ast>0$ such that for any $b\in (0,b^\ast)$,
the discriminant $\mathbf{\Delta}_{n,b}$ remains strictly positive for all $n\in \mathbb{N}^\star$.
However, achieving this for general kernels appears to be challenging.
\\[1mm]
Indeed, from  \eqref{discriminant}, we write
\begin{align*}
  \mathbf{\Delta}_{n,b}=(B_{n,b}-A_{n,b}-2\widetilde{\lambda}_{n,b})(B_{n,b}-A_{n,b}+2\widetilde{\lambda}_{n,b}).
\end{align*}
Thanks to \eqref{def:lam-n-1}, \eqref{def:lam-n-3} and \eqref{eq:def-A-n}-\eqref{eq:def-B-n}, we find that $B_{n,b}-A_{n,b}\pm 2\widetilde{\lambda}_{n,b}$ equals to
\begin{align*}
  \int_0^{\infty} \Big(\phi_1(x)-\phi_n(x)+\phi_1(bx)-\phi_n(bx)-(b+\tfrac{1}{b})\phi_{1,b}(x)\pm 2\phi_{n,b}(x) \Big)
  \frac{\dd \mu(x)}{x}.
\end{align*}
Consequently, in light of \eqref{eq:phi_n_ori}, \eqref{def:phi-n-b} and \eqref{eq:phi1b-conver}, we find that for any fixed $x>0$,
\begin{align*}
  \lim_{b\to 0^+} \Big(\phi_1(x)-\phi_n(x)+\phi_1(bx)-\phi_n(bx)-(b+\tfrac{1}{b})
  \phi_{1,b}(x)\pm 2\phi_{n,b}(x) \Big) = \phi_1(x)-\phi_n(x)-\pi xe^{-x},
\end{align*}
and thus, formally,
\begin{align*}
  \lim_{b\rightarrow 0^+} \mathbf{\Delta}_{n,b} = \bigg(\int_0^{\infty} \Big( \phi_1(x) -\phi_n(x) - \pi x e^{-x} \Big)
  \frac{\dd \mu(x)}{x} \bigg)^2.
\end{align*}
At this level, the pointwise limit is nonnegative,
but it is not clear whether it is strictly positive as the integrand may change sign.
In fact, similar to $\Psi_b(x)$ in the lemma \ref{lem:kernel_function_limit},
there exists $0<x_\star< x^\star $ such that the function $x\in (0,+\infty)\mapsto\phi_1(x)-\phi_n(x)-\pi xe^{-x}$
is strictly negative for $x\in (0,x_\star)$ but strictly positive for $x\in (x^\star,+\infty)$.
\end{remark}
In the following, we present some properties of the functions $\Psi_b$  given by \eqref{def:Psi-bx} along with the functions
$$x\in(0,\infty)\mapsto \phi_1(bx)-\tfrac{1}{b}\phi_{1,b}(x)
\quad\hbox{and}\quad x\in(0,\infty)\mapsto b\phi_{1,b}(x) - \phi_1(x)
$$
which play a crucial role in analyzing the sign of $V_b^1[\mathbf{0}]$, $V_b^2[\mathbf{0}]$
and $V_b^1[\mathbf{0}] - V_b^2[\mathbf{0}]$.
\begin{lemma}\label{lem:kernel_function_limit}
The following assertions are true.
\begin{enumerate}
\item
There exist $x_{\ast}\in (0,+\infty)$ and
$b_{\ast}\in (0,1)$ such that for every $b\in (0,b_{\ast})$ and $x\in (0,x_\ast)$,
\begin{align}\label{eq:negative-argument}
  \phi_1(bx) - \tfrac{1}{b} \phi_{1,b}(x) < 0,\quad b\phi_{1,b}(x) - \phi_1(x)<0,\quad \Psi_b(x)<0.
\end{align}
\item
For any $b\in(0,1),$ there exists $x^\ast = x^\ast(b) > 0$
such that for every $x\in (x^{\ast},+\infty)$,
\begin{align}\label{eq:Psi-b-pos}
  \phi_1(bx) - \tfrac{1}{b} \phi_{1,b}(x) >0,\quad b\phi_{1,b}(x) - \phi_1(x) < 0,\quad \Psi_b(x)>0.
\end{align}
\item
For $b=0.5$ we have
\begin{align*}
  \Psi_{0.5}(x)>0,\quad \forall x>0.
\end{align*}
\end{enumerate}
\end{lemma}


\begin{proof}[Proof of Lemma \ref{lem:kernel_function_limit}]
{\bf (i)}
Recall from the expression formula \eqref{exp:phi-nb} that
\begin{align}\label{eq:phi-1b}
  \phi_{1,b}(x) = 2bx\int_{-1}^{1}e^{-x\sqrt{1+b^2-2by}}
  \frac{(1-y^2)^{\frac{1}{2}}}{(1+b^2-2by)^{\frac{1}{2}}}\dd y,
\end{align}
and
\begin{align}\label{eq:phi-b=1}
  \phi_1(x) = \phi_{1,1}(x) = \sqrt{2} x\int_{-1}^1 e^{-x\sqrt{2-2y}}
  (1+y)^{\frac{1}{2}} \dd y.
\end{align}
Thus, we find
\begin{align}\label{eq:phi1b-conver}
  \lim_{b\rightarrow 0^+} b \phi_{1,b}(x) = 0,\quad \textrm{and}\quad \lim_{b\rightarrow 0^+} \tfrac{1}{b} \phi_{1,b}(x)
  = 2x\int_{-1}^{1}e^{-x}(1-y^2)^{\frac{1}{2}}\dd y
  = \pi xe^{-x}.
\end{align}
The above formula together with
\begin{align*}
  |\sqrt{1+x}-1|\leqslant|x|,\quad
  \forall x>-1, \quad 1+b^2-2by=(1-b)^2+2b(1-y),
\end{align*}
allow getting
\begin{align}\label{eq:phi1b-diff}
  \Big|\tfrac{1}{b}\phi_{1,b}(x)-\pi xe^{-x}\Big| & \leqslant
  2x \bigg(\int_{-1}^{1}e^{-x\sqrt{1+b^2-2by}} \frac{|(1+b^2-2by)^{\frac{1}{2}}-1|}{(1+b^2-2by)^{\frac{1}{2}}}
  (1-y^2)^{\frac{1}{2}}\dd y \nonumber \\
  & \quad \quad\quad+ \int_{-1}^{1}\Big|e^{-x\sqrt{1+b^2-2by}}-e^{-x}\Big|(1-y^2)^{\frac{1}{2}}\dd y \bigg) \nonumber \\
  & \leqslant 2x \bigg( e^{-(1-b)x}\int_{-1}^{1}\frac{|b^2-2by|}{(1-b)} (1-y^2)^{\frac{1}{2}}\dd y \nonumber \\
  & \quad \qquad + \int_{-1}^{1}\int_0^{1}e^{-x(\sqrt{1+b^2-2by}\,\eta + (1-\eta))} \big|x(1-\sqrt{1+b^2-2by})\big|
  (1-y^2)^{\frac{1}{2}}\,\dd \eta \dd y \bigg) \nonumber \\
  & \leqslant 2xe^{-(1-b)x}\Big(\frac{2\pi b}{(1-b)}+2\pi bx\Big) .
\end{align}
On the other hand, we can  deduce from \eqref{eq:phi-1b} that
\begin{equation}\label{eq:phi_1_b_upb}
\begin{aligned}
  0<b\,\phi_{1,b}(x)&\leqslant 2b^2 x e^{-(1-b)x}\int_{-1}^{1}
  \frac{(1-y^2)^{\frac{1}{2}}}{(1+b^2 - 2by)^{\frac{1}{2}}}\dd y \\
  &\leqslant 4b^2 x e^{-(1-b)x},
\end{aligned}
\end{equation}
where we have used the inequality
$\frac{(1-y^2)^{1/2}}{(1+b^2 - 2by)^{1/2}} \leqslant 1$.
Hence, we obtain that for $b\in(0,1)$ and $x>0$,
\begin{align}\label{eq:uni-est-phi-1-b}
  \Big|(b+\tfrac{1}{b})\phi_{1,b}(x)-\pi xe^{-x}\Big|\leqslant  bxe^{-(1-b)x}\Big(\frac{4\pi}{1-b}+4b+4\pi x\Big).
\end{align}
In addition, by virtue of \eqref{ineq:est-phi-n}, we know that
\begin{align}\label{eq:est-phi-1-b}
  \frac{8}{5}\frac{bx}{1+(bx)^2} \leqslant \phi_1(bx) \leqslant \frac{8}{3}\frac{bx}{1+(bx)^2},\quad
  0<b\leqslant 1.
\end{align}
From \eqref{eq:phi1b-diff}, \eqref{eq:uni-est-phi-1-b} and \eqref{eq:est-phi-1-b}, it follows that for any $b\in(0,1)$ and $x>0$,
\begin{align*}
  \phi_1(bx) - \tfrac{1}{b} \phi_{1,b}(x) & \leqslant \phi_1(bx) - \pi x e^{-x}
  - \big( \tfrac{1}{b} \phi_{1,b}(x) - \pi x e^{-x} \big) \\
  & \leqslant x\Big(\frac{8 b}{3(1+(bx)^2)}-\pi e^{-x}\Big)
  + bxe^{-(1-b)x}\Big(\frac{4\pi}{1-b} + 4\pi x\Big) \triangleq x \Pi_1(b,x) ,
\end{align*}
\begin{align*}
  \frac{8 x }{3(1+x^2)} \geqslant \phi_1(x) - b \phi_{1,b}(x)
  \geqslant x \Big(\frac{8}{5(1+x^2)} - 4b^2 e^{-(1-b)x} \Big)
  \triangleq x \Pi_2(b,x),
\end{align*}
and
\begin{align*}
  \Psi_b(x) & = \phi_{1}(x) - b \phi_{1,b}(x) + \phi_1(bx)-\tfrac{1}{b}\phi_{1,b}(x) \\
  & \leqslant x\Big(\frac{8}{3(1+x^2)}+\frac{8 b}{3(1+(bx)^2)}-\pi e^{-x}\Big)
  + bxe^{-(1-b)x}\Big(\frac{4\pi}{1-b} + 4\pi x\Big)
  \triangleq x\, \Pi_3(b,x) .
\end{align*}
Note that
\begin{align*}
  \Pi_1(0,0)= -\pi <0,\quad \Pi_2(0,0) = \tfrac{8}{5}>0,\quad \textrm{and}\quad \Pi_3(0,0)=\tfrac{8}{3}-\pi<0.
\end{align*}
Therefore, by continuity, there exist two constants $x_{\ast}>0$ and $b_{\ast}>0$ such that \eqref{eq:negative-argument} holds.
\\[1mm]
\noindent {\bf (ii)}	We combine \eqref{def:Psi-bx}, \eqref{eq:phi_1_b_upb} and \eqref{eq:est-phi-1-b}, which leads  to the conclusion that for any $b\in(0,1)$ and $x>0$,
\begin{align*}
  \phi_1(bx) - \tfrac{1}{b} \phi_{1,b}(x) \geqslant \frac{8bx}{5(1+(bx)^2)}-4 x e^{-(1-b)x}
  \geqslant \frac{x}{1+(bx)^2} \Big(\frac{8b}{5} - 4 (1+b^2x^2) e^{-(1-b)x} \Big),
\end{align*}
and
\begin{align*}
  \phi_1(x) - b\phi_{1,b}(x) \geqslant \frac{x}{1+x^2} \Big(\frac{8}{5} - 4b^2  (1+x^2) e^{-(1-b)x} \Big).
\end{align*}
Therefore, we get
\begin{equation}\label{eq:positive-argument}
\begin{split}
  \Psi_b(x) 
  \geqslant \frac{8x}{1+x^2}\Big(\frac{1}{5}- (1+x^2) e^{-(1-b)x}\Big).
\end{split}
\end{equation}
This allows to get that for each fixed $b\in(0,1)$ there exists a sufficiently large constant $x^\ast=x^\ast(b)>0$
such that \eqref{eq:Psi-b-pos} holds.
\\[0.5mm]
\noindent {\bf (iii)} First observe that $\Psi_b(0)=0$ and (owing to \eqref{eq:phi-1b}-\eqref{eq:phi-b=1})
\begin{align*}
  \Psi_b^\prime(0) = \lim_{x\rightarrow 0^+} \Psi_b'(x)
  & = \lim_{x\rightarrow 0^+} \Big( \phi_1'(x)+ b\phi_1'(bx)
  -(b+\tfrac{1}{b})\phi_{1,b}'(x)\Big)   \\
  & = \sqrt{2}(1+b)\int_{-1}^1\sqrt{1+y} \,\dd y-2(b^2+1)
  \int_{-1}^1\sqrt{\tfrac{1-y^2}{1+b^2-2by}}\,\dd y \\
  &=\tfrac83(1+b)-2(b^2+1)\int_{-1}^1\sqrt{\tfrac{1-y^2}{1+b^2-2by}}\,\dd y.
\end{align*}
For $b=0.5$, we find by the numerical approximation,
\begin{align*}
  \int_{-1}^1\sqrt{\tfrac{1-y^2}{\frac54-y}}\,\dd y\approx 1.52,
\end{align*}
which implies that
\begin{align*}
  \Psi_{0.5}^\prime(0)>0.
\end{align*}
Hence, $\Psi_{0.5}(x)$ is strictly increasing for some small interval $x\in [0, \delta]$, and thus
\begin{align*}
  \Psi_{0.5}(x)>0,\quad  \forall x\in (0,\delta].
\end{align*}
On the other hand, from \eqref{eq:positive-argument} with $b=0.5$, we have
\begin{align*}
  \Psi_{0.5}(x) \geqslant \frac{8x}{1+x^2}\Big(\frac{1}{5}- \frac{1+x^2}{e^{\frac{1}{2}x}}\Big) >0,\quad \forall x\in [20,\infty).
\end{align*}
To show the positivity of $\Psi_{0.5}(x)$ inside the interval $[\delta, 20]$, we use Matlab software to draw the graph
of $\Psi_{0.5}(x)$ in the interval $[0,20],$ see Figure \ref{fig:A_B_infinity_half},
which clearly shows the required property.
An analytical proof seems difficult to achieve due to the complex structure of the function $\Psi_b(x)$.
This completes the proof.
\end{proof}

\begin{figure}[htbp]
  \centering
  \includegraphics[scale=0.7]{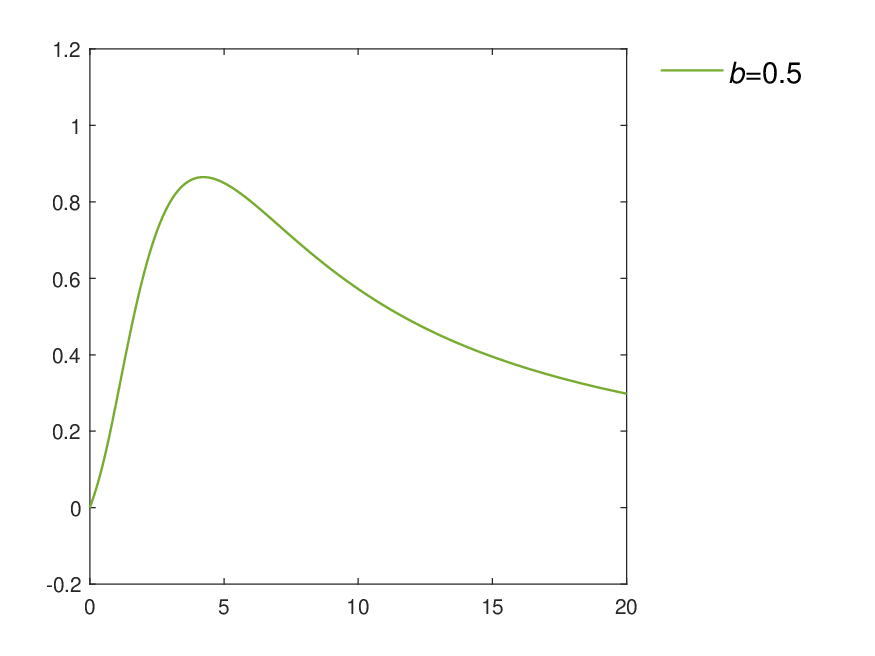} 
  \caption{The graph of $x\mapsto \Psi_b(x)$ for $b=0.5$.}\label{fig:A_B_infinity_half}
\end{figure}

\begin{remark}
  The numerical simulations illustrated in Figure \ref{fig:A_B_infinity} suggest that the function
$\Psi_b(x)$ is always positive for all $b\in [0.3,1)$.
\end{remark}

\begin{figure}[htbp]
\centering
\includegraphics[scale=0.7]{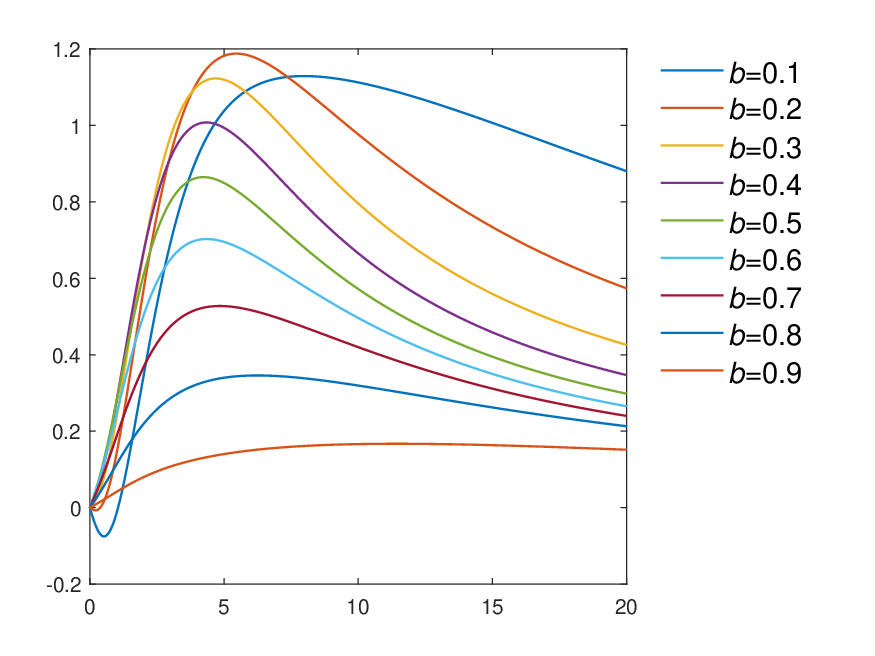} 
\caption{The graph of $x\mapsto \Psi_b(x)$ for different values of $b$.}\label{fig:A_B_infinity}
\end{figure}

In what follows, we shall explore spectrum collision in the convolution case $K_1(\mathbf{x},\mathbf{y})\equiv 0$ (multiple eigenvalues).
We show in particular that a collision cannot occur outside a countable set of values of $b.$

\begin{proposition}\label{prop:limit_zeros}
   Assume that $K(\mathbf{x},\mathbf{y}) = K_0(|\mathbf{x}-\mathbf{y}|)$ satisfies assumptions \eqref{A1}-\eqref{A2}.
Then, the map $b\in (0,1) \mapsto  \mathbf{\Delta}_{\infty,b}$  admits at most a countable number of zeros.
\end{proposition}

\begin{proof}[Proof of Proposition \ref{prop:limit_zeros}]
First, recall that
\begin{align}\label{eq:A-B-infty}
  b\in(0,1)\mapsto \mathbf{\Delta}_{\infty,b}=\left( \int_0^{\infty} \Psi_b(x)\frac{\dd \mu(x)}{x}\right)^2
  = \left(\lambda_{1,b}+\lambda_{1,1} -\tfrac{1}{b}\widetilde{\lambda}_{1,b}-b\widetilde{\lambda}_{1,b}\right)^2.
\end{align}
From Lemma \ref{lem:analytic_lambda}, we infer that $b\in (0,1)\mapsto \mathbf{\Delta}_{\infty,b}$
is real analytic. On the other hand,
by Lemma \ref{lem:kernel_function_limit}-(iii), we see that the function
$x\in(0,\infty)\mapsto \Psi_{0.5}(x)$ is strictly positive. Hence, $\mathbf{\Delta}_{\infty,0.5}>0.$
Applying the principle of isolated zeros, we deduce that
the map $b\in (0,1)\mapsto \mathbf{\Delta}_{\infty,b}$ admits at most a countable set of zeros.
\end{proof}

The following result concerns the function $x\in (0,\infty)\mapsto \phi_1(x) - b\,\phi_{1,b}(x)$
which appears in the definition of $V_b^{2}[\mathbf{0}]$ through \eqref{eq:V-0}.
This result will later be used in Section \ref{subsec:V_b-sign} to examine the sign \mbox{of $V_b^{2}[\mathbf{0}]$.}
\begin{lemma}\label{lem:B_infinity}
  There exists $b^{\ast}> \frac{1}{3}$ such that for all  $b\in (0,b^{\ast})$,
  the function $x\mapsto \phi_1(x) - b\,\phi_{1,b}(x)$ is positive in $(0,\infty)$.
\end{lemma}

\begin{remark}\label{rem:B_infinity}
Based on numerical experiments (see Figure \ref{fig:b-phi-1-b}), one can conjecture that the function
$b\in (0,1)\mapsto b\,\phi_{1,b}(x)$ increases in $b$,
which implies that $b^{\ast}=1$. However, it is challenging to justify this rigorously, as there exist some
$x>0$ and $y\in (-1,1)$ such that the integrand of
$b\,\phi_{1,b}(x)$, given by $2b^2xe^{-x\sqrt{1+b^2-2by}} \frac{\sqrt{1-y^2}}{\sqrt{1+b^2-2by}}$ $($see \eqref{eq:phi-1b}$)$,  does not increase with respect to $b$.
\end{remark}

\begin{proof}[Proof of Lemma \ref{lem:B_infinity}]
Taking advantage of \eqref{eq:phi_1_b_upb} and \eqref{eq:est-phi-1-b}, it follows that
\begin{align*}
  \phi_1(x)-b\,\phi_{1,b}(x)\geqslant \frac{4x}{1+x^2}\Big(\frac{2}{5} - b^2(1+x^2)e^{-(1-b)x}\Big).
\end{align*}
Since the function $b\mapsto F_b(x)\triangleq b^2(1+x^2)e^{-(1-b)x}$ is increasing, then for any $b^\ast\in(0,1)$ we get
$$
\forall b\in (0,b^\ast),\forall x>0,\quad F_b(x)<F_{b^\ast}(x).$$
We will explore the existence of $b_{\ast}>0$
such that $F_{b^\ast}(x) < \frac{2}{5}$ for any $x>0$. From straightforward computations we get
\begin{align*}
  \frac{\dd F_b(x)}{\dd x}=b^2(1-b)e^{-(1-b)x} \bigg(- \Big(x-\frac{1}{1-b}\Big)^2+\frac{1}{(1-b)^2}-1 \bigg),
\end{align*}
and thus $x\mapsto F_b(x)$ reaches its absolute maximum at  $x=0$ or $x=\frac{1}{1-b}+\sqrt{(\frac{1}{1-b})^2-1}$.
For $b=\frac{1}{3}$, we find
\begin{align*}
  F_{\frac{1}{3}}(0)=\frac{1}{9}<\frac{2}{5},\quad F_{\frac{1}{3}}\Big(\frac{3+\sqrt{5}}{2}\Big)\approx 0.1524<\frac{2}{5}.
\end{align*}
This achieves the proof.
\end{proof}

\begin{figure}[htbp]
  \centering
  \includegraphics[scale=0.7]{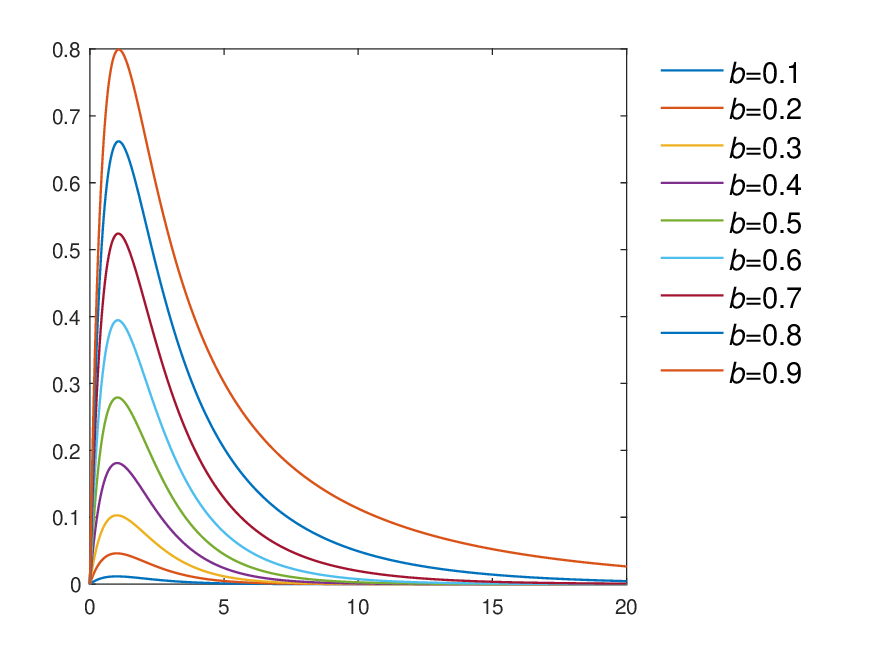} 
  \caption{The graph of $x\mapsto b\phi_{1,b}(x)$ for different values of $b$.}\label{fig:b-phi-1-b}
\end{figure}

\subsection{Spectrum distribution}
The main goal is to analyze the dispersion equation \eqref{eq:Omega-eq}, whose complex solutions correspond to points where the kernel of the linearized operator is nontrivial. By a slight abuse of terminology, we refer to these points as the nonlinear spectrum. The solutions to the quadratic polynomial  \eqref{eq:Omega-eq} are given by
\begin{align}\label{eq:Omega_n_pm}
  \Omega^+_{n,b} \triangleq   \frac{A_{n,b} + B_{n,b}}{2} + \frac{\sqrt{ \mathbf{\Delta}_{n,b}}}{2},\quad
  \Omega^-_{n,b} \triangleq  \frac{A_{n,b} + B_{n,b}}{2}- \frac{\sqrt{ \mathbf{\Delta}_{n,b}}}{2}.
\end{align}
Recall that the discriminant
$\mathbf{\Delta}_{n,b} = (A_{n,b}-B_{n,b})^2- 4 (\widetilde{\lambda}_{n,b}+\widetilde{p}_{n,b})^2$.
\begin{proposition}\label{prop:spectral}
Assume that the kernel function $K$ satisfies the assumptions \eqref{A1}-\eqref{A4}.
Let $b\in \mathcal{S}_{\mathrm{max}}$ such that condition \eqref{eq:assum-1} holds.
There exists a sufficiently large $N\in \mathbb{N}^\star$ depending on $b$ such that the sequences
$(\Omega^+_{n,b})_{n\geqslant N}$ and $(\Omega^-_{n,b})_{n\geqslant N}$ are strictly monotonic with respect to $n$.
Moreover, the following statements are true.
	
(\textbf{i}) If $ V_b^1[\mathbf{0}] > V_b^2[\mathbf{0}]$, then the sequence $(\Omega^+_{n,b})_{n\geqslant N}$ is increasing to $-V_b^2[\mathbf{0}]$ and
$(\Omega^-_{n,b})_{n\geqslant N}$
is decreasing to $- V_b^1[\mathbf{0}]$ with
\begin{equation*}
  - V_b^1[\mathbf{0}] < \Omega^-_{n,b}\leqslant \Omega^-_{N,b}<\Omega^+_{N,b}\leqslant \Omega^+_{n,b} < - V_b^2[\mathbf{0}].
\end{equation*}
	
(\textbf{ii}) If $V_b^1[\mathbf{0}]<V_b^2[\mathbf{0}]$, then $(\Omega_{n,b}^+)_{n\geqslant N}$ is decreasing to $-V_b^1[\mathbf{0}]$
and $(\Omega^-_{n,b})_{n\geqslant N}$ is increasing to $-V_b^2[\mathbf{0}]$ with
\begin{align*}
  \Omega_{N,b}^- \leqslant \Omega^-_{n,b} < -V_b^2[\mathbf{0}] < -V_b^1[\mathbf{0}] < \Omega^+_{n,b} \leqslant \Omega^+_{N,b}.
\end{align*}
\end{proposition}

\begin{remark}\label{remark:low_frequency}
To construct two families of doubly-connected V-states with
$m$-fold symmetry using Crandall-Rabinowitz's theorem, the spectral assumptions must satisfy the following conditions.
\begin{align}\label{eq:m-fold-condition}
  \mathbf{\Delta}_{nm,b}>0,\quad \forall n\geqslant 1,
\end{align}
\begin{align}\label{eq:m-fold-condition2}
  \Omega^{\pm}_{lm}\ne \Omega^{\pm}_{km},
  \quad \Omega^{\pm}_{lm}\ne \Omega^{\mp}_{km},
  \quad \forall k,\,l\in \mathbb{N}^\star\cup\{\infty\}.
\end{align}
Proposition \ref{prop:spectral} ensures that under assumption \eqref{eq:assum-1}, the conditions \eqref{eq:m-fold-condition}-\eqref{eq:m-fold-condition2}
hold at least for sufficiently large $m$. Notably, even in the case of the 2D planar Euler equations,
the validity of \eqref{eq:m-fold-condition}-\eqref{eq:m-fold-condition2} requires
$m$ to exceed a certain threshold, see \cite[Theorem B]{FHMV16}.

\end{remark}

\begin{proof}[Proof of Proposition \ref{prop:spectral}]
(\textbf{i}) We rewrite $\Omega^\pm_{n,b}$ in \eqref{eq:Omega_n_pm} as
\begin{align*}
  \Omega^+_{n,b} = B_{n,b}+\frac{\sqrt{(A_{n,b}-B_{n,b})^2-4(\widetilde{\lambda}_{n,b}+\widetilde{p}_{n,b})^2}-(B_{n,b}-A_{n,b})}{2},
\end{align*}
and
\begin{align*}
  \Omega^-_{n,b}=A_{n,b}-\frac{\sqrt{(A_{n,b}-B_{n,b})^2-4(\widetilde{\lambda}_{n,b}+\widetilde{p}_{n,b})^2}-(B_{n,b}-A_{n,b})}{2}.
\end{align*}
Since $V_b^1[\mathbf{0}] > V_b^2[\mathbf{0}]$, in view of \eqref{eq:Anb-conver}-\eqref{eq:Bnb-conver},
we have $\lim\limits_{n\rightarrow \infty} A_{n,b} < \lim\limits_{n\rightarrow \infty} B_{n,b}$.
Thus, one can choose a sufficiently large number $N\in\mathbb{N}^\star$ such that for any $n\geqslant  N$,
$\mathbf{\Delta}_{n,b}>0$, and
\begin{align*}
  B_{n,b} - A_{n,b}> \frac{V_b^1[\mathbf{0}]-V_b^2[\mathbf{0}]}{2}.
\end{align*}
Applying the mean value theorem to the function $x\mapsto \sqrt{(A_{n,b}-B_{n,b})^2-x^2}$,
together with Proposition \ref{prop:lambda_property}, we infer that for some $\delta\in (0,1)$,
\begin{equation*}
\begin{split}
  \Big|\sqrt{(A_{n,b}-B_{n,b})^2-4 (\widetilde{\lambda}_{n,b}+\widetilde{p}_{n,b})^2}-(B_{n,b}-A_{n,b})\Big|
  & = \frac{4\delta (\widetilde{\lambda}_{n,b}
  +\widetilde{p}_{n,b})^2}{\sqrt{(A_{n,b}-B_{n,b})^2 - 4 \delta^2 (\widetilde{\lambda}_{n,b}+\widetilde{p}_{n,b})^2}} \\
  & \leqslant  \frac{C}{n^4}.
\end{split}
\end{equation*}
Hence, by using \eqref{eq:def-A-n}-\eqref{eq:def-B-n} and \eqref{eq:est-upper3}, we get that
\begin{align*}
  \Omega^+_{n+1,b}-\Omega^+_{n,b} \geqslant B_{n+1,b}-B_{n,b}- \frac{C}{n^{4}}
  \geqslant\lambda_{n,1} - \lambda_{n+1,1} - \frac{C}{n^{4}},
\end{align*}
and
\begin{align*}
  \Omega^-_{n,b}-\Omega_{n+1,b}^{-}  \geqslant A_{n,b}-A_{n+1,b}-\frac{C}{n^{4}}
  \geqslant \lambda_{n,b}-\lambda_{n+1,b}-\frac{C}{n^{4}}.
\end{align*}
In addition, owing to \eqref{eq:est-upper}, we have the crucial lower bound that for every $0<b \leqslant 1$,
\begin{align*}
  \lambda_{n,b}-\lambda_{n+1,b}\geqslant  \frac{1}{C n^3}.
\end{align*}
Choosing $N\in \mathbb{N}^\star$ large enough, we conclude that
\begin{align*}
  \Omega^+_{n+1,b}-\Omega^+_{n,b} \geqslant \frac{1}{C n^3},\quad
  \Omega^-_{n,b} -\Omega^{-}_{n+1}\geqslant  \frac{1}{C n^3}.
\end{align*}
(\textbf{ii})
If $V_b^1[\mathbf{0}] < V_b^2[\mathbf{0}]$, then we infer from \eqref{eq:Anb-conver}-\eqref{eq:Bnb-conver},
that $\lim\limits_{n\rightarrow \infty} A_{n,b} > \lim\limits_{n\rightarrow \infty} B_{n,b}$.
Now, we transform   $\Omega^\pm_{n,b}$ into
\begin{align*}
  \Omega^+_{n,b} = A_{n,b}+\frac{\sqrt{(A_{n,b}-B_{n,b})^2-4(\widetilde{\lambda}_{n,b}+\widetilde{p}_{n,b})^2}-(A_{n,b}-B_{n,b})}{2},
\end{align*}
and
\begin{align*}
  \Omega^-_{n,b} = B_{n,b}-\frac{\sqrt{(A_{n,b}-B_{n,b})^2-4(\widetilde{\lambda}_{n,b}
  +\widetilde{p}_{n,b})^2}-(A_{n,b}-B_{n,b})}{2}.
\end{align*}
Hence, to get the desired result, we implement similar arguments as before, and thus we omit the details.
\end{proof}

\section{Proof of main theorem}\label{sec:proof-thm1}
For $m,k\in \mathbb{N}^{\star}$ and $\alpha\in (0,1)$, we define the following function spaces
\begin{align*}
  \mathbf{C}^{k-\alpha}_{m,e}\triangleq& \Big\{f\in  C^{k-\alpha}(\TT): f(\theta)
  = \sum_{n\geqslant  1}b_n\cos(nm\theta),\,b_n\in \RR,\,\theta\in \TT\Big\},\\
  \mathbf{C}^{k-\alpha}_{m,o}\triangleq& \Big\{f\in  C^{k-\alpha}(\TT): f(\theta)
  = \sum_{n\geqslant  1}b_n\sin(nm\theta),\,b_n\in \RR,\,\theta\in \TT\Big\},
\end{align*}
and
\begin{align}\label{eq:function-space}
  \mathbf{X}_m \triangleq \mathbf{C}_{m,e}^{2-\alpha}\times \mathbf{C}_{m,e}^{2-\alpha},\quad
  \mathbf{Y}_m \triangleq \mathbf{C}_{m,o}^{1-\alpha}\times \mathbf{C}_{m,o}^{1-\alpha}.
\end{align}
\noindent
For $\varepsilon_0>0$, we denote by $\mathbf{B}_{\varepsilon_0}$ the open ball of $\mathbf{X}_m$
centered at $\mathbf{0}$ and of radius $\varepsilon_0$, that is,
\begin{align*}
  \mathbf{B}_{\varepsilon_0}\triangleq \Big\{\mathbf{f}=(f_1,f_2)\in \mathbf{X}_m: \lVert \mathbf{f}\rVert_{\mathbf{X}_m}<\varepsilon_0\Big\}.
\end{align*}

\begin{proposition}\label{prop:regularity}
  Let $m\in\mathbb{N}^\star$, $\alpha\in (0,1)$ and $\mathbf{X}_m$ and $\mathbf{Y}_m$
be the function spaces given by \eqref{eq:function-space}.
Let $F(\Omega,\mathbf{r})$ be the operator introduced in \eqref{eq:vortex-patch}.
Then there exists $\varepsilon_0>0$ small enough such that the following statements hold.
\begin{enumerate}
\item $F:\RR\times\mathbf{B}_{\varepsilon_0}\mapsto \mathbf{Y}_m$ is well-defined.
\item $F:\RR\times \mathbf{B}_{\varepsilon_0}\mapsto \mathbf{Y}_m$ is of class $C^1$.
\item
The partial derivative
$\partial_{\Omega}\partial_rF:\RR\times \mathbf{B}_{\varepsilon_0}\mapsto \mathcal{L}(\mathbf{X}_m,\mathbf{Y}_m)$
exists and is continuous.
\end{enumerate}
\end{proposition}

\begin{proof}[Proof of Proposition \ref{prop:regularity}]
\textbf{(i)} Recall that
$F(\Omega,\mathbf{r})=\Omega \mathbf{r}'+\partial_{\theta}F_0[\mathbf{r}]$.
Then, by the definition of \eqref{eq:function-space},
we only need to check $\partial_{\theta}F_0[\mathbf{r}](\theta)\in \mathbf{Y}_m$.
From \eqref{def:F0-r}, the elements in $\partial_{\theta}F_0[\mathbf{r}]$ can be written as
\begin{align*}
  \partial_\theta \bigg(\int_0^{2\pi}\int_{R_1(\eta)}^{R_2(\eta)}
  K(R_j(\theta)e^{i\theta}, \rho e^{i\eta})\rho \dd \rho \dd \eta \bigg)
  = \big(F_{00}[\mathbf{r}](R_j)+F_{01}[\mathbf{r}](R_j)\big)
  \cdot \partial_\theta \big(R_j(\theta)e^{i\theta}\big), \quad j=1,2,
\end{align*}
where
\begin{align*}
  F_{00}[\mathbf{r}](R_j) \triangleq \int_0^{2\pi}\int_{R_1(\eta)}^{R_2(\eta)} \nabla_{\mathbf{x}}
  K_0(|R_j(\theta)e^{i\theta}-\rho e^{i\eta}|)\rho \dd \rho \dd \eta,
\end{align*}
and
\begin{align*}
	F_{01}[\mathbf{r}](R_j) \triangleq \int_0^{2\pi}\int_{R_1(\eta)}^{R_2(\eta)} \nabla_{\mathbf{x}}
	K_1(R_j(\theta)e^{i\theta}, \rho e^{i\eta})\rho \dd \rho \dd \eta.
\end{align*}
Since $\partial_{\theta}(R_j(\theta)e^{i\theta})\in C^{1-\alpha}(\TT)$,
in order to check $\partial_\theta F_0[\mathbf{r}](\theta) \in C^{1-\alpha}(\TT)\times C^{1-\alpha}(\TT)$,
we only need to show that
\begin{align*}
  \theta\mapsto F_{00}[\mathbf{r}](R_j)\in C^{1-\alpha}(\TT),\quad
  \textrm{and}\quad \theta\mapsto F_{01}[\mathbf{r}](R_j)\in C^{1-\alpha}(\TT), \quad j=1,2.
\end{align*}
Due to \eqref{A4}, the regularity of $F_{01}[\mathbf{r}](R_j)$ is rather clear:
\begin{align*}
  \|F_{01}[r](R_j)\|_{C^1(\TT)} \leqslant
  C \sup_{\mathbf{x},\mathbf{y}\in D_{b,\epsilon_0}} \Big(|\nabla_{\mathbf{x}}K_1(\mathbf{x},\mathbf{y})|+
  |\nabla^2_{\mathbf{x}} K_1(\mathbf{x},\mathbf{y}) \Big)
  \leqslant C,
\end{align*}
where $D_{b,\epsilon_0}\triangleq B(\mathbf{0}, 1+\sqrt{2\varepsilon_0})\setminus B(\mathbf{0}, b-\sqrt{2\varepsilon_0})$ is compactly embedded in $\mathbf{D}$ by choosing $\varepsilon_0$ small enough.
We move to the term $F_{00}[\mathbf{r}](R_j)$.
By using the Gauss-Green formula, we have that, for $j=1,2$,
\begin{equation}\label{eq:Gauss-Green}
\begin{aligned}
  F_{00}[\mathbf{r}](R_j)
  & = -\int_{D}\nabla_\mathbf{y}\Big(K(|R_j(\theta)
  e^{i\theta}-\mathbf{y}|)\Big)\dd \mathbf{y} \\
  & = -\int_{\TT}K_0(|R_j(\theta)e^{i\theta}- R_2(\eta)e^{i\eta}|)(-i\partial_{\eta}(R_2(\eta)e^{i\eta})) \dd \eta\\
  &\quad - \int_{\TT}K_0(|R_j(\theta)e^{i\theta}- R_1(\eta)e^{i\eta}|)(i\partial_{\eta}(R_1(\eta)e^{i\eta})) \dd \eta.
\end{aligned}
\end{equation}
We only need to deal with the following case:
\begin{align}\label{eq:new-term-i}
 \theta\in\mathbb{T}\mapsto \int_{\TT} K_0(|R_2(\theta)e^{i\theta}- R_1(\eta)e^{i\eta}|) \big(i\partial_\eta (R_j(\eta)e^{i\eta})\big)
  \dd \eta\in C^{1-\alpha}(\TT),
\end{align}
since the other cases have been carefully checked in the work of simply connected V-states (\cite[Prop. 4.1]{HXX23b}).
Indeed, it follows from \eqref{A1} and \eqref{eq:integrablity_mu} that
the function $t\mapsto |K_0^{(n)}(t)|$ with $n\geqslant 1$ is decreasing on $(0,\infty)$ and
\begin{align}\label{eq:est-mono-1}
  |K_0^{(n)}(t)| & \leqslant   \int_0^\infty x^{n-1} e^{-t x} \dd \mu(x) \nonumber \\
  & \leqslant \int_0^1 \dd \mu(x) + \frac{C_n}{t^{n+1}} \int_1^\infty \frac{1}{x^2}\dd \mu(x) \nonumber  \\
  & \leqslant 2 \Big(1 + \frac{C_n}{t^{n+1}} \Big) \int_0^\infty \frac{1}{1+x^2}\dd \mu(x) <\infty ,\quad
  \forall t>0, n\in \mathbb{N}^\star.
\end{align}
For $\mathbf{r} = (r_1,r_2) \in \mathbf{B}_{\varepsilon_0}$ with  sufficiently small $\varepsilon_0$,
there exists a small constant $c_{\ast}>0$ such that
\begin{align}\label{eq:lower_bound_distance}
  |R_2(\theta)e^{i\theta}- R_1(\eta)e^{i\eta}| = |\sqrt{1+2 r_2(\theta)}e^{i\theta} - \sqrt{b^2 + 2 r_1(\eta)}e^{i\eta}|
  \geqslant  c_{\ast},\quad \forall \theta,\eta\in \TT.
\end{align}
In view of \eqref{eq:est-mono-1} and $\mathbf{r}(\theta)\in C^{2-\alpha}(\TT)$, we obtain that for any fixed $\eta\in\mathbb{T}$,
\begin{align*}
  \partial_\theta \Big(K_0(|R_2(\theta)e^{i\theta}- R_1(\eta)e^{i\eta}|)\Big)
  & = K'_0(|R_2(\theta)e^{i\theta}-R_1(\eta)e^{i\eta}|)
  \frac{(R_2(\theta)e^{i\theta}-R_1(\eta)e^{i\eta})\cdot
  \partial_{\theta}(R_2(\theta)e^{i\theta})}{|R_2(\theta)e^{i\theta}-R_1(\eta)e^{i\eta}|} \\
  & \in C^{1-\alpha}(\TT),
\end{align*}
and thus \eqref{eq:new-term-i} follows as a direct consequence. Finally, concerning the symmetry of $F(\Omega,\mathbf{r})$, namely,
\begin{align*}
  F(\Omega,\mathbf{r})(\theta)=-F(\Omega,\mathbf{r})(-\theta),\quad
  F(\Omega,\mathbf{r})(\theta+\tfrac{2\pi}{m})=F(\Omega,\mathbf{r})(\theta+\tfrac{2\pi}{m}),
\end{align*}
it follows from the symmetry assumptions in \eqref{A4}:
\begin{align*}
  K(\overline{\mathbf{x}},
  \overline{\mathbf{y}})=K(\mathbf{x},\mathbf{y}),
  \quad K(e^{i\theta}\mathbf{x},e^{i\theta}\mathbf{y}) = K(\mathbf{x},\mathbf{y}).
\end{align*}
We refer to \cite[p. 27]{HXX23a} for more details.
\vskip1mm

\noindent \textbf{(ii)} Noting that $\partial_{\Omega}F(\Omega,\mathbf{r})=\mathbf{r}'$,
it is immediate to see that
$(\Omega,\mathbf{r})\in \RR\times\mathbf{B}_{\varepsilon_0}\mapsto  \partial_{\Omega}F(\Omega,\mathbf{r})$
is a continuous mapping.
Next, we consider the continuity of the map $(\Omega,\mathbf{r})\in \RR\times\mathbf{B}_{\varepsilon_0}\mapsto\partial_{\mathbf{r}}F(\Omega,\mathbf{r})$.
According to \eqref{eq:linear-main}-\eqref{def:eq-Vb2},
we find that the most singular and difficult terms in \eqref{eq:linear-main}
have already been discussed in \cite[Prop. 4.1]{HXX23b}.
Actually, we proved in \cite[Prop. 4.1]{HXX23b} that
\begin{align*}
  \mathbf{r}\in \mathbf{B}_{\varepsilon_0}\mapsto \partial_{\theta}
  \bigg( & \int_{\TT}K_0(|R_j(\theta)e^{i\theta}-R_j(\eta)e^{i\eta}|)
  (i\partial_{\eta}(R_j(\eta)e^{i\eta}))\cdot \tfrac{h_j(\theta)e^{i\theta}}{R_j(\theta)} \dd \eta \\
  & + \int_{\TT}K_0(|R_j(\theta)e^{i\theta}-R_j(\eta)e^{i\eta}|)h_j(\eta)\dd \eta\bigg), \quad j=1,2,
\end{align*}
is continuous. In light of the assumption \eqref{A3} and \eqref{eq:est-mono-1}-\eqref{eq:lower_bound_distance},
the other terms are much easier, and thus we omit the details.
\vskip1mm

\noindent \textbf{(iii)}  In view of
$\partial_{\Omega}\partial_{\mathbf{r}}F(\Omega,\mathbf{r})\mathbf{h}(\theta)=\mathbf{h}'(\theta)$,
the proof is obvious.
\end{proof}

Next, we verify the spectrum assumptions in Crandall-Rabinowitz's theorem (see Theorem \ref{thm:C-R}).
\begin{proposition}\label{prop:CR_conditions}
Let  $K(\mathbf{x},\mathbf{y})= K_0(|\mathbf{x}-\mathbf{y}|) + K_1(\mathbf{x},\mathbf{y})$
satisfy the assumptions \eqref{A1}-\eqref{A4}.
Suppose that $b\in \mathcal{S}_{\mathrm{max}}$ and $m\in \mathbb{N}^\star$ are such that the properties
\eqref{eq:assum-1} and \eqref{eq:m-fold-condition}-\eqref{eq:m-fold-condition2} hold.
Then the following statements are true.
\begin{enumerate}
\item
The kernel of $\partial_{\mathbf{r}}F(\Omega,\mathbf{0}):\mathbf{X}_m\mapsto \mathbf{Y}_{m}$
is nontrivial if and only if $\Omega=\Omega^\pm_{lm,b}$ for some $l\in \mathbb{N}^\star$,
where $\Omega^\pm_{lm,b}$ is given by \eqref{eq:Omega_n_pm}.
In this case, the kernel of $\partial_{\mathbf{r}}F(\Omega^{\pm}_{m,b},\mathbf{0})$
is one dimensional and is generated by
\begin{align*}
  \mathbf{h}_{m,b}:\theta \mapsto
  \begin{pmatrix}
	-\widetilde{\lambda}_{m,b} - \widetilde{p}_{m,b} \\
	\Omega^{\pm}_{m,b} - A_{m,b}
   \end{pmatrix}\cos (m\theta).
\end{align*}
\item
The range of $\partial_{\mathbf{r}}F(\Omega^\pm_{m,b},\mathbf{0})$ is closed and is of co-dimension one.
\item Transversality condition:
\begin{align*}
  \partial_{\Omega}\partial_{\mathbf{r}}F(\Omega^\pm_{m,b},\mathbf{0})(\mathbf{h}_{m,b})
  \notin R(\partial_{\mathbf{r}}F(\Omega^\pm_{m,b},\mathbf{0})).
\end{align*}
\end{enumerate}
\end{proposition}


\begin{proof}[Proof of Proposition \ref{prop:CR_conditions}]
\textbf{(i)}
Consider
\begin{align*}
  \theta\in\mathbb{T} \mapsto \mathbf{h}(\theta)=\sum_{n\geqslant 1}
  \begin{pmatrix}
	h_{1,n}\\
	h_{2,n}
  \end{pmatrix} \cos(n m \theta) \in \mathbf{X}_m,
\end{align*}
then, according to \eqref{eq:linear-0}, \eqref{eq:Lambda-n-b}, \eqref{eq:V-0}, we have
\begin{equation}\label{eq:par-rF-h}
\begin{split}
  \partial_{\mathbf{r}}F(\Omega,\mathbf{0})\mathbf{h}(\theta)
  & = - \sum_{n\geqslant 1} \,n m\Big(\Omega\, \mathbf{Id} + V_b[\mathbf{0}] + \mathbf{\Lambda}_{nm,b}\Big)
  \begin{pmatrix}
	h_{1,n}\\
	h_{2,n}
  \end{pmatrix} \sin(nm \theta) \\
  & = - \sum_{n\geqslant 1} \,n m \mathbf{Q}_{nm,b}(\Omega)
  \begin{pmatrix}
	h_{1,n}\\
	h_{2,n}
  \end{pmatrix} \sin(nm \theta),
\end{split}
\end{equation}
where we have denoted
\begin{align*}
  \mathbf{Q}_{n,b}(\Omega)\triangleq \Omega\, \mathbf{Id} + V_b[\mathbf{0}] + \mathbf{\Lambda}_{n,b}
  =\begin{pmatrix}
	\Omega - A_{n,b} & \widetilde{\lambda}_{n,b}+\widetilde{p}_{n,b}\\
	-\widetilde{\lambda}_{n,b} - \widetilde{p}_{n,b} & \Omega-B_{n,b}
  \end{pmatrix}.
\end{align*}
By analyzing $\mathrm{det}\, \mathbf{Q}_{n,b}(\Omega) = 0$ as in \eqref{eq:Omega-eq},
and under the assumptions
\eqref{eq:m-fold-condition}-\eqref{eq:m-fold-condition2},
we find that the first statement of \textrm{(i)} holds.

Since the discriminant
$\mathbf{\Delta}_{m,b}=(A_{m,b}-B_{m,b})^2 - 4(\widetilde{\lambda}_{m,b} + \widetilde{p}_{m,b})^2 > 0$,
we claim that the matrix
\begin{align*}
  \mathbf{Q}_{m,b}\big(\Omega^\pm_{m,b}\big) =
  \begin{pmatrix}
    \Omega_{m,b}^\pm - A_{m,b} & \widetilde{\lambda}_{m,b} + \widetilde{p}_{m,b} \\
    - \widetilde{\lambda}_{m,b} - \widetilde{p}_{m,b} & \Omega_{m,b}^\pm - B_{m,b}
  \end{pmatrix}
  \ne
  \begin{pmatrix}
	0 & 0\\
	0 & 0
  \end{pmatrix}.
\end{align*}
In fact, if $\widetilde{\lambda}_{m,b} + \widetilde{p}_{m,b}\ne 0$, there is nothing to prove;
otherwise, the condition $\mathbf{\Delta}_{m,b}> 0$ gives  $A_{m,b}\ne B_{m,b}$,
which also implies that the above matrix is not zero matrix.
Here, without loss of generality, we assume that
\begin{align}\label{eq:nonzero_kernel_assumpation}
  \begin{pmatrix}
	- \widetilde{\lambda}_{m,b} - \widetilde{p}_{m,b} \\
	\Omega^\pm_{m,b} - A_{m,b}
  \end{pmatrix}\ne
  \begin{pmatrix}
	0 \\
	0
  \end{pmatrix},
\end{align}
and this vector obviously solves that $\mathbf{Q}_{m,b}\big(\Omega^\pm_{m,b}\big)
\begin{pmatrix}
h_{1,1} \\
h_{2,1}
\end{pmatrix} = \mathbf{0}$; consequently, the second statement of \textrm{(i)} follows.
In particular, for the case $K_1(\mathbf{x},\mathbf{y}) \equiv 0$,
we have $\widetilde{p}_{m,b}=0$ and $\widetilde{\lambda}_{m,b}>0$ (by \eqref{eq:est-upper2}),
thus \eqref{eq:nonzero_kernel_assumpation} obviously holds.
\vskip1mm

\noindent\textbf{(ii)}
First, we define the function space $\mathbf{Z}_{m}\subset \mathbf{Y}_{m}$ such that any element
$\mathbf{g}\in \mathbf{Z}_{m}$ satisfies that
\begin{align}\label{eq:range-linear}
  \mathbf{g}(\theta) = \sum_{n\geqslant  1}
  \begin{pmatrix}
	g_{1,n}\\
	g_{2,n}
   \end{pmatrix}
   \sin (nm\theta),
\end{align}
where $g_{1,n},g_{2,n}\in \RR$ for all $n\in \mathbb{N}^\star$ and there exists $(a_{1},a_{2})\in \RR^2$ such that
\begin{align*}
  \mathbf{Q}_{m,b}\big(\Omega^\pm_{m,b}\big)
  \begin{pmatrix}
	a_{1}\\
	a_{2}
   \end{pmatrix}
   =\begin{pmatrix}
	g_{1,1}\\
	g_{2,1}
   \end{pmatrix}.
\end{align*}
We can see that $\mathbf{Z}_m$ is closed and of codimension $1$ in $\mathbf{Y}_{m}$.
By our construction of $\mathbf{Z}_m$, the range of $\partial_{\mathbf{r}}F(\Omega^\pm_{m,b},\mathbf{0})$
is included in the space $\mathbf{Z}_m$. Therefore, it remains to check that
$\mathbf{Z}_m\subset R\big(\partial_{\mathbf{r}}F(\Omega^{\pm}_{m},\mathbf{0})\big)$.
	
Now let $\mathbf{g}\in \mathbf{Z}_m$, we shall try to find a solution $\mathbf{h}\in \mathbf{X}_m$ solving the equation
\begin{align*}
  \partial_{\mathbf{r}}F(\Omega^\pm_{m,b},\mathbf{0})\mathbf{h}=\mathbf{g},
\end{align*}
which, by \eqref{eq:par-rF-h}, is equivalent to
\begin{align}\label{eq:range-eq}
  -nm \mathbf{Q}_{nm,b}\big(\Omega^\pm_{m,b}\big)
  \begin{pmatrix}
	h_{1,n}\\
	h_{2,n}
   \end{pmatrix}
   =\begin{pmatrix}
	 g_{1,n}\\
	 g_{2,n}
	\end{pmatrix},
    \quad \forall n\geqslant  1.
\end{align}
By the definition of $\mathbf{Z}_m$,
the solution to equation \eqref{eq:range-eq} exists.
We shall use the Mikhlin-type multiplier theorem (see Lemma \ref{lem:multiplier-lemma} below)
to show the regularity of $\mathbf{h}$ arguing as the simply connected case studied in \cite{HXX23b}.
In light of \eqref{eq:m-fold-condition}, the matrix $\mathbf{Q}_{nm,b}(\Omega^{\pm}_{m,b})$ is invertible for $n\geqslant 2$.
Indeed, for every integer $n\neq m$, the inverse matrix of
$\mathbf{Q}_{n,b}(\Omega^\pm_{m,b})$ is
\begin{align*}
  \mathbf{Q}_{n,b}^{-1}(\Omega^{\pm}_{m,b}) = \frac{\mathbf{Q}_{n,b}^*
  (\Omega^{\pm}_{m,b})}{\det(\mathbf{Q}_{n,b}(\Omega^{\pm}_{m,b}))},\quad
  \textrm{with}\quad
  \mathbf{Q}_{n,b}^*
  (\Omega^{\pm}_{m,b}) \triangleq \begin{pmatrix}
	\Omega^{\pm}_{m,b}-B_{n,b} & - \widetilde{\lambda}_{n,b} - \widetilde{p}_{n,b} \\
	\widetilde{\lambda}_{n,b} + \widetilde{p}_{n,b} & \Omega^{\pm}_{m,b} - A_{n,b}
   \end{pmatrix}.
\end{align*}
Hence, it follows that for every $n\geqslant 2$,
\begin{align}\label{eq:relation-range}
  \begin{pmatrix}
    h_{1,n}\\
    h_{2,n}
  \end{pmatrix}
  =-\frac{1}{nm} \mathbf{Q}_{nm,b}^{-1}(\Omega^{\pm}_{m,b})
  \begin{pmatrix}
	g_{1,n}\\
	g_{2,n}
   \end{pmatrix}.
\end{align}
To verify that $\mathbf{h}\in \mathbf{X}_m$, it is enough to show that
\begin{align}\label{eq:H-regularity}
  \sum_{n\geqslant  2}
  \begin{pmatrix}
  h_{1,n} \\
  h_{2,n}
  \end{pmatrix}
  \cos (nm\theta)
  = - \sum_{n\geqslant  2}
  \frac{1}{nm} \mathbf{Q}_{nm,b}^{-1}(\Omega_{m,b}^\pm)
  \begin{pmatrix}
  g_{1,n} \\
  g_{2,n}
  \end{pmatrix}
  \cos (nm\theta)\in C^{2-\alpha}(\TT)\times C^{2-\alpha}(\TT).
\end{align}
Denoting by
\begin{align*}
  \widetilde{g}_{j,n} =
  \begin{cases}
	g_{j,\frac{n}{m}},&\quad n\in m\mathbb{N}^\star,\\
	0, &\quad \text{else},
   \end{cases}
   \quad j=1,2,
\end{align*}
we only need to prove that
\begin{align*}
  \sum_{n\geqslant  1,n\neq m}
  \frac{1}{n} \mathbf{Q}_{n,b}^{-1}(\Omega_{m,b}^\pm)
  \begin{pmatrix}
  \widetilde{g}_{1,n} \\
  \widetilde{g}_{2,n}
  \end{pmatrix}
  \cos (n\theta)\in C^{2-\alpha}(\TT)\times C^{2-\alpha}(\TT).
\end{align*}
Due to $\mathbf{g}\in \mathbf{Z}_{m}$, we know that
$|g_{j,n}| \leqslant C n^{-(1-\alpha)}$,
and $\sum\limits_{n\geqslant 1} \frac{|g_{j,n}|}{nm}\leqslant C$, $j=1,2$;
then, integrating $\mathbf{g}(\cdot)$ from $0$ to $\theta$ gives that
\begin{align*}
  \sum_{n\geqslant  1}
  \frac{1}{nm}
  \begin{pmatrix}
  g_{1,n} \\
  g_{2,n}
  \end{pmatrix} \cos (nm\theta) =
  \sum_{n\geqslant  1} \frac{1}{n}
  \begin{pmatrix}
  \widetilde{g}_{1,n} \\
  \widetilde{g}_{2,n}
  \end{pmatrix}
  \cos (n\theta) \in C^{2-\alpha}(\TT)\times C^{2-\alpha}(\TT).
\end{align*}
According to Lemma \ref{lem:multiplier-lemma},
it suffices to check that for any $n\geqslant N_1$ where $N_1\in\mathbb{N}^\star$ could be sufficiently large,
\begin{align}\label{eq:targ-Qnb}
  |\mathbf{Q}_{n,b}^{-1}(\Omega^\pm_{m,b})| \leqslant  C, \quad
  \big| n\big(\mathbf{Q}_{n+1,b}^{-1}(\Omega^\pm_{m,b})
  -\mathbf{Q}_{n,b}^{-1}(\Omega^\pm_{m,b})\big)\big| \leqslant  C,
\end{align}
where we have denoted $|\mathbf{M}| = \sum\limits_{1\leqslant i,j\leqslant 2} |m_{ij}|$ for a matrix
$\mathbf{M} = (m_{ij})_{2\times 2}$.
\vskip0.5mm
\noindent By virtue of \eqref{eq:m-fold-condition2}, it follows that
$\Omega^\pm_{nm,b}\ne-V_b^1[\mathbf{0}]$ and $\Omega^\pm_{nm,b}\ne-V_b^2[\mathbf{0}]$ for every $n\geqslant 1$,
which implies
\begin{align*}
  \lim_{n\to \infty}\det(\mathbf{Q}_{n,b}(\Omega^\pm_{m,b}))
  = \big(\Omega^\pm_{m,b} +V_b^1[\mathbf{0}]\big)
  \big(\Omega^\pm_{m,b}+V_b^2[\mathbf{0}]\big) \ne 0.
\end{align*}
Then we choose $N_1$ sufficiently large such that for any $n\geqslant  N_1$,
\begin{align}\label{eq:det-lbd}
  |\det(\mathbf{Q}_{n,b}(\Omega^\pm_{m,b}))| \geqslant
  \frac{1}{2}\Big|(\Omega^\pm_{m,b} +V_b^1[\mathbf{0}])(\Omega^\pm_{m,b}+V_b^2[\mathbf{0}])\Big| > 0.
\end{align}
Thanks to Proposition \ref{prop:lambda_property}, and recalling
$A_{n,b} = - V_b^1[\mathbf{0}] + \lambda_{n,b} + p_{n,b}$ and
$B_{n,b} = - V_b^2[\mathbf{0}] -\lambda_{n,1} + p_{n,1}$,
we find $|\mathbf{Q}_{n,b}^{-1}(\Omega^\pm_{m,b})|\leqslant  C$
for any $n\geqslant  N_1$.
Next, we consider the difference in \eqref{eq:targ-Qnb}. Note that
\begin{equation*}
\begin{split}
  & \det(\mathbf{Q}_{n,b} (\Omega^\pm_{m,b})) - \det(\mathbf{Q}_{n+1,b}(\Omega^\pm_{m,b})) \\
  & = (\widetilde{\lambda}_{n,b} + \widetilde{p}_{n,b})^2 - (\widetilde{\lambda}_{n+1,b} + \widetilde{p}_{n+1,b})^2
  + (\Omega^\pm_{m,b} - A_{n,b})(B_{n+1,b}-B_{n,b}) \\
  &\quad + (A_{n+1,b}-A_{n,b})(\Omega^\pm_{m,b} - B_{n+1,b}).
\end{split}
\end{equation*}
Taking advantage of Proposition \ref{prop:lambda_property} and \eqref{eq:det-lbd},
we deduce that for $n\geqslant  N_1$,
\begin{align*}
  n|\det(\mathbf{Q}_{n,b}(\Omega^\pm_{m,b})) - \det(\mathbf{Q}_{n+1,b}(\Omega^\pm_{m,b}))|
  & \leqslant  C + C n |B_{n+1,b}-B_{n,b}|
  + C n |A_{n+1,b}-A_{n,b}| \\
  & \leqslant C + C n |\lambda_{n,b} -\lambda_{n+1,b}|
  + C n |\lambda_{n,1} - \lambda_{n+1,1}| \\
  & \leqslant C,
\end{align*}
and thus
\begin{align*}
  \big| n \big(\mathbf{Q}_{n+1,b}^{-1}(\Omega^\pm_{m,b})-\mathbf{Q}_{n,b}^{-1}(\Omega^\pm_{m,b}) \big)\big|
  & \leqslant n\Big|\frac{1}{\mathrm{det}\,(\mathbf{Q}_{n+1,b}(\Omega^\pm_{m,b}) )}
  - \frac{1}{\mathrm{det}\,(\mathbf{Q}_{n,b}(\Omega^\pm_{m,b}) )} \Big| \big|\mathbf{Q}^*_{n+1,b}(\Omega_{m,b}^\pm) \big| \\
  & \quad + \frac{n }{\big|\mathrm{det}\,(\mathbf{Q}_{n,b}(\Omega^\pm_{m,b}) )\big| }
  \big|\mathbf{Q}^*_{n+1,b}(\Omega_{m,b}^\pm) - \mathbf{Q}^*_{n,b}(\Omega_{m,b}^\pm) \big| \\
  & \leqslant  C.
\end{align*}
Therefore, we complete the proof of \eqref{eq:targ-Qnb} and establish the desired result.
\vskip1mm

\noindent \textbf{(iii)} From \eqref{eq:par-rF-h}, we see that
\begin{align*}
  \partial_{\Omega}\partial_{\mathbf{r}}F(\Omega^\pm_{m,b},\mathbf{0})\mathbf{h} (\theta) =
  -\sum_{n\geqslant 1} nm
  \begin{pmatrix}
    h_{1,n} \\
    h_{2,n}
  \end{pmatrix} \sin(nm\theta)
  = \mathbf{h}'(\theta).
\end{align*}
Consequently,
\begin{align*}
  \partial_{\Omega}\partial_{\mathbf{r}}F(\Omega^{\pm}_{m,b},\mathbf{0}) \mathbf{h}_{m,b}(\theta)
  =\mathbf{h}'_{m,b}(\theta) = - m
  \begin{pmatrix}
    - \widetilde{\lambda}_{m,b} - \widetilde{p}_{m,b} \\
	\Omega^\pm_{m,b} - A_{m,b}
  \end{pmatrix}
  \sin(m\theta).
\end{align*}
Next, we intend to show the desired result in \textrm{(iii)} by contradiction.
Suppose that $\mathbf{h}'_{m,b}\in R(\partial_{\mathbf{r}}F(\Omega^\pm_{m,b},\mathbf{0}))$,
namely, there exists some
$\mathbf{g}_{m,b}\in \mathbf{X}_m$ such that
\begin{align}\label{eq:contradiction_assumpation}
  \partial_{\mathbf{r}}F(\Omega^\pm_{m,b},\mathbf{0})\mathbf{g}_{m,b}=\mathbf{h}'_{m,b}.
\end{align}
Since $\Omega_{m,b}^{\pm}$, $V_b[\mathbf{0}]$ and $\mathbf{\Lambda}_{n,b}$ are all real-valued,
it follows from \eqref{eq:linear-0} and Parserval's theorem that
\begin{align}\label{eq:conjugate_operator_F_0}
  (\partial_{\mathbf{r}}F(\Omega^{\pm}_{m,b},\mathbf{0})\mathbf{h}_1,\mathbf{h}_2)_{L^2}
  = - (\mathbf{h}_1,\partial_{\mathbf{r}}F(\Omega^{\pm}_{m,b},\mathbf{0})\mathbf{h}_2)_{L^2},\quad
  \mathbf{h}_1,\mathbf{h}_2\in H^1(\mathbb{T})\times H^1(\mathbb{T}).
\end{align}
In view of \eqref{eq:contradiction_assumpation}, \eqref{eq:conjugate_operator_F_0}
and $\mathbf{g}_{m,b}\in \mathbf{X}_m\subset C^1(\mathbb{T})$, we get that
\begin{equation*}
\begin{aligned}
  \lVert \mathbf{h}'_{m,b}\rVert^2_{L^2}
  & = (\partial_{\mathbf{r}}F(\Omega^\pm_{m,b},\mathbf{0})\mathbf{g}_{m,b},\mathbf{h}'_{m,b}) \\
  & = -(\mathbf{g}_{m,b},\partial_{\mathbf{r}}F(\Omega^\pm_{m,b},\mathbf{0})\mathbf{h}'_{m,b}).
\end{aligned}
\end{equation*}
Furthermore, based on the fact that the linear operator $\partial_{\mathbf{r}}F(\Omega^\pm_{m,b},\mathbf{0})$
is a Fourier multiplier (see \eqref{eq:par-rF-h})
and $\mathbf{h}_{m,b}$ belongs to the kernel of
$\partial_{\mathbf{r}}F(\Omega^\pm_{m,b},\mathbf{0})$, we deduce that
\begin{align*}
  \partial_{\mathbf{r}}F(\Omega^\pm_{m,b},\mathbf{0})\mathbf{h}'_{m,b}
  = \partial_{\theta}\Big(\partial_{\mathbf{r}}F(\Omega^\pm_{m,b},\mathbf{0})\mathbf{h}_{m,b}\Big)=\mathbf{0}.
\end{align*}
Putting together the above two estimates yields
$\lVert \mathbf{h}'_{m,b}\rVert^2_{L^2}=0$,
which clearly contradicts our assumption \eqref{eq:nonzero_kernel_assumpation}.
Hence, we finish the proof of the transversality condition.
\end{proof}

\noindent Now, we give a complete statement for our main results, which directly yields Theorem \ref{thm:general_kernel}.
\begin{theorem}\label{thm:general-statement}
  Suppose that $K(\mathbf{x},\mathbf{y})=K_0(|\mathbf{x}-\mathbf{y}|)+K_1(\mathbf{x},\mathbf{y})$ satisfies assumptions
\eqref{A1}-\eqref{A4}.
\begin{enumerate}
\item Let $b\in \mathcal{S}_{\mathrm{max}}$ (recalling $\mathcal{S}_{\mathrm{max}}$ is given by \eqref{def:S-set}) and $m\in \mathbb{N}^\star$
be such that conditions \eqref{eq:assum-1} and \eqref{eq:m-fold-condition}-\eqref{eq:m-fold-condition2} hold.
Then, there exist a constant $a>0$ and continuous functions
\begin{align*}
  \Omega^\pm :(-a,a)\mapsto \RR,\quad \mathbf{r}^\pm :(-a,a)\mapsto C^{2-\alpha}(\TT)\times C^{2-\alpha}(\TT),
\end{align*}
satisfying $\Omega^\pm(0)=\Omega_{m,b}^\pm$, and $\mathbf{r}^\pm(0) = \mathbf{0}$,
such that $(\Omega^{\pm}(s),\mathbf{r}^{\pm}(s))_{-a<s<a}$
is one-parameter non-trivial solutions of the doubly connected V-states equation \eqref{eq:vortex-patch-eqs}.
Moreover, the mapping
\begin{align*}
  \theta\in \TT\mapsto
  \begin{pmatrix}
    \sqrt{b^2+2r^{\pm}_1(s,\theta)}\\
	\sqrt{1+2r^{\pm}_2(s,\theta)}
  \end{pmatrix}
  e^{i\theta},\quad\textrm{with}\quad \mathbf{r}^{\pm}(s)(\theta)
  = (r_1^\pm(s,\theta),r^\pm_2(s,\theta)),
\end{align*}
parameterizes the boundaries of the $m$-fold symmetric doubly rotating patches with
angular velocities $\Omega^{\pm}(s)$.
\item Let $b\in \mathcal{S}_{\mathrm{max}}$ be such that \eqref{eq:assum-1} holds,
then conditions \eqref{eq:m-fold-condition}-\eqref{eq:m-fold-condition2}
are satisfied at least for sufficiently large $m$.
\item For the special case $K_1(\mathbf{x},\mathbf{y})\equiv 0$,
the condition \eqref{eq:assum-1} holds for any $b\in \mathcal{S}_{\mathrm{max}}$
except for at most a countable set.
\end{enumerate}
\end{theorem}

\begin{proof}[Proof of Theorem \ref{thm:general-statement}]
\textbf{(i)} It follows from applying the Crandall-Rabinowitz theorem
(see Theorem \ref{thm:C-R}) to the equation of doubly connected V states
\eqref{eq:vortex-patch}. In fact, equality \eqref{eq:rankine_vorticity_double}, Propositions \ref{prop:regularity}
and \ref{prop:CR_conditions} have verified all the conditions of Theorem \ref{thm:C-R}. \\
\noindent \textbf{(ii)} It has been proved in Proposition \ref{prop:spectral}. \\
\noindent \textbf{(iii)} It is a direct consequence of Proposition \ref{prop:limit_zeros}.
\end{proof}

\section{Applications to geophysical flows}\label{sec:applications}
In this section, as applications of Theorems \ref{thm:general_kernel} and \ref{thm:general-statement},
we shall explore some special cases of the active scalar equations
\eqref{eq:ASE}-\eqref{eq:stream_function} where the kernel $K$ satisfies assumptions \eqref{A1}-\eqref{A4}.
This includes several important models that describe geophysical flows.
Based on the quantitative version of Theorem \ref{thm:general_kernel}, namely Theorem \ref{thm:general-statement},
through a detailed spectral analysis of these geophysical models one by one, one can obtain explicit conditions on $m$ and $b$,
ensuring the existence of doubly connected V-states, like \cite{FHMV16,HM16b,FHH16,Rou23a,DHHM}.
However, we do not pursue this aspect here; instead, we apply our results to various examples to emphasize the unification of our theorem.
In addition, to illustrate new results, we provide the proofs for Theorem \ref{Cor:rad-domain} and Theorem \ref{Cor:Euler}
within Sections \ref{subsec:gSQG-rad}-\ref{subsec:Euler-ext-disc}. \\

\subsection{2D Euler equation in the whole space}\label{subsec:Euler}
For the 2D Euler equation in the whole plane, it corresponds to the equation
\eqref{eq:ASE}-\eqref{eq:stream_function} with $\mathbf{D}=\mathbb{R}^2$
and $K(\mathbf{x},\mathbf{y}) = K_0(|\mathbf{x}-\mathbf{y}|) = -\frac{1}{2\pi} \log|\mathbf{x}-\mathbf{y}|$.
The existence of nontrivial doubly connected V-states has already been obtained in \cite{FHMV16,HM16b}.
Here, we recover this result based on our Theorem \ref{thm:general-statement}.
\\[0.5mm]
\noindent Clearly, as shown in \cite[Sec. 5]{HXX23b}, $K_0(t) = -\frac{1}{2\pi}\log t$ satisfies conditions \eqref{A1}-\eqref{A2} with the non-negative measure
$\dd \mu(x) = \frac{1}{2\pi} \dd x$ and $\alpha\in (0,1)$.
As a result, we can directly get the existence of $m$-fold symmetric doubly connected V-states bifurcated from the annulus
$\mathbf{1}_{\mathbb{D}\setminus \overline{b\mathbb{D}}}$, as long as $b\in \mathcal{S}$ and $m\in\mathbb{N}^\star$ is sufficiently large.
\\[0.5mm]
\noindent Next, we will show that the set $\mathcal{S}$ defined by \eqref{def:S-set} equals $\mathcal{S}_{\mathrm{max}}=(0,1)$.
By virtue of \eqref{eq:Lambda-n-b}-\eqref{def:lambda-nb} and \eqref{eq:V-0}, we have that
\begin{align}\label{eq:Lamb-V[0]}
  \mathbf{\Lambda}_{n,b}
  =  \begin{pmatrix}
	-\lambda_{n,b} & \widetilde{\lambda}_{n,b} \\
	-\widetilde{\lambda}_{n,b} & \lambda_{n,1}
  \end{pmatrix},
  \quad \textrm{and}\quad
    V_b[\mathbf{0}]=
    \begin{pmatrix}
    V_b^1[\mathbf{0}] & 0\\
    0 & V_b^2[\mathbf{0}]
  \end{pmatrix}
  =
  \begin{pmatrix}
	\lambda_{1,b} -\frac{1}{b} \widetilde{\lambda}_{1,b} &	0\\
	0	&   b \widetilde{\lambda}_{1,b}- \lambda_{1,1}
  \end{pmatrix},
\end{align}
with
\begin{equation*}
\begin{aligned}
  \lambda_{n,b} = -\frac{1}{2\pi}\int_{\mathbb{T}} \log(2b|\sin \tfrac{\eta}{2}|) e^{in\eta}\dd \eta,\quad
  \widetilde{\lambda}_{n,b}  = -\frac{1}{2\pi} \int_{\mathbb{T}}
  \log(|b-e^{i\eta}|) e^{in\eta}\dd \eta.
\end{aligned}
\end{equation*}
It follows from \cite[Eq. 4.397.6]{GR15} and \cite[Eq. 4.384.3]{GR15} that
\begin{equation}\label{eq:lam-tild-Euler}
\begin{split}
  \widetilde{\lambda}_{n,b}
  & = -\frac{1}{2\pi} \int_{-\pi}^{\pi} \log \sqrt{1+b^2-2b\cos \eta} \cos (n\eta)\, \dd \eta\\
  & = -\frac{1}{2\pi} \int_{0}^{\pi} \log (1+b^2-2b\cos \eta) \cos (n\eta)\, \dd \eta
  = \frac{b^n}{2n},
\end{split}
\end{equation}
and
\begin{align}\label{eq:lam-Euler}
  \lambda_{n,b}  = -\frac{1}{2\pi} \int_0^{2\pi} \log(2b \sin\tfrac{\eta}{2}) \cos(n\eta) \dd \eta
  = -\frac{1}{\pi} \int_0^\pi \log(2b \sin \eta) \cos(2n\eta) \dd \eta = \frac{1}{2n}.
\end{align}
Hence, $\mathbf{\Lambda}_{n,b}$ and $V_b[\mathbf{0}]$ in the 2D whole-plane Euler case becomes
\begin{align*}
  \mathbf{\Lambda}_{n,b}=
  \begin{pmatrix}
	-\frac{1}{2n} & \frac{b^n}{2n}\\[1mm]
	-\frac{b^n}{2n} & \frac{1}{2n}
  \end{pmatrix},
  \quad \textrm{and}\quad
  V_b[\mathbf{0}]=
  \begin{pmatrix}
	V_b^1[\mathbf{0}] & 0 \\
	0 & V_b^2[\mathbf{0}]
  \end{pmatrix}
  =
  \begin{pmatrix}
	0 & 0 \\
	0 & \frac{-1+b^2}{2}
  \end{pmatrix}.
\end{align*}
Notice that \eqref{eq:assum-1} holds with
\begin{align*}
  \textrm{$V_b^2[\mathbf{0}] < 0=V_b^1[\mathbf{0}]$,\quad for any $b\in (0,1)$.}
\end{align*}

Finally, we indicate that with the help of spectral study done in \cite{FHMV16}, one can reprove \cite[Theorem B]{FHMV16} by using Theorem \ref{thm:general-statement}.

\subsection{gSQG equation in the whole space}
For the gSQG equation in whole plane, it corresponds to  the equation \eqref{eq:ASE}-\eqref{eq:stream_function}
within $\mathbf{D}=\mathbb{R}^2$, where the convolution kernel is given by
\begin{align*}
  K(\mathbf{x}, \mathbf{y}) = c_{\beta}|\mathbf{x}- \mathbf{y}|^{-\beta},\quad
  c_{\beta}=\tfrac{\Gamma(\beta/2)}{\pi 2^{2-\beta}\Gamma(1-\beta/2)},\quad\beta\in (0,1).
\end{align*}

\noindent Since the assumptions \eqref{A1}-\eqref{A2} in the whole-space gSQG case have already been verified in
\cite[Sec. 5]{HXX23b} with $\dd \mu(x) = \frac{c_\beta}{\Gamma(\beta)} x^\beta \dd x$
and $\alpha\in (0,1-\beta]$,
provided that $b\in \mathcal{S}$ and $m\in\mathbb{N}^\star$ is sufficiently large,
we can directly apply Theorem \ref{thm:general_kernel}  to this case and get the existence of $m$-fold symmetric doubly connected V-states.

\noindent Next, we will show that the set $\mathcal{S} = (0,1)$. Remind that
\begin{equation*}
\begin{aligned}
  \lambda_{n,b} = c_\beta\int_{\mathbb{T}} (2b|\sin \tfrac{\eta}{2}|)^{-\beta} e^{in\eta}\dd \eta,\quad
  \widetilde{\lambda}_{n,b}  = c_\beta \int_{\mathbb{T}}
  |b- e^{i\eta}|^{-\beta} e^{in\eta}\dd \eta.
\end{aligned}
\end{equation*}
Through changing of variables $\tau = e^{i\eta}$, and applying (3.19) and (5.24) in  \cite{FHH16}, we have
\begin{equation}\label{eq:gSQG_identity_spectrum}
\begin{aligned}
  \widetilde{\lambda}_{n,b} = c_\beta\int_0^{2\pi}|be^{i\eta}-1|^{-\beta}e^{in\eta}\dd \eta
  & = - i c_{\beta}\int_{|\tau|=1,\tau\in \CC}|1-b\tau|^{-\beta}\tau^{n-1} \dd \tau \\
  & = 2\pi c_\beta b^n \frac{(\tfrac{\beta}{2})_n}{n!} F \big(\tfrac{\beta}{2},n+\tfrac{\beta}{2};n+1; b^2\big)
  = \Lambda_{n,b}(\beta),
\end{aligned}
\end{equation}
and
\begin{align}\label{eq:gSQG_spectrum2}
  \lambda_{n,b} = b^{-\beta} \widetilde{\lambda}_{n,1} = b^{-\beta} \Lambda_{n,1}(\beta),
\end{align}
where $F(a,b;c;\mathbf{z})$ ($a,b\in\mathbb{R}$, $c\in\mathbb{R}\setminus(-\mathbb{N})$) is the hypergeometric function defined
on the unit disc $\mathbb{D}$:
\begin{align}\label{eq:hypergeom}
  F(a,b;c;\mathbf{z}) = \sum_{n=0}^\infty \frac{(a)_n (b)_n}{(c)_n} \frac{\mathbf{z}^n}{n!},\quad \forall \mathbf{z}\in\mathbb{D},
\end{align}
and $(x)_n$ is the Pochhammer symbol given by $(x)_0 \triangleq 1$ and $(x)_n \triangleq x (x+1)\cdots (x+n) $ for any $n\geqslant 1$.
Therefore, in view of \eqref{eq:Lamb-V[0]} and \eqref{eq:gSQG_identity_spectrum}-\eqref{eq:gSQG_spectrum2}, we get
\begin{align}\label{eq:Vb0-SQG}
  V_b[\mathbf{0}]=
  \begin{pmatrix}
	V_b^1[\mathbf{0}] & 0 \\
	0 & V_b^2[\mathbf{0}]
  \end{pmatrix}
  =
  \begin{pmatrix}
	-\Lambda_{1,b}(\beta) + b^{-\beta}\Lambda_{1,1}(\beta) &	0\\
	0	& -\Lambda_{1,1}(\beta) + b^2 \Lambda_{1,b}(\beta)
  \end{pmatrix},
\end{align}
Notice that, from \cite[Lemma 5.2]{FHH16},
the map $b\in (0,1)\mapsto \Lambda_{n,b}(\beta)$ is strictly increasing, and thus
\begin{align*}
  V_b^2[\mathbf{0}] < 0 < V_b^1[\mathbf{0}],\quad \forall b\in (0,1).
\end{align*}

Furthermore, based on the spectral study in \cite{FHH16}, one can also reprove \cite[Theorem 1.1]{FHH16} from Theorem \ref{thm:general-statement}.

\subsection{QGSW equation in the whole space}\label{subsec:QGSW}
For the whole-space QGSW equation, it is the equation \eqref{eq:ASE}-\eqref{eq:stream_function} with $\mathbf{D}=\mathbb{R}^2$
and the kernel function $K(\mathbf{x},\mathbf{y}) = \frac{1}{2\pi }\mathbf{K}_0(\varepsilon|\mathbf{x}-\mathbf{y}|)$,
where $\mathbf{K}_0(\cdot)$ is the zero-th modified Bessel function.
We here reprove the result of \cite[Theorem 1.1]{Rou23a} for the whole-space QGSW equation,
which states that for any $\varepsilon>0$ and $b\in (0,1)$,
there exist $m_0\in\mathbb{N}^\star$ depending on $b$ and $\varepsilon$
such that for any $m\geqslant m_0$, two curves of $m$-fold doubly connected V-states bifurcating from
$\mathbf{1}_{\mathbb{D}\setminus \overline{b\mathbb{D}}}$ can be constructed.

\noindent Noting that assumptions \eqref{A1}-\eqref{A2} have been verified in \cite[Sec. 5]{HXX23b} for any $\varepsilon>0$
with non-negative measure $\dd \mu(x) = \frac{1}{2\pi} \frac{x}{\sqrt{x^2 -\varepsilon^2}}
\mathbf{1}_{x\geqslant \varepsilon} \dd x$ and $\alpha\in (0,1)$,
we can directly apply Theorem \ref{thm:general_kernel} to this case and show the existence of
$m$-fold symmetric doubly connected V-states
with sufficiently large $m$ and $b\in \mathcal{S}$.
In the following, we show that the set $\mathcal{S}=(0,1)$ in this case.

\noindent In light of \cite[Eq. (3.6)]{Rou23a}, we have
\begin{align*}
  & \widetilde{\lambda}_{n,b}
  = \frac{1}{2\pi}\int_0^{2\pi}\mathbf{K}_0(\varepsilon|1-be^{i\theta}|)\cos (n\theta)\dd \theta
  = \mathbf{I}_n( b \varepsilon)\mathbf{K}_n(\varepsilon), \\
  & \lambda_{n,b}
  = \frac{1}{2\pi}\int_0^{2\pi}\mathbf{K}_0(\varepsilon b |1-e^{i\theta}|)\cos (n\theta)\dd \theta
  = \mathbf{I}_n(b\varepsilon )\mathbf{K}_n(b\varepsilon),
\end{align*}
where $\mathbf{I}_n$ and $\mathbf{K}_n$ are the modified Bessel functions of the first and second kind, defined by
\begin{equation}\label{def:In-z}
  \mathbf{I}_n(\mathbf{z}) =
  \sum_{m=0}^\infty \frac{\left(\frac{\mathbf{z}}{2}\right)^{n+2m}}{m!\Gamma(n+m+1)},\quad|\mbox{arg}(\mathbf{z})|<\pi,
\end{equation}
and
\begin{align}\label{def:Kn-z}
  \mathbf{K}_n(\mathbf{z}) = \lim_{\nu\rightarrow n} \mathbf{K}_\nu(\mathbf{z}),\quad
  \mathbf{K}_\nu(\mathbf{z})=\frac{\pi}{2}\frac{\mathbf{I}_{-\nu}(\mathbf{z})-\mathbf{I}_\nu(\mathbf{z})}
  {\sin(\nu\pi)},\quad\nu\in\mathbb{C}\setminus\mathbb{Z},\quad|\mbox{arg}(\mathbf{z})|<\pi.
\end{align}
Thus, we can use \eqref{eq:Lamb-V[0]} to deduce that
\begin{align*}
  V_b[\mathbf{0}]=
  \begin{pmatrix}
	V_b^1[\mathbf{0}] & 0 \\
	0 & V_b^2[\mathbf{0}]
  \end{pmatrix}
  =
  \begin{pmatrix}
	-\frac{1}{b}\mathbf{I}_1(b\varepsilon )\mathbf{K}_1(\varepsilon)+\mathbf{I}_1(b\varepsilon)\mathbf{K}_1(b\varepsilon) & 0\\
	0 & -\mathbf{I}_1(\varepsilon)\mathbf{K}_1(\varepsilon)+b\mathbf{I}_1(b\varepsilon )\mathbf{K}_1(\varepsilon)
  \end{pmatrix}.
\end{align*}
According to (A.2), (A.3) and (A.7) in \cite{Rou23a}, it follows that
\begin{align}\label{eq:basic_ineq_I_K}
  0<\mathbf{I}_1(b\varepsilon )\leqslant b\mathbf{I}_1(\varepsilon),
  \quad 0<\mathbf{K}_1(\varepsilon)<b\mathbf{K}_1(b\varepsilon),\quad \forall b\in (0,1),
\end{align}
thus we immediately get
\begin{align*}
  V_b^2[\mathbf{0}] < 0 < V_b^1[\mathbf{0}], \quad \forall b\in (0,1).
\end{align*}
Hence, $\mathcal{S}=\mathcal{S}_{\mathrm{max}} = (0,1)$, and using Theorem \ref{thm:general_kernel},
we have recovered the result in \cite[Theorem 1.1]{Rou23a}.


\subsection{2D Euler equation in the disc}\label{subsec:Euler-rad}
Consider the 2D incompressible Euler equation in vorticity form in the disc $R\mathbb{D} = B(\mathbf{0},R)$ ($R>1$)
with slip (non-penetration) boundary condition, that is, $u\cdot \mathbf{n}=0$ on $\partial (R\mathbb{D})$
with $\mathbf{n}$ the outer normal unit vector.
It corresponds to the active scalar equation \eqref{eq:ASE}-\eqref{eq:stream_function} in $\mathbf{D}=R\mathbb{D}$, $R>1$
and the kernel function $K(\mathbf{x},\mathbf{y})$ given by \eqref{eq:Green_Euler}.
\vskip0.5mm


\noindent Clearly, $t\in (0,\infty)\mapsto K_0(t)=-\frac{1}{2\pi}\log t$ satisfies assumptions \eqref{A1}-\eqref{A2} with $\alpha\in(0,1)$
and the additional term $(\mathbf{x},\mathbf{y})\in R\mathbb{D}\times R\mathbb{D}\mapsto K_1(\mathbf{x}, \mathbf{y})
=\tfrac{1}{2\pi}\log |R-R^{-1}\mathbf{x}\overline{\mathbf{y}}|$ is smooth and satisfies assumptions \eqref{A3}-\eqref{A4}.
As a result, Theorem \ref{thm:general_kernel} can be applied to obtain the desired existence result of
$m$-fold symmetric doubly connected V-states,
provided that $b\in \mathcal{S}$ and $m\in\mathbb{N}^\star$ is sufficiently large.

\noindent Note that $\widetilde{\lambda}_{n,b}$ and $\lambda_{n,b}$ are the same as in \eqref{eq:lam-tild-Euler} and
\eqref{eq:lam-Euler}. Observing that
\begin{align*}
  G_1(\rho_1,\theta,\rho_2,\eta) = K_1(\rho_1e^{i\theta},\rho_2e^{i\eta})
  =\tfrac{1}{2\pi} \Big(\log R + \log|1-  R^{-2}\rho_1\rho_2 e^{i\theta}e^{-i\eta}|\Big),
\end{align*}
and  $\partial_{\rho_1}G_1(\rho_1,\theta,\rho_2,\eta)=\tfrac{\rho_2}{\rho_1}\partial_{\rho_2}G_1(\rho_1,\theta,\rho_2,\eta)$,
and through integration by parts and using the following fact (see e.g. 4.397 in \cite{GR15})
\begin{align}\label{eq:aux-Euler}
  \int_0^{2\pi} \log |1 - a e^{-i\eta}| \dd \eta = \frac{1}{2}\int_0^{2\pi}\log(1-2a\cos \eta+a^2)\dd \eta=0,\quad |a|\leqslant  1,
\end{align}
we rewrite $\mathsf{c}_b$ and $\widetilde{\mathsf{c}}_b$ (given by \eqref{c-1} and
\eqref{c-2}) as
\begin{align*}
  \mathsf{c}_b &  =\frac{1}{b^2}\int_0^{2\pi}\int_b^1\partial_{\rho}G_1(b,0,\rho,\eta)\rho^2\dd \rho \dd \eta \\
  &  = \frac{1}{2\pi b^2} \int_0^{2\pi} \int_b^1\partial_\rho \Big( \log |1-R^{-2} b \rho e^{-i\eta}| \Big)\rho^2\dd \rho \dd \eta \\
  & = \frac{1}{2\pi b^2} \int_0^{2\pi} \Big(\log |1-R^{-2}b e^{-i\eta}| - b^2 \log |1-R^{-2}b^2 e^{-i\eta}|\Big) \dd \eta \\
  & \quad - \frac{1}{\pi b^2} \int_0^{2\pi} \int_b^1 \log|1 - R^{-2}b\rho e^{-i\eta}| \rho \dd \rho \dd \eta \\
  & = 0 ,
\end{align*}
and similarly,
\begin{align*}
  \widetilde{\mathsf{c}}_b =\int_0^{2\pi}\int_b^{1}\partial_{\rho}G_1(1,0,\rho,\eta)\rho^2\dd \rho \dd \eta
  = \frac{1}{2\pi} \int_0^{2\pi} \int_b^1\partial_\rho \Big( \log |1-R^{-2} \rho e^{-i\eta}| \Big)\rho^2\dd \rho \dd \eta = 0 .
\end{align*}
Using the fact that (see e.g. 4.397 in \cite{GR15})
\begin{align}\label{eq:aux-Euler-n}
  \int_0^{2\pi} \log |1-a e^{-i\eta}|e^{in\eta} \dd \eta  = \frac{1}{2}\int_0^{2\pi}\log(1-2a\cos \eta+a^2)\cos (n\eta)\dd \eta
  =-\frac{\pi}{n}a^n,\quad \forall |a|<1,
\end{align}
we find
\begin{align*}
  \widetilde{p}_{n,b} = \int_0^{2\pi} K_1(b,e^{i\eta}) e^{in\eta} \dd \eta
  = \frac{1}{2\pi} \int_0^{2\pi} \log|1-R^{-2} be^{-i\eta}| \cos(n\eta) \dd \eta
  = - \frac{1}{2n} (R^{-2}b)^n,
\end{align*}
and
\begin{align*}
  p_{n,b} = \int_0^{2\pi} K_1(b, b e^{i\eta}) e^{in\eta} \dd \eta
  = \frac{1}{2\pi} \int_0^{2\pi} \log|1-R^{-2} b^2 e^{-i\eta}| \cos(n\eta) \dd \eta
  = - \frac{1}{2n} (R^{-2}b^2)^n,
\end{align*}
we gather \eqref{eq:Lambda-n-b}, \eqref{eq:V-0} and the above estimates to infer that
\begin{align*}
  \mathbf{\Lambda}_{n,b}=
  \begin{pmatrix}
	-\frac{1}{2n} + \tfrac{b^{2n}R^{-2n}}{2n}& \frac{b^n}{2n}-\frac{b^{n}R^{-2n}}{2n}\\[1mm]
    -\frac{b^n}{2n} + \frac{b^{n}R^{-2n}}{2n}	& \frac{1}{2n} -\frac{R^{-2n}}{2n}
  \end{pmatrix},
  \quad \textrm{and} \quad
  V_b[\mathbf{0}]=
  \begin{pmatrix}
	V_b^1[\mathbf{0}] & 0\\
	0 & V_b^2[\mathbf{0}]
  \end{pmatrix}
  =
  \begin{pmatrix}
	0 & 0\\
	0 & \frac{-1+b^2}{2}
  \end{pmatrix}.
\end{align*}
In particular, we see that
\begin{align*}
  V_b^2[\mathbf{0}] < 0 = V_b^1[\mathbf{0}],\quad \forall b\in (0,1),
\end{align*}
which verifies the condition \eqref{eq:assum-1} for every $b\in (0,1)$, that is, $\mathcal{S}=(0,1)$.


Finally, based on the spectral analysis in Lemmas 16 and 18 in \cite{DHHM} (or the analysis in Section \ref{subsec:Euler-annulus}),
we deduce the main result of \cite[Theorem 6]{DHHM} from Theorem \ref{thm:general-statement}.
In addition, since $K_1(\mathbf{x},\mathbf{y})$ is harmonic with respect to $\mathbf{x},\mathbf{y}\in B(\mathbf{0},R)$,
we remark that $K_1(\mathbf{x},\mathbf{y})$ is analytic about $\mathbf{x}$ and $\mathbf{y}$.

\subsection{gSQG equation in the disc}\label{subsec:gSQG-rad}
For the gSQG equation in the radial domain $R\mathbb{D}=B(\mathbf{0},R)$ ($R>1$) with the slip boundary condition,
it is the equation \eqref{eq:ASE} in $\mathbf{D}=R\mathbb{D}$ with the stream function $\psi$ solving the following equation
\begin{align*}
  \psi=(-\Delta)^{-1+\frac{\beta}{2}}\omega,\;\;\beta\in (0,1),\quad \text{in}\,\; \mathbf{D}= R\mathbb{D},
  \quad\;\;\; \psi|_{\partial \mathbf{D}}=0,
\end{align*}
where $(-\Delta)^{-1+\frac{\beta}{2}}$ is the spectral fractional Laplacian defined by
\begin{align*}
  (-\Delta)^{-1 +\frac{\beta}{2}} \omega (\mathbf{x})
  = \frac{1}{\Gamma(1-\frac{\beta}{2})} \int_0^{+\infty} t^{-\frac{\alpha}{2}} e^{t\Delta} \omega(\mathbf{x}) \dd t
  = \int_{\mathbf{D}} K(\mathbf{x},\mathbf{y}) \omega(\mathbf{y}) \dd \mathbf{y}.
\end{align*}
In view of Lemma \ref{lem:disc-exp} and \eqref{eq:eigen_func_scale},
we know that $\frac{x_{n,k}^2}{R^2}$ ($n\in\mathbb{N}$, $ k\in\mathbb{N}^\star$) is the eigenvalue counted with their multiplicities
and $\frac{1}{R}\big(\phi^{(1)}_{n,k}(\mathbf{x}),\phi^{(2)}_{n,k}(\mathbf{x}) \big)$ are the associated
$L^2(\mathbf{D})$-normalized eigenfunctions of the Laplacian operator
$-\Delta$ supplemented with Dirichlet boundary condition, namely,
they solve the spectrum problem \eqref{eq:eigen_func_scale}. Thus the spectral Green function $K(\mathbf{x},\mathbf{y})$
in \eqref{eq:stream_function} can be expressed as (see \cite[Eq. (12)]{HXX23a})
\begin{equation}\label{eq:Kexp-SQG-rad}
\begin{split}
  K(\mathbf{x},\mathbf{y})=&\sum_{n\in \mathbb{N},k\geqslant  1}\Big(\frac{x^2_{n,k}}{R^2}\Big)^{\frac{\beta}{2}-1}
  \frac{1}{R^2} \Big( \phi^{(1)}_{n,k}(\tfrac{\mathbf{x}}{R}) \phi^{(1)}_{n,k}(\tfrac{\mathbf{y}}{R})
  + \phi^{(2)}_{n,k}(\tfrac{\mathbf{x}}{R}) \phi^{(2)}_{n,k}(\tfrac{\mathbf{y}}{R})\Big) \\
  = & \sum_{n\in \mathbb{N},k\geqslant  1}\frac{x_{n,k}^{\beta-2}}{R^{\beta}}\Big(\phi^{(1)}_{n,k}
  (\tfrac{\mathbf{x}}{R})\phi^{(1)}_{n,k}(\tfrac{\mathbf{y}}{R})
  +\phi^{(2)}_{n,k}(\tfrac{\mathbf{x}}{R})\phi^{(2)}_{n,k}(\tfrac{\mathbf{y}}{R})\Big),
\end{split}
\end{equation}
where $\phi^{(i)}_{n,k}(\mathbf{x})$ ($i=1,2$) is given by \eqref{def:lambd-phi-nk} and $x_{n,k}$ is the $k$-th zero point of
the Bessel function $J_n(\cdot)$.
On the other hand, according to \cite[Lemma 2.3]{HXX23a}, the kernel $K(\mathbf{x},\mathbf{y})$ can also be written as
\begin{align*}
  K(\mathbf{x},\mathbf{y})=K_0(|\mathbf{x}-\mathbf{y}|)+K_1(\mathbf{x},\mathbf{y})
  =c_{\beta}|\mathbf{x}- \mathbf{y}|^{-\beta}+K_1(\mathbf{x},\mathbf{y}),
  \quad c_{\beta}\triangleq\tfrac{\Gamma(\frac{\beta}{2})}{\pi 2^{2-\beta}\Gamma(1-\frac{\beta}{2})},
\end{align*}
and $K_1\in C^{\infty}(\mathbf{D}\times \mathbf{D})$.
Here, as an application of Theorem \ref{thm:general_kernel},
we intend to construct non-trivial $m$-fold symmetric doubly connected
V-states bifurcating from $\mathbf{1}_{\mathbb{D}\setminus \overline{b\mathbb{D}}}$
with sufficiently large $m\in\mathbb{N}^\star$ for this radial-domain gSQG model.
To our best knowledge, this was not obtained in the literature.
\vskip0.5mm

\noindent It is clear to see that $K_0(t)=c_\beta t^{-\beta}$ satisfies the assumptions
\eqref{A1}-\eqref{A2} with $\alpha\in(0,1-\beta]$, and by virtue of \cite[Lemma 2.4]{HXX23a},
the function $K_1(\mathbf{x},\mathbf{y})$ satisfies the assumptions \eqref{A3}-\eqref{A4}.
Hence, we can directly apply Theorem \ref{thm:general_kernel}, provided that we can find some $b\in (0,1)$ such that
$V_{b}^{1}[\mathbf{0}]\ne V_{b}^{2}[\mathbf{0}]$.
\vskip0.5mm

\noindent
Next, by virtue of \eqref{eq:Kexp-SQG-rad} and \eqref{def:lambd-phi-nk}, we see that
\begin{align*}
  G(\rho_1,\theta,\rho_2,\eta)=K(\rho_1e^{i\theta},\rho_2e^{i\eta})=R^{-\beta}\sum_{n\in \NN,k\geqslant  1}x^{\beta-2}_{n,k}A^2_{n,k}
  J_n(x_{n,k}\tfrac{\rho_1}{R})J_n(x_{n,k}\tfrac{\rho_2}{R})\cos\big(n(\theta-\eta)\big),
\end{align*}
where $A_{n,k}$ is given by \eqref{eq:Ank}.
Thus, from \eqref{eq:Vb1-2[0]}, \eqref{eq:K-der} and the following fact (see e.g. 6.561 of \cite{GR15})
\begin{align}\label{eq:bessel_ids}
	J_0'(r)=-J_1(r),\quad \int_0^{a}tJ_0(t)\dd t=aJ_1(a),
\end{align}
we find that
\begin{align*}
  V_b^1[\mathbf{0}] & = \frac{1}{b}\int_0^{2\pi}\int_b^{1}\partial_{\rho_1}G(b,\theta,\rho,\eta)\rho \dd \rho\dd \eta \\
  & = \frac{1}{b} R^{-\beta -1} \int_0^{2\pi} \int_b^1  \sum_{n\in\mathbb{N},k\geqslant 1}
  x_{n,k}^{\beta-1} A_{n,k}^2 J_n'(x_{n,k} \tfrac{b}{R}) J_n(x_{n,k} \tfrac{\rho}{R}) \cos\big(n(\theta-\eta)\big)
  \rho \dd \rho \dd \eta \\
  & = \frac{1}{b} R^{-\beta -1} 2\pi  \sum_{k\geqslant 1}
  x_{0,k}^{\beta-1} A_{0,k}^2 J_0'(x_{0,k} \tfrac{b}{R}) \int_b^1  J_0(x_{0,k} \tfrac{\rho}{R}) \rho \dd \rho  \\
  & = -2R^{-\beta}\sum_{k\geqslant  1}x^{\beta-2}_{0,k}\frac{1}{J_1^2(x_{0,k})} J_1(x_{0,k}\tfrac{b}{R})
  \Big(\tfrac{1}{b}J_1(x_{0,k} \tfrac{1}{R})  - J_1(x_{0,k}\tfrac{b}{R})\Big),
\end{align*}
and similarly,
\begin{align*}
  V_b^2[\mathbf{0}]=&\int_0^{2\pi}\int_{b}^{1}\partial_{\rho_1}G(1,\theta,\rho,\eta)\rho \dd \rho\dd \eta \\
  = & -2R^{-\beta}\sum_{k\geqslant  1}x^{\beta-2}_{0,k}\frac{1}{J_1^2(x_{0,k})} J_1(x_{0,k}\tfrac{1}{R})
  \Big(J_1(x_{0,k}\tfrac{1}{R})-bJ_1(x_{0,k}\tfrac{b}{R})\Big).
\end{align*}
By using Sneddon's formula (see Lemma \ref{lem:sned-form} below) and changing of variables,
we also can derive another expression formulas for $V_b[\mathbf{0}]$:
\begin{equation}\label{def:Vb1-SQG}
\begin{split}
  V_b^1[\mathbf{0}] & = - \frac{2 \sin (\tfrac{\pi}{2}\beta)}{\pi R^\beta}  \bigg(
  \int_0^{\infty}\mathbf{I}_1(\tfrac{b}{R}\rho) \big(\tfrac{1}{b}\mathbf{K}_1(\tfrac{1}{R}\rho)
  - \mathbf{K}_1(\tfrac{b}{R}\rho)\big)\rho^{\beta-1}\dd \rho \\
  & \quad \qquad  + \int_0^{\infty}\mathbf{I}_1(\tfrac{b}{R}\rho)
  \big(\tfrac{1}{b}\mathbf{I}_1(\tfrac{1}{R}\rho)-\mathbf{I}_1(\tfrac{b}{R}\rho)\big)
  \frac{\mathbf{K}_0(\rho)}{\mathbf{I}_0(\rho)}\rho^{\beta-1}\dd \rho \bigg) \\
  & = - \frac{2 \sin (\tfrac{\pi}{2}\beta)}{\pi }  \bigg(
  \int_0^{\infty} \mathbf{I}_1(b\rho) \big(\tfrac{1}{b}\mathbf{K}_1(\rho)
  - \mathbf{K}_1(b\rho)\big)\rho^{\beta-1}\dd \rho \\
  & \quad \qquad  + \int_0^{\infty}\mathbf{I}_1(b\rho)
  \big(\tfrac{1}{b}\mathbf{I}_1(\rho)-\mathbf{I}_1(b\rho)\big)
  \frac{\mathbf{K}_0(R\rho)}{\mathbf{I}_0(R\rho)}\rho^{\beta-1}\dd \rho \bigg),
\end{split}
\end{equation}
and
\begin{align*}
  V_b^2[\mathbf{0}] & = - \frac{2\sin(\tfrac{\pi}{2}\beta)}{\pi}  \bigg(
  \int_0^{\infty}\mathbf{K}_1(\rho) \big(\mathbf{I}_1(\rho)-b\mathbf{I}_1(b\rho)\big)
  \rho^{\beta-1}\dd \rho \\
  & \quad \qquad  + \int_0^{\infty}\mathbf{I}_1(\rho)
  \big( \mathbf{I}_1(\rho)-b\mathbf{I}_1(b\rho) \big)
  \frac{\mathbf{K}_0(R\rho)}{\mathbf{I}_0(R\rho)}\rho^{\beta-1}\dd \rho \bigg),
\end{align*}
where $\mathbf{I}_n$ and $\mathbf{K}_n$ are the modified Bessel functions of the first and second kind, respectively, defined by \eqref{def:In-z} and \eqref{def:Kn-z}.

\noindent Next, we check that $V_{b}^{1}[\mathbf{0}]\ne V_{b}^{2}[\mathbf{0}]$ holds for every $b\in (0,1)$.
Notice that
\begin{equation}\label{eq:Vb1-2-SQG}
\begin{split}
  V_b^1[\mathbf{0}] - V_b^2[\mathbf{0}]
  & = \frac{2\sin(\frac{\pi}{2}\beta)}{\pi} \int_0^\infty \mathbf{I}_1(b\rho)
  \big(\mathbf{K}_1(b\rho) -\tfrac{1}{b} \mathbf{K}_1( \rho) \big) \rho^{\beta-1} \dd \rho  \\
  & \quad + \frac{2\sin(\frac{\pi}{2}\beta)}{\pi}
  \int_0^{\infty}\mathbf{K}_1(\rho) \big(\mathbf{I}_1(\rho)-b\mathbf{I}_1(b\rho)\big)
  \rho^{\beta-1}\dd \rho  \\
  & \quad + \frac{2\sin(\frac{\pi}{2}\beta)}{\pi }
  \int_0^{\infty} \Big( \mathbf{I}_1^2 (\rho)
  - (b + \tfrac{1}{b}) \mathbf{I}_1(\rho) \mathbf{I}_1(b \rho)
  + \mathbf{I}_1^2(b\rho) \Big)
  \frac{\mathbf{K}_0(R\rho)}{\mathbf{I}_0(R\rho)}\rho^{\beta-1}\dd \rho .
\end{split}
\end{equation}
Owing to  \eqref{eq:basic_ineq_I_K}, we find that for every $\rho>0$ and $b\in(0,1)$,
\begin{align*}
  \mathbf{K}_1(b\rho)  - \tfrac{1}{b}\mathbf{K}_1(\rho)  >0,
  \quad \textrm{and}\quad
  \mathbf{I}_1(\rho) - b \mathbf{I}_1(b \rho)>0,
\end{align*}
which combined with the fact $\mathbf{I}_0(\rho)>0$ and $\mathbf{K}_0(\rho)>0$ for all $\rho\in (0,\infty)$ immediately yields that
the first two terms on the right-hand side of \eqref{eq:Vb1-2-SQG} are both positive.
For the remaining third term on the right-hand side of \eqref{eq:Vb1-2-SQG}, noting that
\begin{align*}
  \mathbf{I}_1(x)= \frac{x}{2}\mathcal{I}(x),\quad \textrm{with}\quad
  \mathcal{I}(x)\triangleq \sum_{m=0}^{\infty}\frac{(\frac{x}{2})^{2m}}{m!\Gamma(m+1)},
\end{align*}
the wanted positivity follows from the following estimate that
\begin{equation}\label{eq:QGSW_ineq_2}
\begin{split}
  \mathbf{I}^2_1(bx)+\mathbf{I}_1^2(x)- \big(b+\tfrac{1}{b}\big) \mathbf{I}_1(bx)\mathbf{I}_1(x)
  & = \frac{x^2}{4} \Big(b^2 \mathcal{I}^2(bx)+\mathcal{I}^2(x)-(b^2+1)\mathcal{I}(bx)\mathcal{I}(x)\Big) \\
  & = \frac{x^2}{4} \big(\mathcal{I}(x) - \mathcal{I}(bx) \big)
  \big(\mathcal{I}(x) - b^2 \mathcal{I}(bx) \big)>0.
\end{split}
\end{equation}
Consequently, we conclude that
\begin{align*}
  \textrm{$V_b^1[\mathbf{0}]-V_b^2[\mathbf{0}]>0$,\quad $\forall b\in (0,1)$}.
\end{align*}
Therefore, we can apply Theorem \ref{thm:general_kernel} to this radial-domain gSQG case for all $b\in (0,1)$
and get the existence of $m$-fold symmetric doubly connected V-states bifurcating from $\mathbf{1}_{\mathbb{D}\setminus \overline{b\mathbb{D}}}$
with sufficiently large $m\in\mathbb{N}^\star$.

\noindent
Besides, based on the above analysis, it also holds that $V_b^2[\mathbf{0}] < 0$,
but we are not sure about the sign of $V_b^1[\mathbf{0}]$.  In addition, by letting $R\to +\infty$,
we can recover the $V_b[\mathbf{0}]$ in the whole space given by \eqref{eq:Vb0-SQG}.
We take $V_b^1[\mathbf{0}]$ as an example.
In view of \eqref{def:In-z} and (A.3), (A.7) in \cite{Rou23a}, we see that for every $x\in (0,+\infty)$ and
$n\in \mathbb{N}$,
\begin{align*}
  \lim_{x\rightarrow +\infty}\mathbf{I}_n(x)= + \infty,
  \quad  \mathbf{I}_n(x)>0,\quad \mathbf{K}_n(x)>0,\quad \mathbf{K}_n'(x) < 0,
\end{align*}
which implies that
\begin{align}\label{eq:limit_K_I_n}
  \lim_{R\to +\infty}\frac{\mathbf{K}_n(R\rho)}{\mathbf{I}_n(R\rho)}=0,\quad \forall  \rho>0,\,n\in \mathbb{N}.
\end{align}
Thus, using the dominated convergence theorem, \eqref{def:J} and \eqref{eq:gSQG_identity_spectrum},
it follows that
\begin{align*}
  \lim_{R\to +\infty}V_b^1[\mathbf{0}] & = -\frac{2}{\pi}\sin \big(\tfrac{\pi}{2}\beta \big)
  \int_0^{\infty}\mathbf{I}_1(b\rho) \big(\tfrac{1}{b}\mathbf{K}_1(\rho)
  - \mathbf{K}_1(b\rho)\big) \rho^{\beta-1} \dd \rho \\
  & = -\Lambda_{1,b}(\beta) + b^{-\beta}\Lambda_{1,1}(\beta),
\end{align*}
and $V_b^{1}[\mathbf{0}]$ is positive at least for sufficiently large $R$. In fact, from the numerical experiments
we can conjecture that $V_b^{1}[\mathbf{0}]$ is positive for any $\beta\in (0,1)$, $b\in (0,1)$ and $R>1$.


\subsection{QGSW equation in the disc}\label{subsec:QGSW-rad}
Consider the QGSW model in the radial domain $R\mathbb{D}=B(\mathbf{0},R)$ ($R>1$) with the slip boundary condition,
it is the equation \eqref{eq:ASE} in $\mathbf{D}=R\mathbb{D}$ with the stream function $\psi$
satisfying the following Dirichlet problem
\begin{align*}
  (-\Delta + \varepsilon^2) \psi = \omega\quad \textrm{in}\;\, \mathbf{D},\quad \psi |_{\partial \mathbf{D}} = 0.
\end{align*}
Moreover, $\psi$ has the expression formula \eqref{eq:stream_function} with the spectral Green function $K(\mathbf{x},\mathbf{y})$
given by
\begin{align*}
  K(\mathbf{x},\mathbf{y})=K_0(|\mathbf{x}-\mathbf{y}|)+K_1(\mathbf{x},\mathbf{y}),\quad \text{with}\quad K_0(|\mathbf{x}-\mathbf{y}|)
  = \frac{1}{2\pi}\mathbf{K}_0(\varepsilon|\mathbf{x}-\mathbf{y}|),\,\varepsilon>0,
\end{align*}
and $K(\mathbf{x},\mathbf{y})$ solving the following elliptic problem
\begin{align}\label{eq:K1}
  -\Delta_{\mathbf{y}}K_1(\mathbf{x},\mathbf{y})
  +\varepsilon^2K_1(\mathbf{x},\mathbf{y})=0\quad \text{in}\,\; \mathbf{D},
  \quad K_1(\mathbf{x},\mathbf{y})|_{\mathbf{y}\in \partial \mathbf{D}}
  = -\frac{1}{2\pi}\mathbf{K}_0(\varepsilon|\mathbf{x}-\mathbf{y}|).
\end{align}
To the best of our knowledge, the question of the existence of nontrivial doubly connected
V-states in the radial-domain QGSW model is still unsolved, and it is our goal here to address this issue.
\vskip0.5mm

\noindent To start with, $K_0(t) = \frac{1}{2\pi} \mathbf{K}_0(\varepsilon t)$ satisfies the assumptions \eqref{A1}-\eqref{A2}
as in the whole-place QGSW case, and we refer to \cite[Sec. 5.7]{HXX23b} for details of verifying the assumptions
\eqref{A3}-\eqref{A4} for $K_1(\mathbf{x},\mathbf{y})$; indeed, $K_1(\mathbf{x},\mathbf{y})\in C^\infty(\mathbf{D}\times \mathbf{D})$.
Hence, in order to apply Theorem \ref{thm:general_kernel}, it is required to show $V_b^{1}[\mathbf{0}]\ne V_b^{2}[\mathbf{0}]$
for some $b\in (0,1)$.

\noindent Next, we first give another explicit expression formula of $K(\mathbf{x},\mathbf{y})$ on
$\mathbf{D}=R\mathbb{D}=B(\mathbf{0},R)$. According to Lemma \ref{lem:disc-exp} and \eqref{eq:eigen_func_scale},
the eigenvalues $\frac{1}{R^2}x_{n,k}^2$ and the associated eigenfunctions
$\big(\frac{1}{R}\phi_{n,k}^{(1)}(\frac{\mathbf{x}}{R}),\frac{1}{R}\phi_{n,k}^{(2)}(\frac{\mathbf{x}}{R})\big)$
solve the spectral Laplacian problem \eqref{eq:eigen_func_scale}.
Thus, we infer that
\begin{align*}
  \psi(\mathbf{x}) = \int_{\mathbf{D}} \sum_{n\in \mathbb{N},k\geqslant  1}\frac{1}{\frac{x^2_{n,k}}{R^2}+\varepsilon^2}
  \Big(\frac{1}{R^2} \phi^{(1)}_{n,k}(\tfrac{\mathbf{x}}{R}) \phi^{(1)}_{n,k}(\tfrac{\mathbf{y}}{R})
  + \frac{1}{R^2} \phi^{(2)}_{n,k}(\tfrac{\mathbf{x}}{R}) \phi^{(2)}_{n,k}(\tfrac{\mathbf{y}}{R})\Big)
  \omega(\mathbf{y})\dd \mathbf{y},
\end{align*}
which yields that
\begin{align*}
  K(\mathbf{x},\mathbf{y}) 
  & = \sum_{n\in \mathbb{N},k\geqslant  1}\frac{1}{x^2_{n,k}+\varepsilon^2R^2}\Big(\phi^{(1)}_{n,k}
  (\tfrac{\mathbf{x}}{R})\phi^{(1)}_{n,k}(\tfrac{\mathbf{y}}{R})
  +\phi^{(2)}_{n,k}(\tfrac{\mathbf{x}}{R})\phi^{(2)}_{n,k}(\tfrac{\mathbf{y}}{R})\Big).
\end{align*}
Consequently, we find
\begin{align*}
 G(\rho_1,\theta,\rho_2,\eta)= K(\rho_1e^{i\theta},\rho_2e^{i\eta})=\sum_{n\in \mathbb{N},k\geqslant  1}\frac{1}{x^2_{n,k}
  + \varepsilon^2R^2}A^2_{n,k}J_n(x_{n,k}\tfrac{\rho_1}{R})J_n(x_{n,k}\tfrac{\rho_2}{R})\cos\big(n(\theta-\eta)\big),
\end{align*}
with $A_{n,k}$ given by \eqref{eq:Ank}.
Thanks to \eqref{eq:Vb1-2[0]}-\eqref{def:Lambda-nb}, \eqref{eq:K-der}, together with \eqref{eq:bessel_ids},
we deduce that
\begin{equation}\label{eq:QGSW_radial_domain_1}
\begin{aligned}
  V_b^1[\mathbf{0}] & = \frac{1}{b}\int_0^{2\pi}\int_b^{1}\partial_{\rho_1}G(b,\theta,\rho,\eta)\rho \dd \rho\dd \eta \\
  & = \frac{1}{b} \int_0^{2\pi} \int_b^1 \sum_{n\in\mathbb{N},k\geqslant 1}
  \frac{1}{x_{n,k}^2 +\varepsilon^2 R^2} \frac{x_{n,k}}{R} A_{n,k}^2 J_n'(x_{n,k}\tfrac{b}{R}) J_n(x_{n,k}\tfrac{\rho}{R} )
  \cos\big(n(\theta -\eta)\big) \rho \dd \rho \dd \eta \\
  & = \frac{2\pi}{b} \sum_{k\geqslant 1} \frac{1}{x_{0,k}^2 +\varepsilon^2 R^2}
  A_{0,k}^2 J_0'(x_{0,k} \tfrac{b}{R}) \int_b^1 \tfrac{x_{0,k}}{R} J_0(x_{0,k} \tfrac{\rho}{R}) \rho \dd \rho  \\
  & = -2\sum_{k\geqslant  1}\frac{1}{x^2_{n,k}+\varepsilon^2R^2} \frac{J_1(x_{0,k}\frac{b}{R})}{J_1^2(x_{0,k})}
  \Big(b^{-1}J_1(x_{0,k}\tfrac{1}{R})-J_1(x_{0,k}\tfrac{b}{R})\Big),
\end{aligned}
\end{equation}
and
\begin{equation}\label{eq:QGSW_radial_domain_2}
\begin{aligned}
  V_b^2[\mathbf{0}]=&\int_0^{2\pi}\int_b^{1}\partial_{\rho_1} G(1,\theta,\rho,\eta)\rho \dd \rho\dd \eta\\
  = & -2\sum_{k\geqslant  1}\frac{1}{x^2_{n,k}+\varepsilon^2R^2}\frac{J_1(x_{0,k}\frac{1}{R})}{J_1^2(x_{0,k})}
  \Big(J_1(x_{0,k}\tfrac{1}{R})-bJ_1(x_{0,k}\tfrac{b}{R})\Big).
\end{aligned}
\end{equation}
By virtue of Lemma \ref{lem:sum_up_QGSW} below, 
it follows that
\begin{align}\label{eq:Vb1-QGSW-rad}
  V_b^1[\mathbf{0}]=-\frac{\mathbf{K}_0( R\varepsilon) \mathbf{I}_1(b\varepsilon)}{\mathbf{I}_0( R\varepsilon)}
  \Big(\tfrac{1}{b}\mathbf{I}_1(\varepsilon) -\mathbf{I}_1( b\varepsilon)\Big)
  -  \mathbf{I}_1(b\varepsilon )\Big(\tfrac{1}{b} \mathbf{K}_1(\varepsilon)
  -  \mathbf{K}_1(b\varepsilon) \Big),
\end{align}
and
\begin{align*}
  V_b^2[\mathbf{0}] = - \frac{\mathbf{K}_0(R \varepsilon ) \mathbf{I}_1(\varepsilon)}{\mathbf{I}_0( R\varepsilon)}
  \Big(\mathbf{I}_1(\varepsilon )   - b \mathbf{I}_1( b\varepsilon)\Big)
  - \mathbf{K}_1(\varepsilon ) \Big(\mathbf{I}_1(\varepsilon )
  - b\mathbf{I}_1(b\varepsilon )\Big),
\end{align*}
and
\begin{align*}
  V_b^1[\mathbf{0}] - V_b^2[\mathbf{0}] & = \frac{\mathbf{K}_0( R\varepsilon)}{\mathbf{I}_0( R\varepsilon)}
  \Big(\mathbf{I}_1^2(\varepsilon) - ( \tfrac{1}{b} + b)\mathbf{I}_1(\varepsilon) \mathbf{I}_1(b\varepsilon)
  + \mathbf{I}_1^2( b\varepsilon)\Big) \\
  & \quad +  \mathbf{I}_1(b\varepsilon )\Big(\mathbf{K}_1(b\varepsilon) -  \tfrac{1}{b} \mathbf{K}_1(\varepsilon) \Big)
  + \mathbf{K}_1(\varepsilon ) \Big(\mathbf{I}_1(\varepsilon )  - b\mathbf{I}_1(b\varepsilon )\Big) .
\end{align*}
In view of \eqref{eq:basic_ineq_I_K} and \eqref{eq:QGSW_ineq_2}, we find that for any $b\in (0,1)$,
\begin{align*}
  V_b^2[\mathbf{0}]<0,\quad V_b^1[\mathbf{0}]-V_b^{2}[\mathbf{0}]>0.
\end{align*}
Therefore, we have verified all the conditions in Theorem \ref{thm:general_kernel}
and get the existence of nontrivial $m$-fold symmetric doubly connected V-states bifurcating from $\mathbf{1}_{\mathbb{D}\setminus \overline{b\mathbb{D}}}$
for sufficiently large $m\in\mathbb{N}^\star$.

\noindent
Due to \eqref{eq:limit_K_I_n}, and taking $R\to \infty$, we recover $V_b[\mathbf{0}]$
in the QGSW equation of the whole space. Then,  $V_b^{1}[\mathbf{0}]$ is positive for $R>1$ large enough.
From the numerical experiments, it strongly suggests that $V_b^{1}[\mathbf{0}]$ given by
\eqref{eq:Vb1-QGSW-rad} is positive for every $\varepsilon>0$, $b\in (0,1)$ and $R>1$.

\noindent Finally, we remark that because $K_1$ solves the elliptic problem \eqref{eq:K1},
it follows from the classical elliptic PDE theory (e.g. \cite{Muller02}) that $K_1(\mathbf{x},\mathbf{y})$
is analytic about $\mathbf{x}$ and $\mathbf{y}$ in any domain compactly embedded in $\mathbf{D}=B(\mathbf{0},R)$.

\subsection{2D Euler equation in the annulus}\label{subsec:Euler-annulus}
Consider the 2D Euler equation in vorticity form in the annulus
$R_2\mathbb{D}\setminus \overline{R_1 \mathbb{D}}$ ($R_2>1>b>R_1>0$)
with the slip (non-penetration) boundary condition, that is,
$u\cdot \mathbf{n}=0$ on $\partial (R_2\mathbb{D})\cup \partial (R_1\mathbb{D})$
with $\mathbf{n}$ the outer normal unit vector.
It corresponds to the active scalar equation \eqref{eq:ASE}
in $\mathbf{D}=R_2\mathbb{D}\setminus \overline{R_1\mathbb{D}}$ and the stream function $\psi$ solves the Dirichlet problem of the Poisson equation
\begin{align*}
  -\Delta \psi = \omega\quad \textrm{in}\,\,\mathbf{D} = R_2\mathbb{D}\setminus \overline{R_1\mathbb{D}},\quad\; \psi|_{\partial \mathbf{D}} = 0.
\end{align*}
According to \cite{GT14,Hickey29} or \cite[Appendix V]{EP96}, $\psi$ has the expression formula \eqref{eq:stream_function}
with the Green function $K(\mathbf{x},\mathbf{y})$ given by
\begin{equation}\label{eq:K-Euler-annulus}
\begin{split}
  K(\mathbf{x},\mathbf{y})=\frac{1}{2\pi}\bigg(& -\log |\mathbf{x}-\mathbf{y}|
  + \mathcal{A}_0(|\mathbf{y}|)+\mathcal{B}_0(|\mathbf{y}|)\log |\mathbf{x}| \\
  &-\sum_{m=1}^{\infty}\frac{1}{m} \Big(\mathcal{A}_m(|\mathbf{y}|)|\mathbf{x}|^m
  +\mathcal{B}_m(|\mathbf{y}|)|\mathbf{x}|^{-m} \Big) \cos m(\theta-\eta)\bigg),
\end{split}
\end{equation}
where $\mathbf{x}=(|\mathbf{x}|\cos \theta,|\mathbf{x}|\sin \theta)$,
$\mathbf{y}=(|\mathbf{y}|\cos \eta,|\mathbf{y}|\sin \eta)$, and
\begin{align*}
  \mathcal{A}_0(r)\triangleq \log R_2 \frac{\log (R_1/r)}{\log (R_1/R_2)},
  &\qquad \mathcal{B}_0(r)\triangleq \frac{\log (r/R_2)}{\log (R_1/R_2)},\\
  \mathcal{A}_m(r)\triangleq \frac{r^{m}-\big(\frac{R_1^2}{r}\big)^m}
  {R_2^{2m}- R_1^{2m}},
  &\qquad \mathcal{B}_m(r)\triangleq \frac{R_1^{2m}\big(\big(\frac{R_2^2}{r}\big)^m
  -r^m \big)}{R_2^{2m}-R_1^{2m}}.
\end{align*}
In order to show that $K(\mathbf{x},\mathbf{y}) = -\frac{1}{2\pi} \log |\mathbf{x}-\mathbf{y}| + K_1(\mathbf{x},\mathbf{y})$
satisfies the assumptions \eqref{A1}-\eqref{A4} with $\alpha\in(0,1)$,
it suffices to check the symmetric assumption \eqref{A4}: for every $\mathbf{x}\neq \mathbf{y}\in \mathbf{D}$,
$K_1(\mathbf{x},\mathbf{y}) = K_1(\mathbf{y},\mathbf{x})$
is the classical property of Green function (e.g. see \cite[pp. 35]{Evans10}) and it can also be checked from the explicit formula \eqref{eq:K-Euler-annulus};
the properties $K_1(\mathbf{x}, \mathbf{y})=K_1(\mathbf{x}e^{i\vartheta}, \mathbf{y}e^{i\vartheta})$ ($\forall\vartheta \in \mathbb{R}$) and
$K_1(\mathbf{x}, \mathbf{y}) = K_1(\overline{\mathbf{x}}, \overline{\mathbf{y}})$ follows from \eqref{eq:K-Euler-annulus} and the fact that
$\cos m(\theta+\vartheta-(\eta+\vartheta))=\cos m(\theta-\eta)=\cos m(-\theta-(-\eta))$.
Hence, we can apply Theorem \ref{thm:general-statement} to this case.

\noindent Now, we compute $V_b^{1}[\mathbf{0}]$ and $V_b^{2}[\mathbf{0}]$.
We first write
\begin{align*}
  G_1(\rho_1,\theta,\rho_2,\eta) & = K_1(\rho_1 e^{i\theta},\rho_2 e^{i\eta}) \\
  & = \frac{1}{2\pi}\bigg(\mathcal{A}_0(\rho_2)+\mathcal{B}_0(\rho_2)
  \log \rho_1 -\sum_{m=1}^\infty \frac{1}{m}\Big(\mathcal{A}_m(\rho_2) \rho_1^m
  +\mathcal{B}_m(\rho_2)\rho_1^{-m}\Big) \cos m(\theta-\eta)\bigg).
\end{align*}
Recalling \eqref{c-1} and \eqref{c-2}, we deduce that
\begin{align}\label{def:Cb}
  \widetilde{\mathsf{c}}_b & = \int_0^{2\pi}\int_{b}^{1}\partial_{\rho_1}G_1(1,0,\rho,\eta)\rho
  \dd \rho \dd \eta = \int_b^1 \mathcal{B}_0(\rho) \rho\, \dd\rho \nonumber \\
  & = \frac{1}{\log (R_1/R_2)}\int_{b}^{1}\log (\tfrac{\rho}{R_2}) \rho \dd \rho
  \triangleq \mathfrak{C}_b,
\end{align}
and similarly,
\begin{align*}
  \mathsf{c}_b = \frac{1}{b}\int_0^{2\pi}\int_{b}^{1}\partial_{\rho_1}G_1(b,0,\rho,\eta)\rho
  \dd \rho \dd \eta
  =\frac{1}{b}\int_{b}^{1}\mathcal{B}_0(\rho)\frac{\rho}{b}\dd \rho
  = \frac{1}{b^2}\mathfrak{C}_b.
\end{align*}
Through integration by parts, we can exactly compute $\mathfrak{C}_b$:
\begin{align*}
   \mathfrak{C}_b=\frac{1}{\log (R_1/R_2)}
  \Big(-\frac{b^2}{2}\log b-\frac{1-b^2}{4}-\frac{1-b^2}{2}\log R_2\Big).
\end{align*}
From the definition of $\mathfrak{C}_b$ in \eqref{def:Cb}, we also see that
\begin{align*}
  0< \mathfrak{C}_b < \int_b^1 \rho \dd \rho = \frac{1-b^2}{2},\quad \forall 1<R_2<+\infty,\quad \textrm{and} \quad
  \lim_{R_2 \rightarrow +\infty} \mathfrak{C}_b = \frac{1-b^2}{2}.
\end{align*}

Then, recalling \eqref{eq:V-0}, together with the computations in Section \ref{subsec:Euler}, it follows that
\begin{align*}
  V_b[\mathbf{0}]=
  \begin{pmatrix}
    V_b^1[\mathbf{0}] & 0\\
	0 & V_b^2[\mathbf{0}]
  \end{pmatrix}
  =
  \begin{pmatrix}
    \frac{1}{b^2}\mathfrak{C}_b & 0\\
	0 & -\frac{1-b^2}{2}+\mathfrak{C}_b
  \end{pmatrix},
\end{align*}
and
\begin{align*}
  V_b^1[\mathbf{0}]-V_b^2[\mathbf{0}]=\frac{1-b^2}{2}+\frac{1-b^2}{b^2}\mathfrak{C}_b>0,\quad \forall b\in (R_1,1),
\end{align*}
which verifies the condition \eqref{eq:assum-1} for every $b\in \mathcal{S}=(R_1,1)$.
Furthermore, in view of \eqref{def:Pnb}, we find that
\begin{align*}
  p_{n,b} & = \int_{\TT}K_1(b,be^{i\eta})e^{in\eta}\dd \eta=\int_{\TT}G_1(b,0,b,\eta)e^{in\eta}\dd \eta\\
  & = -\frac{1}{2\pi n}\Big(\mathcal{A}_n(b)b^{-n}+\mathcal{B}_n(b)b^{-n}\Big)\int_0^{2\pi}\cos^2 (n\eta)\dd \eta\\
  & =  -\frac{1}{2n}\Big(\mathcal{A}_n(b)b^{n}+\mathcal{B}_n(b)b^{-n}\Big),
\end{align*}
and
\begin{align*}
  \widetilde{p}_{n,b}= \int_{\TT}K_1(b,e^{i\eta})e^{in\eta}\dd \eta 
  =-\frac{1}{2 n}\Big(\mathcal{A}_n(1)b^n+\mathcal{B}_n(1)b^{-n}\Big).
\end{align*}
Thanks to \eqref{eq:Lambda-n-b} and \eqref{eq:lam-tild-Euler}-\eqref{eq:lam-Euler}, we obtain
\begin{align*}
  \mathbf{\Lambda}_{n,b} & = \frac{1}{2n}
  \begin{pmatrix}
	\mathcal{A}_n(b)b^n + \mathcal{B}_n(b)b^{-n}-1 & -(\mathcal{A}_n(1)-1)b^n-\mathcal{B}_n(1)b^{-n}\\
	(\mathcal{A}_n(1)-1)b^n+\mathcal{B}_n(1)b^{-n} & -\mathcal{A}_n(1)-\mathcal{B}_n(1)+1
  \end{pmatrix} \\
  & = \frac{1}{2n} \frac{1}{1 - (\frac{R_1}{R_2})^{2n}}
  \begin{pmatrix}
    (R_2^{-2n} - b^{-2n} ) (b^{2n} - R_1^{2n})  & (1 - R_2^{-2n}) (b^n - R_1^{2n} b^{-n})   \\[1mm]
    - (1 - R_2^{-2n})( b^n - R_1^{2n} b^{-n} )
    &  (1- R_2^{-2n}) (1- R_1^{2n})
  \end{pmatrix}.
\end{align*}

Now, motivated by the argument in \cite[Sec. 4.1.1]{DHHM},
we shall derive an explicit condition on $m$ to ensure the existence of $m$-fold symmetric doubly connected V-states in this case.
The equation \eqref{eq:Omega-eq} of $\Omega$ becomes
\begin{align*}
  \Omega^2 - (A_{n,b} + B_{n,b}) \Omega + A_{n,b} B_{n,b} + \frac{1}{4 n^2} \frac{(1 - R_2^{-2n})^2}{(1- R_1^{2n} R_2^{-2n})^2}
  b^{2n}\Big( 1 - \frac{R_1^{2n}}{b^{2n}}\Big)^2 =0 ,
\end{align*}
with
\begin{align}
  A_{n,b} & = -\frac{1}{b^2} \mathfrak{C}_b - \frac{1}{2n} \frac{1}{1- (R_1R_2^{-1})^{2n}}
  \big( \tfrac{1}{R_2^{2n}} - \tfrac{1}{b^{2n}} \big) (b^{2n} - R_1^{2n}) , \label{eq:Euler-ann-A-n}\\
  B_{n,b} & = \frac{1-b^2}{2} - \mathfrak{C}_b -  \frac{1}{2n} \frac{1}{1- (R_1 R_2^{-1})^{2n}}
  \big(1- \tfrac{1}{R_2^{2n}}\big) (1- R_1^{2n}); \label{eq:Euler-ann-B-n}
\end{align}
and the discriminant given by \eqref{discriminant} reduces to
\begin{equation}\label{eq:Euler-ann-discri}
\begin{split}
  \mathbf{\Delta}_{n,b} & = \Big(\frac{(1-b^2)(b^2 + 2 \mathfrak{C}_b)}{2 b^2}
  +\frac{1}{2n} \frac{1}{1-(R_1R_2^{-1})^{2n}} \Big(\tfrac{b^{2n} +1 -2 R_1^{2n}}{R_2^{2n}}
  + R_1^{2n} + \tfrac{R_1^{2n}}{b^{2n}}-2\Big)\Big)^2 \\
  & \quad - \frac{(1 - R_2^{-2n})^2}{(1- R_1^{2n} R_2^{-2n})^2}
  \frac{b^{2n}}{n^2} \Big( 1 - \frac{R_1^{2n}}{b^{2n}}\Big)^2 .
\end{split}
\end{equation}
Due to the fact that $1- \frac{R_1^{2n}}{b^{2n}}>0$ for every $b\in (R_1,1)$, the inequality $\mathbf{\Delta}_{n,b}>0$ is equivalent to
\begin{align}\label{eq:Delta>0cas1}
  \frac{(1-b^2)(b^2 + 2 \mathfrak{C}_b)}{2 b^2}
  +\frac{1}{2n} \tfrac{1}{1- \frac{R_1^{2n}}{R_2^{2n}}} \Big(\tfrac{b^{2n} +1 -2 R_1^{2n}}{R_2^{2n}}
  + R_1^{2n} + \tfrac{R_1^{2n}}{b^{2n}}-2\Big) - \frac{1 - R_2^{-2n}}{1- \frac{R_1^{2n}}{ R_2^{2n}}}
  \frac{b^n}{n} \Big( 1 - \tfrac{R_1^{2n}}{b^{2n}}\Big)>0,
\end{align}
or
\begin{align}\label{eq:Delta>0cas2}
  \frac{(1-b^2)(b^2 + 2 \mathfrak{C}_b)}{2 b^2}
  + \frac{1}{2n} \tfrac{1}{1- \frac{R_1^{2n}}{R_2^{2n}}} \Big(\tfrac{b^{2n} +1 -2 R_1^{2n}}{R_2^{2n}}
  + R_1^{2n} + \tfrac{R_1^{2n}}{b^{2n}}-2\Big) + \tfrac{1 - R_2^{-2n}}{1- \frac{R_1^{2n}}{ R_2^{2n}}}
  \frac{b^n}{n} \Big( 1 - \frac{R_1^{2n}}{b^{2n}}\Big)<0.
\end{align}
However, the second case \eqref{eq:Delta>0cas2} cannot happen if $n\geqslant 2$. In fact, noticing that
\begin{align*}
  & \frac{1}{2n} \frac{1}{1- \frac{R_1^{2n}}{R_2^{2n}}} \Big(\tfrac{b^{2n} +1 -2 R_1^{2n}}{R_2^{2n}}
  + R_1^{2n} + \tfrac{R_1^{2n}}{b^{2n}}-2\Big) + \frac{1 - R_2^{-2n}}{1- \frac{R_1^{2n}}{ R_2^{2n}}}
  \frac{b^n}{n} \Big( 1 - \frac{R_1^{2n}}{b^{2n}}\Big) \\
  & \qquad= \frac{1}{2n} \frac{1}{1- \frac{R_1^{2n}}{R_2^{2n}}}\Big(\frac{(b^n-1)^2+2R_1^{2n}(b^{-n}-1)}{R_2^{2n}}+R_1^{2n}(1-b^{-n})^2-2(1-b^n)\Big) \\
  & \qquad= \frac{1}{2n} \frac{1-b^{n}}{1- \frac{R_1^{2n}}{R_2^{2n}}}\Big(\frac{(1-b^n)+2R_1^{2n}b^{-n}}{R_2^{2n}}+R_1^{2n}b^{-2n}(1-b^{n})-2\Big),
\end{align*}
we find that for every $n\geqslant 2$ and $0<R_1<b<1<R_2$,
\begin{align*}
  & \textrm{L.H.S.  of} \;\eqref{eq:Delta>0cas2} \\
  & = \frac{1-b^2}{2} + \frac{1-b^2}{b^2}\mathfrak{C}_b -\frac{1-b^n}{n} +
  \frac{1-b^n}{2n(1- \frac{R_1^{2n}}{R_2^{2n}})}\Big(\frac{(1-b^n)+2R_1^{2n}(b^{-n}-1)}{R_2^{2n}} + R_1^{2n}b^{-2n}(1-b^{n})\Big) \\
  &  >0,
\end{align*}
where the last inequality is ensured by the fact that $\frac{1-b^n}{n} \leqslant \frac{1-b^2}{2}$ for every $n\geqslant 2$ and $b\in(0,1)$.
\\[0.5mm]
Note that \eqref{eq:Delta>0cas1} is also equivalent to the following condition that
\begin{align}\label{eq:Delta>0equa}
  n > \frac{b^2}{(1-b^2)(b^2 + 2\mathfrak{C}_b)} \frac{1}{1- (\frac{R_1}{R_2})^{2n}}
  \Big(2- R_1^{2n} - \tfrac{R_1^{2n}}{b^{2n}} -\tfrac{b^{2n} +1 - 2R_1^{2n}}{R_2^{2n}}
  +  2(1- R_2^{-2n}) b^n \big(1- \tfrac{R_1^{2n}}{b^{2n}}\big) \Big) .
\end{align}
Next, we plan to show that if $m_0\geqslant 2$ satisfies \eqref{eq:Delta>0equa},
then each $m\geqslant m_0$ satisfies the condition \eqref{eq:Delta>0equa} too.
By introducing
\begin{align*}
  x\in \mathbb{R}^+\mapsto f_t(x)
  & \triangleq \frac{1-t^{2x}}{x},\quad t\in (0,1), \\
  x\in \mathbb{R}^+\mapsto g_{t_1,t_2}(x) & \triangleq (1-t_1^{2x})(1-t_2^{2x}) = 1- t_1^{2x} - t_2^{2x} + (t_1 t_2)^{2x},
  \quad t_1,t_2\in (0,1),
\end{align*}
we extend the mapping $n\in \mathbb{N}^{\star}\mapsto \mathbf{\Delta}_{n,b}$ as
\begin{equation*}
\begin{split}
  x\in \mathbb{R}^+\mapsto \mathbf{\Delta}_{x,b}
  \triangleq \frac{1}{4} \Pi_{x,b}^2  - \mathsf{E}_{x,b}^2,
\end{split}
\end{equation*}
with
\begin{align}
  \Pi_{x,b} & \triangleq \frac{(1-b^2)(b^2 +2\mathfrak{C}_b)}{b^2}
  - \frac{1}{f_{R_1/R_2}(x)\,x^2}\Big(g_{\frac{b}{R_2},\frac{R_1}{b}}(x)+g_{\frac{1}{R_2},R_1}(x)\Big) \label{def:ann-Pixb}
  \end{align}
  and
  \begin{align*}
  \mathsf{E}_{x,b} & \triangleq\frac{b^{x}}{x^2f_{R_1/R_2}(x)}g_{\frac{1}{R_2},\frac{R_1}{b}}(x). 
\end{align*}
Notice that for all $t\in (0,1)$, the function $x\in [2,+\infty)\mapsto f_t(x)$
is strictly decreasing (see \cite[Eq. (58)]{DHHM}), and
\begin{align*}
  f_t'(x) = \frac{t^{2x} \big(\log \frac{1}{t^{2x}} +1\big) -1}{x^2} = \frac{t^{2x} \log \frac{1}{t^{2x}}}{x^2} - \frac{f_t(x)}{x}.
\end{align*}
which satisfies that
\begin{align}\label{eq:property-f-t}
  f_t'(x)<0,\qquad x f_t'(x)+f_{t}(x)>0,\quad \forall x\in [2,\infty), \quad \forall  t\in (0,1).
\end{align}
Direct computations show that for any $t_1,t_2\in (0,1)$, it holds that for all $x\in (0,\infty)$,
\begin{align*}
  g_{t_1,t_2}(x)>0,\quad g'_{t_1,t_2}(x)=-t_1^{2x}(\log t_1^{2})(1-t_2^{2x})-t_2^{2x}(\log t_2^{2})(1-t_1^{2x})>0,
\end{align*}
and
\begin{equation}\label{eq:part-g_t-diff-12}
\begin{split}
  g_{t_1,t_2}(x)-xg'_{t_1,t_2}(x)&=(1-t_2^{2x})\big(1-t_1^{2x}+t_1^{2x}\log t_1^{2x}\big)\\
  &\quad \, +(1-t_1^{2x})\big(1-t_2^{2x}+t_2^{2x}\log t_2^{2x}\big)>0,
\end{split}
\end{equation}
where we have used the that $1-r+r\log r>0$ for all $r\in (0,1)$. From \eqref{def:ann-Pixb}, \eqref{eq:property-f-t}
and \eqref{eq:part-g_t-diff-12}, we find that
\begin{align*}
  \partial_{x}\Pi_{x,b} & = \partial_{x}\Big(- \frac{1}{f_{R_1/R_2}(x)\,x^2}\Big)\Big(g_{\frac{b}{R_2},\frac{R_1}{b}}(x)
  + g_{\frac{1}{R_2},R_1}(x)\Big)
  - \frac{1}{f_{R_1/R_2}(x)\,x^3}\Big(x g'_{\frac{b}{R_2},\frac{R_1}{b}}(x)
  + x g'_{\frac{1}{R_2},R_1}(x)\Big)\\
  & = \Big(\partial_{x}\Big(- \frac{1}{f_{R_1/R_2}(x)\,x^2}\Big)- \frac{1}{f_{R_1/R_2}(x)\,x^3}\Big)
  \Big(g_{\frac{b}{R_2},\frac{R_1}{b}}(x) + g_{\frac{1}{R_2},R_1}(x)\Big)\\
  &\quad+\frac{1}{f_{R_1/R_2}(x)\,x^3}\Big(g_{\frac{b}{R_2},\frac{R_1}{b}}(x)
  - x g'_{\frac{b}{R_2},\frac{R_1}{b}}(x) + g_{\frac{1}{R_2},R_1}(x) -x g'_{\frac{1}{R_2},R_1}(x)\Big)\\
  &=
  \Big(\frac{1}{x}\partial_{x}\Big(-\frac{1}{f_{R_1/R_2}(x)\,x}\Big)\Big)
  \Big(g_{\frac{b}{R_2},\frac{R_1}{b}}(x)+g_{\frac{1}{R_2},R_1}(x)\Big) \\
  & \quad +\frac{1}{f_{R_1/R_2}(x)\,x^3}\Big(g_{\frac{b}{R_2},\frac{R_1}{b}}(x)
  - x g'_{\frac{b}{R_2},\frac{R_1}{b}}(x) + g_{\frac{1}{R_2},R_1}(x) -x g'_{\frac{1}{R_2},R_1}(x)\Big),
\end{align*}
and thus
\begin{align*}
  \partial_x \Pi_{x,b} > \frac{f_{R_1/R_2}(x)+xf'_{R_1/R_2}(x)}{x^3f^2_{R_1/R_2}(x)}
  \Big(g_{\frac{b}{R_2},\frac{R_1}{b}}(x)+g_{\frac{1}{R_2},R_1}(x)\Big) >0.
\end{align*}
Similarly, we also have
\begin{align*}
  - \partial_x \mathsf{E}_{x,b}
  & = -b^{x}\log b \Big(\frac{g_{\frac{1}{R_2},\frac{R_1}{b}}(x)}{x^2f_{R_1/R_2}(x)}\Big)
  - b^{x}\partial_x\Big(\frac{g_{\frac{1}{R_2},\frac{R_1}{b}}(x)}{x^2f_{R_1/R_2}(x)}\Big) \\
  & > - b^{x}\partial_x\Big(\frac{g_{\frac{1}{R_2},\frac{R_1}{b}}(x)}{x^2f_{R_1/R_2}(x)}\Big)\\
  &  = \frac{ b^x }{x^3f_{R_1/R_2}(x)} \Big(g_{\frac{1}{R_2},\frac{R_1}{b}}(x)
  -xg'_{\frac{1}{R_2},\frac{R_1}{b}}(x) \Big) \\
  &\quad+b^{x}g_{\frac{1}{R_2},\frac{R_1}{b}}(x)\Big(\partial_{x}
  \Big(-\frac{1}{x^2f_{R_1/R_2}(x)}\Big)-\frac{1}{x^{3}f_{R_1/R_2}(x)}\Big)\\
  & \geqslant b^{x}g_{\frac{1}{R_2},\frac{R_1}{b}}(x)\frac{f_{R_1/R_2}(x)
  + xf'_{R_1/R_2}(x)}{x^{3}f^2_{R_1/R_2}(x)}>0.
\end{align*}
In addition, arguing as the above analysis for $\mathbf{\Delta}_{n,b}$ in \eqref{eq:Delta>0cas1},
the condition $\mathbf{\Delta}_{x,b}>0$ for every $x\geqslant 2$ is equivalent to
\begin{align*}
  \frac{1}{2} \Pi_{x,b}  -  \mathsf{E}_{x,b}>0.
\end{align*}
Hence, if $\mathbf{\Delta}_{\tilde{x},b}>0$ for some $\tilde{x}\geqslant 2$, then
\begin{align*}
  \partial_{x}\mathbf{\Delta}_{x,b}\Big|_{x=\tilde{x}} = \Big(\frac{1}{2}\Pi_{x,b}\,\partial_{x}\Pi_{x,b}
  - 2 \mathsf{E}_{x,b} \,
  \partial_{x}\mathsf{E}_{x,b} \Big)\Big|_{x=\tilde{x}}
  > \mathsf{E}_{\tilde{x},b} \Big(\partial_x\Pi_{x,b}
  + 2\partial_x\big(-  \mathsf{E}_{x,b}\big)\Big)\Big|_{x=\tilde{x}}>0,
\end{align*}
and as a result, we conclude that the mapping $x\in [\tilde{x},+\infty)\mapsto \mathbf{\Delta}_{x,b}$ is positive and strictly increasing,
which leads to the desired result. \\
Finally, we show that the mapping
$n\in [m_0,+\infty)\cap \mathbb{N}\mapsto \Omega^{+}_{n,b} =\frac{A_{n,b}+B_{n,b}}{2}+\frac{\sqrt{\mathbf{\Delta}_{n,b}}}{2}$
is strictly increasing and $n\in [m_0,+\infty)\cap \mathbb{N}\mapsto \Omega^{-}_{n,b}
  = \frac{A_{n,b}+B_{n,b}}{2}-\frac{\sqrt{\mathbf{\Delta}_{n,b}}}{2}$
is strictly decreasing for any $m_0\geqslant 2$ that satisfies \eqref{eq:Delta>0cas1}.
As above, we extend the mapping $n\mapsto \Omega^{\pm}_{n,b}$ as
\begin{align*}
  x\mapsto \Omega^{\pm}_{x,b}
  & \triangleq  \frac{1}{2}\Big(\frac{1-b^2}{2} - \frac{1+b^2}{b^2}\mathfrak{C}_b
  - \frac{1}{2x^2 f_{R_1/R_2}(x)}\big(g_{\frac{1}{R_2},R_1}(x)-g_{\frac{b}{R_2},\frac{R_1}{b}}(x)\big)\Big)
  \pm \frac{\sqrt{\mathbf{\Delta}_{x,b}}}{2}.
\end{align*}
Through straightforward calculations, we see that
\begin{align*}
  \partial_{x}\Omega_{x,b}^{\pm }=\frac{1}{4}\Big(\mathsf{A}_1-\mathsf{A}_2
  + \mathsf{B}_1-\mathsf{B}_2\pm \frac{\partial_{x}\mathbf{\Delta}_{x,b}}{\sqrt{\mathbf{\Delta}_{x,b}}}\Big),
\end{align*}
where
\begin{align*}
  \mathsf{A}_1 & \triangleq\frac{1}{x}\partial_{x}\Big(-\frac{1}{xf_{R_1/R_2}(x)}\Big)g_{\frac{1}{R_2},R_1}(x)>0,\\
  \mathsf{A}_2 & \triangleq\frac{1}{x}\partial_{x}\Big(-\frac{1}{xf_{R_1/R_2}(x)}\Big)g_{\frac{b}{R_2},\frac{R_1}{b}}(x)>0,\\
  \mathsf{B}_1 & \triangleq \frac{1}{x^{3}f_{R_1/R_2}(x)}\Big(g_{\frac{1}{R_2},R_1}(x)-x g'_{\frac{1}{R_2},R_1}(x)\Big)>0\\
  \mathsf{B}_2 & \triangleq \frac{1}{x^{3}f_{R_1/R_2}(x)}\Big(g_{\frac{b}{R_2},\frac{R_1}{b}}(x)
  - x g'_{\frac{b}{R_2},\frac{R_1}{b}}(x)\Big)>0.
\end{align*}
In addition, recalling our computation for $\partial_{x}\mathbf{\Delta}_{x,b}$,
we infer that
\begin{align*}
  \frac{\partial_{x}\mathbf{\Delta}_{x,b}}{\sqrt{\mathbf{\Delta}_{x,b}}}
  & = \frac{\frac{1}{2}\Pi_{x,b}\,\partial_x \Pi_{x,b}-2\mathsf{E}_{x,b}\,\partial_x\mathsf{E}_{x,b} }{\sqrt{\mathbf{\Delta}_{x,b}}} \\	
  & > \frac{\frac{1}{2}\Pi_{x,b}}{\sqrt{\mathbf{\Delta}_{x,b}}}\partial_{x}\Pi_{x,b}
  = \frac{\frac{1}{2}\Pi_{x,b}}{\sqrt{\mathbf{\Delta}_{x,b}}} \Big(\mathsf{A}_1+\mathsf{A}_2+\mathsf{B}_1+\mathsf{B}_2\Big).
\end{align*}
Furthermore, if $m_0\geqslant 2$ is an integer satisfies $\mathbf{\Delta}_{m_0,b}>0$, we can deduce from \eqref{eq:Delta>0cas1} that
\begin{align*}
  \frac{1}{2}\Pi_{x,b}>\sqrt{\mathbf{\Delta}_{x,b}}>0 ,\quad \forall x\in [m_0,+\infty),
\end{align*}
which implies that $\frac{\frac{1}{2}\Pi_{x,b}}{\sqrt{\mathbf{\Delta}_{x,b}}}>1$ for all $x\in [m_0,+\infty)$.
Thus, based on the above analysis, we have that for all $x\in [m_0,+\infty)$,
\begin{align*}
  \partial_{x}\Omega_{x,b}^{+}>\frac{1}{4}\Big(2\mathsf{A}_1+2\mathsf{B}_1\Big)>0,\quad
  \partial_{x}\Omega_{x,b}^{-}<\frac{1}{4}\Big(-2\mathsf{A}_2-2\mathsf{B}_2\Big)<0.
\end{align*}
Therefore, once $m_0\geqslant 2$ satisfies \eqref{eq:Delta>0equa} (that is, $\mathbf{\Delta}_{m_0,b}>0$),
then for any $m\geqslant m_0$, both conditions \eqref{eq:m-fold-condition}-\eqref{eq:m-fold-condition2}
are satisfied and consequently, by Theorem \ref{thm:general-statement},
there exist two families of $m$-fold symmetric doubly connected V-states bifurcating from
$\mathbf{1}_{\mathbb{D}\setminus \overline{b\mathbb{D}}}(\mathbf{x})$
at two different angular velocities $\Omega_{m,b}^{\pm}$.

\subsection{2D Euler equation in the exterior of a disc.}\label{subsec:Euler-ext-disc}
Consider the 2D Euler equation in vorticity form in the exterior of a disc
$\mathbb{R}^2\setminus\overline{R\mathbb{D}} = \{\mathbf{z}\,:\, |\mathbf{z}|>R \}$ ($R\in (0,b)$)
with the slip boundary condition $u\cdot \mathbf{n}=0$ on $\partial (R\mathbb{D})$,
where $\mathbf{n}$ is the outer normal unit vector.
It is the active scalar equation \eqref{eq:ASE}-\eqref{eq:stream_function}
in $\mathbf{D}= \mathbb{R}^2\setminus\overline{R\mathbb{D}} $,
and the kernel function $K(\mathbf{x},\mathbf{y})$ is given by
\begin{align}\label{eq:K-ext-disc}
  K(\mathbf{x},\mathbf{y})=-\frac{1}{2\pi}\log |\mathbf{x}-\mathbf{y}|+\frac{1}{2\pi}\log |R-R^{-1}\mathbf{x}\overline{\mathbf{y}}|,
  \quad \mathbf{x},\mathbf{y}\in \mathbf{D},
\end{align}
which uses an identical formula to that of the disc case presented in \eqref{eq:Green_Euler}.
This Green function $K(\mathbf{x},\mathbf{y})$ can be derived using the reflection method as detailed in \cite{Evans10};
on the other hand, using the Taylor's expansion $-\log (1-\mathbf{z}) = \sum_{m=1}^\infty \frac{1}{m} \mathbf{z}^m$ for every
$|\mathbf{z}|<1$, it can also be recovered from \eqref{eq:K-Euler-annulus} by passing $R_2\rightarrow \infty$ and setting $R_1=R$:
\begin{align}\label{eq:K-extdom-der}
  K(\mathbf{x},\mathbf{y}) & = - \frac{1}{2\pi} \log |\mathbf{x} -\mathbf{y}| +
  \frac{1}{2\pi} \Big(- \log \frac{R}{|\mathbf{y}|} + \log |\mathbf{x}| - \sum_{m=1}^\infty \frac{1}{m}
  R^{2m} |\mathbf{x}|^{-m} |\mathbf{y}|^{-m} \cos m(\theta -\eta)\Big) \nonumber \\
  & = -\frac{1}{2\pi} \log |\mathbf{x} -\mathbf{y}|
  + \frac{1}{2\pi} \Big( \log \frac{|\mathbf{x}| |\mathbf{y}|}{R}
  - \mathrm{Re} \sum_{m=1}^\infty \frac{1}{m} \big(R^2 |\mathbf{x}|^{-1} |\mathbf{y}|^{-1} e^{-i (\theta -\eta)}\big)^m \Big) \nonumber \\
  & = - \frac{1}{2\pi} \log |\mathbf{x} -\mathbf{y}|
  + \frac{1}{2\pi} \Big( \log \frac{|\mathbf{x}| |\mathbf{y}|}{R}  - \mathrm{Re} \sum_{m=1}^\infty \frac{1}{m}
  \Big(\frac{R^2}{\mathbf{x} \overline{\mathbf{y}}} \Big)^m\Big) \nonumber \\
  & = - \frac{1}{2\pi} \log |\mathbf{x} -\mathbf{y}|
  + \frac{1}{2\pi} \Big( \log \frac{|\mathbf{x}| |\mathbf{y}|}{R}
  + \log \Big| 1 - \frac{R^2}{\mathbf{x} \overline{\mathbf{y}}}\Big| \Big) \nonumber \\
  & = -\frac{1}{2\pi}\log |\mathbf{x}-\mathbf{y}|+\frac{1}{2\pi}\log |R-R^{-1}\mathbf{x}\overline{\mathbf{y}}|.
\end{align}
Obviously, the kernel function
$(\mathbf{x},\mathbf{y})\in (\mathbb{R}^2\setminus \overline{R\mathbb{D}})\times (\mathbb{R}^2\setminus \overline{R\mathbb{D}})
\mapsto K_1(\mathbf{x},\mathbf{y})=\tfrac{1}{2\pi}\log |R-R^{-1}\mathbf{x}\overline{\mathbf{y}}|$
is smooth and satisfies assumptions \eqref{A3}-\eqref{A4}.
Hence, we can apply Theorem \ref{thm:general-statement} to this case.
\\
Since  $K(\mathbf{x},\mathbf{y})$ can be obtained from \eqref{eq:K-Euler-annulus} by letting $R_2\rightarrow \infty$ and setting $R_1=R$, then
both $V_b[\mathbf{0}]$ and $\mathbf{\Lambda}_{n,b}$ can be derived from their counterparts in Section
\ref{subsec:Euler-annulus} using this identical approach:
\begin{align*}
  V_b[\mathbf{0}]=
  \begin{pmatrix}
     V_b^1[\mathbf{0}] & 0\\
     0 & V_b^2[\mathbf{0}]
  \end{pmatrix}
  =
  \begin{pmatrix}
     \frac{1-b^2}{2b^2} & 0\\
     0 & 0
  \end{pmatrix},
  \quad  V_b^2[\mathbf{0}] = 0 <  \frac{1-b^2}{2b^2} = V_b^1[\mathbf{0}],
\end{align*}
and
\begin{align*}
  \mathbf{\Lambda}_{n,b}=\frac{1}{2n}
  \begin{pmatrix}
    (\frac{R}{b})^{2n} - 1 & b^n \big( 1 - (\frac{R}{b})^{2n}\big) \\[1mm]
    - b^n \big(1- (\frac{R}{b})^{2n}\big)  & 1 - R^{2n}
  \end{pmatrix}.
\end{align*}
We note that by directly using \eqref{eq:K-ext-disc},
these formulas can also be obtained from the definitions,
similar to those in Section \ref{subsec:Euler-rad}.
\\
At last, as in Section \ref{subsec:Euler-annulus},
we can show that, for any $m\geqslant 2$ such that
\begin{align}\label{eq:Euler-exterior-cond-1}
  \mathbf{\Delta}_{m,b} = \Big(\frac{1-b^2}{2 b^2}
  +\frac{1}{2m} \Big(R^{2m} + \tfrac{R^{2m}}{b^{2m}}-2\Big)\Big)^2
   -  \frac{b^{2m}}{m^2} \Big( 1 - \frac{R^{2m}}{b^{2m}}\Big)^2  >0,
\end{align}
there exist two families of doubly connected $m$-fold symmetric V-states bifurcating from the trivial solution
$\mathbf{1}_{\mathbb{D}\setminus \overline{b\mathbb{D}}}(\mathbf{x})$ at two different angular velocities $\Omega^{\pm}_{m,b}$ given by
\begin{align*}
  \Omega^\pm_{m,b} = \frac{1}{2}\Big(-\frac{1-b^2}{2b^2} + \frac{1-(\frac{R}{b})^{2m}}{m} -\frac{1-R^{2m}}{m}
  \Big) \pm \frac{\sqrt{\mathbf{\Delta}_{m,b}}}{2}  .
\end{align*}
Furthermore, for $m\geqslant 2$, the condition \eqref{eq:Euler-exterior-cond-1} is equivalent to
\begin{align}\label{eq:Euler-exterior-cond-2}
	m > \frac{b^2}{1-b^2} \Big(2- R^{2m}+ (\tfrac{R}{b})^{2m} +2b^m \big( 1-  (\tfrac{R}{b})^{2m} \big)\Big).
\end{align}

\section{More discussion}\label{sec:conclusion_discuss}
In this section, we first examine the possible signs of $V_b^1[\mathbf{0}]$, $V_b^2[\mathbf{0}],$
as well as  the spectra $\Omega_{n,b}^{\pm}$ given by \eqref{eq:Omega_n_pm}. Due to their intricate structure, the analysis will focus on large $n$.
Next, we will conduct a brief study of the linear stability of the annulus considering the influence of the kernel structure.

\subsection{The sign of $V_b^1[\mathbf{0}]$, $V_b^2[\mathbf{0}]$ and $\Omega_{n,b}^{\pm}$}\label{subsec:V_b-sign}
In this subsection, we will show that even for the convolution case,
both scenarios described  in Proposition \ref{prop:spectral} can occur.
\\[0.5mm]
\noindent First, according to \eqref{eq:Anb-conver}-\eqref{eq:Bnb-conver} and the estimates in the proof of Proposition \ref{prop:spectral},
we find that
\begin{align*}
  \lim_{n\rightarrow\infty} \Omega_{n,b}^+ = - \min\{V_b^1[\mathbf{0}], V_b^2[\mathbf{0}]\},
  \quad\textrm{and} \quad
  \lim_{n\rightarrow\infty} \Omega_{n,b}^- = - \max\{V_b^1[\mathbf{0}], V_b^2[\mathbf{0}]\}.
\end{align*}
Thus, by continuity, we can summarize the sign of $\Omega_{n,b}^\pm$ with sufficiently large $n$ in the following result.
Note that in view of Lemma \ref{lem:B_infinity} and Remark \ref{rem:B_infinity},
the case $ V_b^2[\mathbf{0}] > 0$ is unlikely to occur, as suggested by numerical experiments.
\begin{proposition}\label{prop:differet_branch}
Under the assumption of Proposition \ref{prop:spectral}, the following statements are true.
\begin{enumerate}
\item For the case $V_b^2[\mathbf{0}] <0 < V_b^1[\mathbf{0}]$, there exists a sufficiently large $N^{\ast}\in\mathbb{N}^\star$
such that for any $n\geqslant N^{\ast}$,
the frequency $\Omega^+_{n,b}>0$ and $\Omega^-_{n,b}<0$.
\item For the case $ V_b^2[\mathbf{0}] < V_b^1[\mathbf{0}] < 0$, there exists a sufficiently large $N^{\ast}\in\mathbb{N}^\star$
such that for any $n\geqslant N^{\ast}$,
the frequency $\Omega^+_{n,b}>0$ and $\Omega^-_{n,b}>0$.
\item For the case $V_b^1[\mathbf{0}]<V_b^2[\mathbf{0}]<0$, there exists a sufficiently large $N^{\ast}\in\mathbb{N}^\star$
such that for any $n\geqslant N^{\ast}$,
the frequency $\Omega^+_{n,b}>0$ and $\Omega^-_{n,b}>0$.
\end{enumerate}
\end{proposition}

\begin{remark}
Interestingly, for the whole plane or disc or annulus cases, all applications to geophysical flows in Section \ref{sec:applications} satisfy
$ V_b^1[\mathbf{0}] > V_b^2[\mathbf{0}]$ for all $b\in \mathcal{S}=\mathcal{S}_{\mathrm{max}}$.
We list these results in Table \ref{tab:A_B_infinity},
and one can find that case (i) of Proposition \ref{prop:differet_branch} includes several examples.
\begin{table}[hbpt]
\begin{tabular}{ccccccc}
\toprule
  & \multicolumn{2}{c}{Whole space}& \multicolumn{2}{c}{Disc\footnotemark} & \multicolumn{2}{c}{Annulus} \\
  \cmidrule{2-7}
  & $V_b^1[\mathbf{0}]$\;\; & $V_b^2[\mathbf{0}]$\;\; & $V_b^1[\mathbf{0}]$\;\; & $V_b^2[\mathbf{0}]$ \;\; & $V_b^1[\mathbf{0}]$\;\; & $V_b^2[\mathbf{0}]$ \\
  \midrule
  Euler & $0$ & $-$ & $0$ &  $-$ & $+$ & $-\,/\,0$ \\
  gSQG & $+$ &  $-$  & ${+}$  &  $-$ & $*$ & $*$ \\
  QGSW &  $+$  & $-$  &  ${+}$  & $-$ & $*$ & $*$  \\
\bottomrule
\end{tabular}
\vskip1mm
\caption{The sign of $V_b^1[\mathbf{0}]$ and $V_b^2[\mathbf{0}]$ for any $b\in \mathcal{S}_{\mathrm{max}}$.}\label{tab:A_B_infinity}
\end{table}
\end{remark}
\footnotetext{The sign of $V_b^{1}[\mathbf{0}]$ for the gSQG equation and QGSW equation in radial domains $\mathbf{D}=R \mathbb{D}$
($R>1$) are strongly suggested by numerical experiments for $R>1$ not so large,
but they are not rigorously justified for all $R>1$ so far.}


Now, for $K(\mathbf{x},\mathbf{y})=K_0(|\mathbf{x}-\mathbf{y}|)$ satisfying assumptions \eqref{A1}-\eqref{A2},
we present some non-trivial examples of $K_0(\cdot)$ that verify cases (i) and (iii) of Proposition \ref{prop:differet_branch}.
\begin{proposition}\label{prop:measure_examples}
\begin{enumerate}
\item There exist positive constants $b_{\ast}>0$ and $x_\ast>0$ such that, for any positive integrable function
$f: (0,x_\ast)\rightarrow \mathbb{R}^+$ and any $C_0\in \RR$,
the following kernel function,
\begin{align}\label{eq:K0-exam1}
  K_0(t)= \int_{t}^1 \int_0^\infty e^{-\tau x}f(x) \mathbf{1}_{(0,x_\ast)} \dd x\dd \tau+ C_0,\quad  t>0,
\end{align}
equivalently, $\mathrm{d} \mu(x) = f(x) \mathbf{1}_{(0,x_\ast)} \dd x$ in \eqref{eq:K0prim},
satisfies the conditions \eqref{A1}-\eqref{A2} and
\begin{align*}
  V_b^1[\mathbf{0}] < V_b^2[\mathbf{0}] < 0,\quad \forall\, b\in (0,b_{\ast}).
\end{align*}
\item For any $b\in (0,1)$, there exists a sufficiently large constant $x^\ast=x^\ast(b)>0$
such that for every positive function $f:(x^\ast,\infty)\rightarrow \mathbb{R}^+$ with the upper bound
\begin{align*}
  f(x)\leqslant Cx^{1-\gamma},\quad \,	\gamma>0,
\end{align*}
and any $C_0\in \RR$, the following kernel associated with $\mathrm{d} \mu(x) = f(x) \mathbf{1}_{(x^\ast,\infty)} \dd x$ in \eqref{eq:K0prim},
\begin{align}\label{eq:K0-exam2}
  K_0(t)=\int_t^1\int_0^{\infty} e^{-\tau x} f(x) \mathbf{1}_{(x^\ast,\infty)}  \dd x \dd \tau + C_0,\quad  t>0,
\end{align}

satisfies the conditions \eqref{A1}-\eqref{A2} and
\begin{align*}
  V_b^2[\mathbf{0}] < 0 < V_b^1[\mathbf{0}].
\end{align*}
\end{enumerate}
\end{proposition}

\begin{remark}
  The case $(ii)$ in Proposition \ref{prop:differet_branch} can also occur.
From Fig. \ref{fig:Psi-phi-compare}, there exists some interval $(\tilde{x}_\ast,\tilde{x}^\ast)\subset (0,\infty)$ such that the maps
$x\in (\tilde{x}_\ast,\tilde{x}^\ast) \mapsto \Psi_b(x)$ (defined by \eqref{def:Psi-bx}) and
$x\in (\tilde{x}_\ast,\tilde{x}^\ast)\mapsto \frac{1}{b}\phi_{1,b}(x)-\phi_1(bx)$
are both positive for $b=0.1$. In light of Lemma \ref{lem:B_infinity},
the map $x\in (\tilde{x}_\ast,\tilde{x}^\ast) \mapsto \phi_1(x)-b\phi_{1,b}(x)$ is always positive.
Then, by picking the non-negative measure $\dd \mu(x) = f(x)\mathbf{1}_{(\tilde{x}_\ast,\tilde{x}^\ast)}\dd x$ with the positive integrable function
$f(x): (\tilde{x}_\ast,\tilde{x}^\ast)\rightarrow \mathbb{R}^+$, it follows from \eqref{def:Psi-bx} and
\eqref{eq:V10-V20-exp} that
\begin{align*}
   V_b^2[\mathbf{0}] < V_b^1[\mathbf{0}] < 0, \quad b = 0.1.
\end{align*}
The associated kernel function $K_0(t)$ satisfies the assumptions \eqref{A1}-\eqref{A2} with $\alpha\in (0,1)$
as we have done in Proposition \ref{prop:measure_examples}.
\end{remark}

\begin{figure}[htbp]
  \centering
  \includegraphics[scale=0.7]{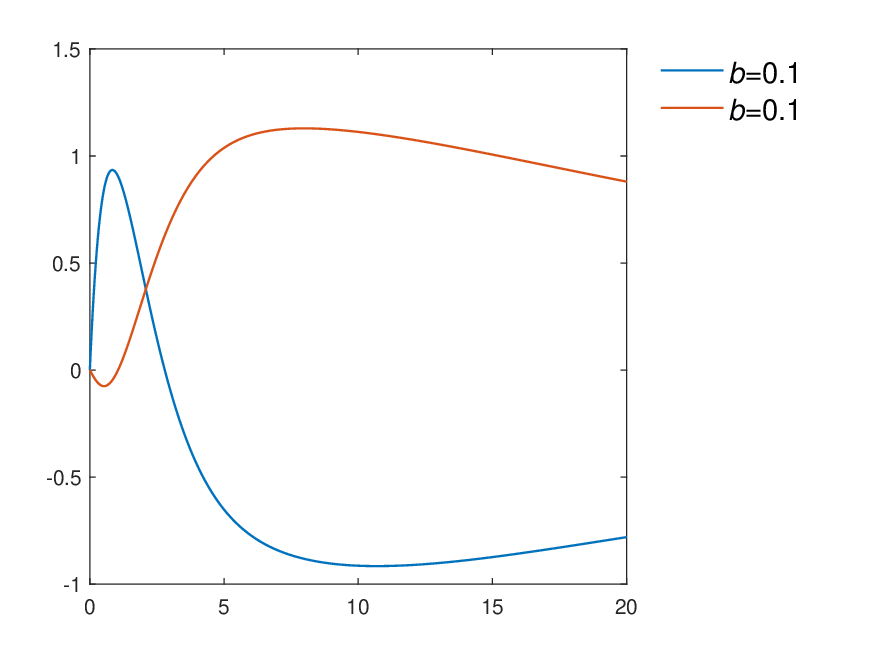} 
  \caption{The graph of $\color[rgb]{0,0.447,0.741}{x\mapsto \frac{1}{b}\phi_{1,b}(x)-\phi_1(bx)}$ and
  $\color[rgb]{0.85,0.325,0.098}{x\mapsto \Psi_b(x)}$ for $b=0.1$.}\label{fig:Psi-phi-compare}
\end{figure}

\begin{proof}[Proof of Proposition \ref{prop:measure_examples}]
Recall from \eqref{eq:V-0}, \eqref{def:lam-n-1} and \eqref{def:lam-n-3} that
\begin{align}\label{eq:V10-V20-exp}
  V_b^1[\mathbf{0}]=\int_0^{\infty}\big(\phi_{1}(bx)-\tfrac{1}{b}\phi_{1,b}(x)\big)\frac{\dd \mu(x)}{x},\quad
  V_b^2[\mathbf{0}]=\int_0^{\infty}\big(b\phi_{1,b}(x) - \phi_{1}(x)\big)\frac{\dd \mu(x)}{x}.
\end{align}
\textbf{(i)} In view of \eqref{eq:negative-argument} in Lemma \ref{lem:kernel_function_limit}, there exist small constants
$b_{\ast}>0$ and $x_{\ast}>0$,  such that by choosing the non-negative measure $\dd\mu(x)=f(x)\mathbf{1}_{(0,x_\ast)}\dd x$,
we have
\begin{align*}
  V_b^1[\mathbf{0}]<0, \quad V_b^2[\mathbf{0}]<0, \quad V_b^1[\mathbf{0}] - V_b^2[\mathbf{0}] <0.
\end{align*}
In addition, by the definition of $K'_0(t)$ given by \eqref{eq:K0prim}, we see that the associated kernel function $K_0(t)$ has
the expression formula \eqref{eq:K0-exam1}, which clearly satisfies \eqref{A1}.
Next, taking advantage of integration by parts gives that for $\alpha\in (0,1)$ and $K_0(1)=C_0$,
\begin{align*}
  & \int_0^1 |K_0(t) - C_0| t^{-\alpha + \alpha^2} \dd t = \Big|\int_0^1 (K_0(t) - C_0) t^{-\alpha + \alpha^2} \dd t\Big| \\
  &\quad =  \frac{1}{1-\alpha +\alpha^2} \bigg| \Big(t^{1-\alpha +\alpha^2} (K_0(t)-C_0)\Big)\Big|_{t=0}^1
  + \int_0^1 \Big(\int_0^{x_\ast}e^{-t x}f(x)\dd x\Big) t^{1-\alpha + \alpha^2}\dd t \bigg| \\
  & \quad \leqslant \frac{1}{1-\alpha+\alpha^2} \int_0^{x_\ast} f(x) \dd x <\infty,
\end{align*}
where in the last line we have used that
\begin{align*}
  \lim_{t\to 0}\Big|t^{1-\alpha+\alpha^2}(K_0(t)-C_0)\Big|
  & = \lim_{t\to 0}\Big|t^{1-\alpha+\alpha^2}\int_t^1\int_0^{\infty} e^{-\tau x} f(x) \mathbf{1}_{{(0,x_\ast)}}  \dd x \dd \tau\Big|\\
  & \leqslant \lim_{t\to 0}\Big|t^{1-\alpha+\alpha^2}\int_0^{\infty}f(x) \mathbf{1}_{{(0,x_\ast)}}  \dd x\Big|
  =0,\quad \forall \alpha\in (0,1).
\end{align*}
Thus, it immediately implies the integrability assumption \eqref{A2} with any $\alpha\in(0,1)$.

\noindent \textbf{(ii)}
For any $b\in(0,1)$, thanks to \eqref{eq:Psi-b-pos} in Lemma \ref{lem:kernel_function_limit},
there exists a sufficiently large constant $x^\ast >0$, such that by choosing the non-negative measure
$\dd \mu(x) = f(x)\mathbf{1}_{(x^\ast,\infty)}\dd x$, we have
\begin{align*}
  V_b^1[\mathbf{0}] >0,\quad  V_b^2[\mathbf{0}]<0,\quad V_b^1[\mathbf{0}] - V_b^2[\mathbf{0}] > 0.
\end{align*}
Then the associated kernel function $K_0(t)$ can be expressed as \eqref{eq:K0-exam2} and also verifies the assumption \eqref{A1}.
Under the asymptotic assumption of $f(x)$, we argue as in the above case (i) to deduce that for any $0<\gamma'<\min\{\gamma,1\}$,
\begin{align*}
  \int_{x^\ast}^{\infty}e^{-tx}f(x)\dd x & \leqslant C \int_{x^\ast}^{\infty}e^{-tx}x^{1-\gamma'}\dd x
  = Ct^{\gamma'-2} \int_{x^\ast t}^{\infty}e^{-x}x^{1-\gamma'}\dd x\\
  & \leqslant Ct^{\gamma'-2},
\end{align*}
and
\begin{align*}
  \int_0^1 |K_0(t)|t^{-\alpha + \alpha^2}\dd t & \leqslant \int_0^1 |K_0(t) - C_0| t^{-\alpha + \alpha^2} \dd t
  + |C_0| \int_0^1 t^{-\alpha +\alpha^2} \dd t \\
  & \leqslant
  \frac{1}{1-\alpha + \alpha^2} \int_0^1 \Big( \int_{x^\ast}^\infty e^{-tx} f(x) \dd x \Big) t^{1-\alpha + \alpha^2}\dd t
  + \frac{|C_0|}{1-\alpha +\alpha^2} \\
  & \leqslant C \int_0^1 t^{\gamma'-1-\alpha + \alpha^2} \dd t<\infty,\quad \forall 0< \alpha \leqslant \gamma'<1.
\end{align*}
Hence, $K_0(t)$ given by \eqref{eq:K0-exam2} satisfies assumptions \eqref{A1}-\eqref{A2}
with $0<\alpha<\min\{\gamma,1\}$.
\end{proof}

\subsection{Stable and unstable states}\label{subsec:stable-mode}
In this subsection, we clarify the meaning of the spectrum determined by \eqref{eq:Omega-eq}.
Let $t\in [0,T)\mapsto \mathbf{1}_{D_t}(\mathbf{x})$ with $T>0$ be a doubly connected patch solution close to
the stationary solution $\mathbf{1}_{\DD\setminus \overline{b\mathbb{D}}}$ for equation \eqref{eq:ASE}.
We parameterize the boundary $\partial D_t$ using the polar coordinates, as follows
\begin{align*}
  \mathbf{z}(t,\cdot):\mathbb{T}&\mapsto \partial D_t, \\
  \theta & \mapsto \mathbf{z}(t,\theta)\triangleq e^{it\Omega}
  \begin{pmatrix}
	\sqrt{b^2+2r_1(t,\theta)}e^{i\theta}\\
	\sqrt{1+2r_2(t,\theta)}e^{i\theta}
  \end{pmatrix},
\end{align*}
where $\mathbf{z}(t,\theta) = (z_1(t,\theta),z_2(t,\theta))^T$.
Following the same lines as the derivation of \eqref{eq:vortex-patch-eqs}, the contour dynamics equation of $\mathbf{z}(t,\cdot)$ reads
\begin{align}\label{eq:perturbed-patch-eqs}
  \partial_t r_j+ \partial_\theta \bigg(\int_{D_t}K(z_j(t,\theta),\mathbf{y})\dd \mathbf{y}\bigg)=0, \quad j=1,2.
\end{align}
In particular, choosing $r_j(\theta,t)=\widehat{r}_j(\theta+\Omega t)$ where $\widehat{r}_j(\cdot):\mathbb{T}\mapsto \RR$
is a smooth map, we can recover the doubly connected V-states equation \eqref{eq:vortex-patch-eqs}.
Denoting by
\begin{align*}
  F_0[\mathbf{r}](t,\theta)\triangleq
  \begin{pmatrix}
	\int_{D_t}K(z_1(t,\theta),\mathbf{y})\dd \mathbf{y} \\[1mm]
	\int_{D_t}K(z_2(t,\theta),\mathbf{y})\dd \mathbf{y}
  \end{pmatrix},
  \quad \mathbf{r}=
  \begin{pmatrix}
	r_1\\
	r_2
  \end{pmatrix},
\end{align*}
the equation \eqref{eq:perturbed-patch-eqs} can be written in the vector form:
\begin{align}\label{eq:perturbed-eq}
  \partial_t\mathbf{r}(t,\theta)+\partial_{\theta}F_0[\mathbf{r}](t,\theta)=0.
\end{align}
Arguing as obtaining \eqref{eq:linear-0} (with fixing $\Omega=0$),
for every $\mathbf{h}(\theta)=\sum\limits_{n\in \mathbb{Z}}
  \begin{pmatrix}
	h_{1,n}\\
	h_{2,n}
  \end{pmatrix}e^{in\theta}$,
we can get the linear operator,
\begin{align}\label{eq:linear_operator_F_0}
  \partial_{\mathbf{r}}\partial_{\theta}F_0[\mathbf{0}]\mathbf{h}
  =\sum_{n\in \mathbb{Z}\setminus \{0\}} i\,n\,(V_b[\mathbf{0}]+\mathbf{\Lambda}_{n,b})
  \begin{pmatrix}
	h_{1,n}\\
	h_{2,n}
  \end{pmatrix}
  e^{in\theta},
\end{align}
where $\mathbf{\Lambda}_{n,b}$ and $V_b[\mathbf{0}]$ are respectively given by
\eqref{eq:Lambda-n-b} and \eqref{eq:V-0}.
Comparing \eqref{eq:linear-0} with \eqref{eq:linear_operator_F_0},
we find that $-i\,n\,\Omega_{n,b}^{\pm}$ are the eigenvalues of the linear operator
$\partial_{\mathbf{r}}\partial_{\theta}F_0[\mathbf{0}]\mathbf{h}$,
where $\Omega_{n,b}^{\pm}$ are the solutions \mbox{of \eqref{eq:Omega-eq}.}
\\[0.5mm]
\noindent From the viewpoint of the linear stability of the dynamic system \eqref{eq:perturbed-eq} around equilibrium
$\mathbf{r}\equiv 0$, the case $\mathbf{\Delta}_{n,b}<0$
leads to some eigenvalues with positive real parts of the linear operator
$\partial_{\mathbf{r}}\partial_{\theta}F_0[\mathbf{0}]$,
which correspond to unstable states. We refer to \cite{Vishik18a,Vishik18b,ABCD+,CFMS25}
and the references therein for recent study on unstable annular type patches,
which play a key role in the proof of non-uniqueness for the 2D Euler and gSQG equations with forcing.
Conversely, if $\mathbf{\Delta}_{n,b}>0$, then $-i\,n\, \Omega_{n,b}^{\pm}$
are pure imaginary numbers which could be viewed as stable states.

\section{Appendix}\label{sec:appendix}
\noindent In this appendix, we collect some auxiliary results that will be used in this paper.
\\[0.5mm]
\noindent The first lemma states that the eigenvalues and eigenfunctions of the spectral Laplacian $-\Delta$
on the unit disc $\mathbb{D}\subset \mathbb{R}^2$ with Dirichlet boundary condition have explicit expression through Bessel functions
(e.g. see Section 5.5 of Chapter V in \cite{CH09} or \cite[pp. 127]{Borth20}).
\begin{lemma}\label{lem:disc-exp}
The eigenvalues and the eigenfunctions solving the spectral problem associated with the Laplacian $-\Delta$:
\begin{align}\label{def:phi-j2}
  \hbox{for}\; j\geqslant1,\,\;\;-\Delta\phi_j = \lambda_j \phi_j\quad \textrm{in}\;\,\mathbb{D},\quad \phi_j \big|_{\partial \mathbb{D}} =0,\quad
  \int_{\mathbb{D}}\phi_j^2(\mathbf{x}) \dd \mathbf{x}=1,
\end{align}
are described by the double-index families $(\lambda_{n,k})_{n\in \mathbb{N},k\in \mathbb{N}^\star}$
and $\big((\phi_{n,k}^{(1)}, \phi_{n,k}^{(2)})\big)_{n\in\mathbb{N},k\in\mathbb{N}^\star}$ such that
\begin{align}\label{def:lambd-phi-nk}
  \lambda_{n,k}=x_{n,k}^2, \quad \phi_{n,k}^{(1)}(\mathbf{x})=J_n(x_{n,k}|\mathbf{x}|)A_{n,k}\cos(n\theta),
  \quad \phi_{n,k}^{(2)}(\mathbf{x})=J_n(x_{n,k}|\mathbf{x}|)A_{n,k}\sin(n\theta),
\end{align}
where
\begin{align}\label{eq:Ank}
  \pi A_{0,k}^2=\frac{1}{J_{1}^2(x_{0,k})}\quad\hbox{and}\quad
  \pi A_{n,k}^2=\frac{2}{J_{n+1}^2(x_{n,k})},\quad\forall n\in\mathbb{N}^\star,
\end{align}
and $J_n(\cdot)$ denotes the Bessel function of order $n$  and $x_{n,k}$ is the $k$-th zero point of $J_n$.
\end{lemma}

\begin{remark}
Notice that, if $\phi(\mathbf{x})$ is a solution of the spectral problem \eqref{def:phi-j2},
then the function $\phi_R(\mathbf{x})=\frac{1}{R}\phi(\frac{\mathbf{x}}{R})$ satisfies that
\begin{align}\label{eq:eigen_func_scale}
  -\Delta \phi_R = \tfrac{\lambda}{R^2} \phi_R\quad \textrm{in}\;\, R\mathbb{D},\quad
  \phi_R\big|_{\partial (R\mathbb{D})} = 0,
  \quad \int_{R\mathbb{D}} \phi_R^2(\mathbf{x})\dd \mathbf{x}=1.
\end{align}
\end{remark}

Next, we introduce Sneddon's formula that can be found, for instance, in $(2.2.9)$ \mbox{in \cite{Sne66}}
or $(24)$-$(25)$ in \cite{Mar22}.
\begin{lemma}\label{lem:sned-form}
  Let $0<a,b\leqslant 1$, $n,\beta,\gamma\in \NN$ and $1<q <\beta + \gamma - 2n +2$. Then we have
\begin{equation}\label{eq:sned-form}
\begin{split}
  \sum_{k=1}^\infty \frac{J_\beta(a x_{n,k}) J_\gamma(b x_{n,k})}{x_{n,k}^q J_{n+1}^2(x_{n,k})}
  = \mathbb{J} + \frac{1}{\pi} \sin \big(\tfrac{\pi}{2}(\beta +\gamma -2n-q )\big)
  \int_0^\infty \rho^{1-q} \mathbf{I}_\beta(a\rho) \mathbf{I}_\gamma(b \rho)
  \frac{\mathbf{K}_n(\rho)}{\mathbf{I}_n(\rho)} \dd \rho,
\end{split}
\end{equation}
where
\begin{equation}\label{def:J}
\begin{split}
  \mathbb{J} & \triangleq  \frac{1}{\pi} \sin\big(\tfrac{\pi}{2}(\gamma-\beta + q)\big)
  \int_0^\infty \rho^{1-q} \mathbf{I}_\beta(a\rho) \mathbf{K}_\gamma(b\rho) \dd \rho  \\
  & = \frac{a^\beta \Gamma\big(1+\tfrac{\beta+\gamma -q}{2}\big)}{2^q b^{2+\beta-q}
  \Gamma(\beta+1)\Gamma(\tfrac{\gamma-\beta +q}{2})}
  F\big(1+\tfrac{\beta+\gamma -q}{2},1+\tfrac{\beta-\gamma -q}{2}; \beta+1;\tfrac{a^2}{b^2}\big),
\end{split}
\end{equation}
where $\mathbf{I}_n(\cdot)$ and $\mathbf{K}_n(\cdot)$ are the modified Bessel functions of the first and second kind,
and $F(a,b;c;\mathbf{z})$ denotes the hypergeometric function given by \eqref{eq:hypergeom}.
\end{lemma}

In order to sum up the series appearing in \eqref{eq:QGSW_radial_domain_1} and \eqref{eq:QGSW_radial_domain_2},
we derive an interesting summation identity by using formula (3.47) in \cite{Saharian07}
(we here reprove this formula to clarify that there is a missing coefficient in the original formula).
\begin{lemma}\label{lem:sum_up_QGSW}
For any $\varepsilon>0$, $k\in\mathbb{N}^\star$, and $0<Y\leqslant X<1$, the following identity holds:
\begin{align*}
  \sum_{k\geqslant 1}\frac{1}{x^2_{0,k}+\varepsilon^2}\frac{J_1(Xx_{0,k})J_1(Yx_{0,k})}{J^2_1(x_{0,k})}
  = \frac{1}{2} \frac{\mathbf{I}_1(Y\varepsilon)}{\mathbf{I}_0(\varepsilon)}
  \Big(\mathbf{I}_1(X\varepsilon)\mathbf{K}_0(\varepsilon)
  +\mathbf{I}_0(\varepsilon)\mathbf{K}_1(X\varepsilon)\Big).
\end{align*}
\end{lemma}

\begin{proof}[Proof of Lemma \ref{lem:sum_up_QGSW}]

Denote by $F(\mathbf{z})\triangleq \frac{\mathbf{z}}{\mathbf{z}^2+\varepsilon^2}J_1(Y\mathbf{z})$,
with $\mathbf{z}=x+yi$, $x,y\in \RR$. Recalling that the Bessel function
$J_1(\mathbf{z}) = \sum\limits_{m=0}^\infty \frac{(-1)^m (\frac{\mathbf{z}}{2})^{1+2m}}{m! (m+1)!} $
satisfies $J_1(\mathbf{z})=-J_1(-\mathbf{z})$ and
$J_1(\mathbf{z})=\frac{1}{2\pi}\int_{-\pi}^{\pi}e^{i\mathbf{z} \sin \theta-i\theta}\dd \theta$ (see 8.411 of \cite{GR15}),
it follows that
\begin{align*}
  F(xi)=\frac{ix J_1(xi)}{(xi)^2+\varepsilon^2}=\frac{-ix J_1(-xi)}{(-xi)^2+\varepsilon^2}=F(-xi),\quad \forall x\in \RR,
\end{align*}
and
\begin{align*}
  |F(\mathbf{z})|\leqslant C|\mathbf{z}|^{-1}e^{Y|y|},\quad \forall |\mathbf{z}|\gg 1,
\end{align*}
and $F(\mathbf{z})$ is analytic except two poles $\mathbf{z}=\pm \varepsilon i$.

\noindent Taking $A=1$, $B=0$, $\nu =0$, $m=1$, $t=X$ and
$F(\mathbf{z})=\frac{\mathbf{z}}{\mathbf{z}^2+\varepsilon^2}J_1(Y\mathbf{z})$
in Corollary 1 of \cite[pp. 17]{Saharian07}, it follows from (3.1), (3.8) and (3.47) in \cite{Saharian07} that
\begin{align}\label{eq:ids_residue_theorems}
  \sum_{k\geqslant 1}\frac{1}{x^2_{0,k}+\varepsilon^2}\frac{J_1(Xx_{0,k})J_1(Yx_{0,k})}{J^2_1(x_{0,k})}
  =\frac{\pi}{4}\sum_{\mathbf{z}=\pm \varepsilon i}\text{Res}\Big\{\Big(J_1(X\mathbf{z})\mathbf{Y}_0(\mathbf{z})
  -\mathbf{Y}_1(X\mathbf{z})J_0(\mathbf{z})\Big)
  \frac{F(\mathbf{z})}{J_0(\mathbf{z})},\mathbf{z}\Big\},
\end{align}
where $\mathbf{Y}_n(\cdot)$ is the Bessel function of the second kind.
Note that the coefficients on the right-hand side of \cite[Eq. (3.47)]{Saharian07}
should be slightly changed (there is a missing constant $\frac{i}{2}$).
Here, inspired by the remark below Corollary 1 in \cite{Saharian07},
we sketch the proof of \eqref{eq:ids_residue_theorems} to make this point clear.

\noindent First, we introduce the contour $C_{h,l,r}$ in Fig. \ref{fig:contour},
which is a rectangle with vertices $(0,-h)$, $(0,h)$, $(l,-h)$
and $(l,h)$ except that the point $(0,\pm\varepsilon)$ are bypassed by semicircles with radius $r$.

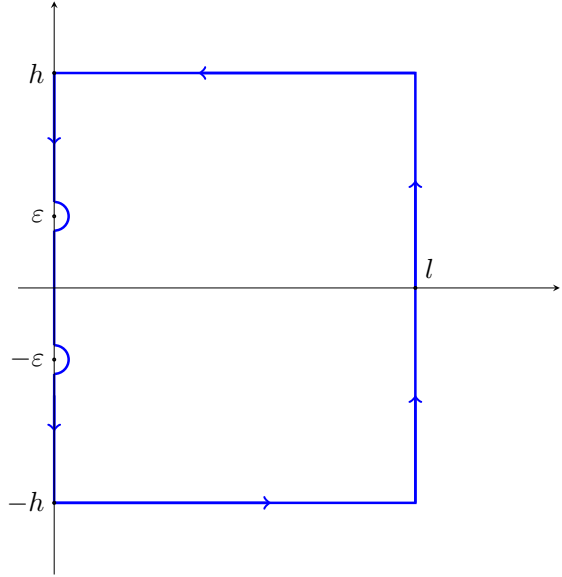
\begin{figure}[htbp]
   \scalebox{0.95}{\begin{tikzpicture}
	\draw[line width=1pt,blue] (0,-1.2)--(0,-3)--(5,-3)--(5,3)--(0,3)--(0,1.2);
	\draw[line width=1pt,blue] (0,-0.8)--(0,0.8);
	\draw[line width=1pt,blue,-to] (0,-3)--(3,-3);
	\draw[line width=1pt,blue,-to] (5,-3)--(5,-1.5); \draw[line width=1pt,blue,-to] (5,0)--(5,1.5);
	\draw[line width=1pt,blue,-to] (5,3)--(2,3);
	\draw[line width=1pt,blue,-to] (0,3)--(0,2);
	\draw[line width=1pt,blue,-to] (0,-1.5)--(0,-2);
	\draw[line width=1pt,blue] (0,0.8) arc (-90:90:0.2);
	\draw[line width=1pt,blue] (0,-1.2) arc (-90:90:0.2);
			
	\node[anchor=east] () at (0,1) {$\varepsilon$};
	\draw[line width=1pt,black] (0,1) circle [radius=0.01]; 	
	\draw[line width=1pt,black] (0,-1) circle [radius=0.01];
	\node[anchor=east] () at (0,-1) {$-\varepsilon$};
	\draw[line width=1pt,black] (0,-3) circle [radius=0.01];
	\node[anchor=east] () at (0,-3) {$-h$};
	\draw[line width=1pt,black] (0,3) circle [radius=0.01];
	\node[anchor=east] () at (0,3) {$h$};
	\draw[line width=1pt,black] (5,0) circle [radius=0.01];
	\node[anchor=south west] () at (5,0) {$l$};
		
	\draw[-stealth,line width=0.2pt] (-0.5,0) -- (7,0);
	\draw[-stealth,line width=0.2pt] (0,-4) -- (0,4);
	
		
%
%
	\end{tikzpicture}
    }
	\caption{The contour $C_{h,l,r}$.}\label{fig:contour}
\end{figure}
\noindent Denoting by
\begin{align*}
  F_1(\mathbf{z}) \triangleq \Big(J_1(X\mathbf{z}) \mathbf{Y}_0(\mathbf{z})- \mathbf{Y}_1(X\mathbf{z})J_0(\mathbf{z})\Big)
  \frac{F(\mathbf{z})}{J_0(\mathbf{z})},
\end{align*}
it follows from residue's theorem that
\begin{align}\label{eq:residue_ids}
  \frac{1}{2\pi i}\oint_{C_{h,j,r}}F_1(\mathbf{z})\dd \mathbf{z}=\sum_{
  k\geqslant 1, x_{0,k}<l}\text{Res}\Big\{F_1(\mathbf{z}),\mathbf{z}=x_{0,k}\Big\}.
\end{align}
Using the properties of Bessel functions and $F(\mathbf{z})$,
we find that the integrals along the segments of the imaginary axis cancel each other,
the integrals along the semicircle will converge to $-\frac{1}{2} \mathrm{Res}\{F_1(\mathbf{z}) \}$
at $\mathbf{z}=(0,\pm \varepsilon)=\pm \varepsilon i$
as $r$ goes to zero and the remaining three segments of $C_{h,l,r}$ will approach to zero as $h,l\to \infty$.
Hence, by taking $h,l\to \infty$ and $r\to 0^{+}$ in \eqref{eq:residue_ids}, we get that
\begin{align}\label{eq:poles_ids}
  -\frac{1}{2}\sum_{\mathbf{z}=\pm \varepsilon i}\text{Res}\Big\{F_1(\mathbf{z}),\mathbf{z}\Big\}
  = \sum^\infty_{k=1}\text{Res}\Big\{F_1(\mathbf{z}),\mathbf{z}=x_{0,k}\Big\}.
\end{align}
Since $J_{0}(\mathbf{z})$ has only simple zeros and $J_0'(\mathbf{z})=-J_1(\mathbf{z})$,
it follows from the fact (see e.g. \cite[Sec. 3.1]{MOS66})
\begin{align*}
  J_1(\mathbf{z})\mathbf{Y}_0(\mathbf{z})-\mathbf{Y}_1(\mathbf{z})J_0(\mathbf{z})=\frac{2}{\pi \mathbf{z}},
\end{align*}
and $J_0(x_{0,k})=0$ for every $k\in \mathbb{N}^\star$
that
\begin{align*}
  \text{Res}\Big\{F_1(\mathbf{z}),\mathbf{z}=x_{0,k}\Big\}
  & = \Big(J_1(Xx_{0,k}) \mathbf{Y}_0(x_{0,k}) - \mathbf{Y}_1(Xx_{0,k})J_0(x_{0,k})\Big)
  \frac{F(x_{0,k})}{J'_0(x_{0,k})}\\
  & = -\frac{J_1(Xx_{0,k})J_1(Yx_{0,k})}{x^2_{0,k}+\varepsilon^2}\frac{x_{0,k}\mathbf{Y}_0(x_{0,k})}{J_1(x_{0,k})}\\
  & =  -\frac{2}{\pi}\frac{1}{x^2_{0,k}+\varepsilon^2}\frac{J_1(Xx_{0,k})J_1(Yx_{0,k})}{J^2_1(x_{0,k})},
\end{align*}
which together with \eqref{eq:poles_ids} gives the desired formula \eqref{eq:ids_residue_theorems}.

\noindent Now, we move to the computation of the right-hand side of \eqref{eq:ids_residue_theorems}.
Note that $F(\mathbf{z})$ only contains poles at $\mathbf{z}=\pm \varepsilon i$, we directly compute to deduce that
\begin{align*}
  \text{Res}\Big\{\Big(J_1(Xz) \mathbf{Y}_0(\mathbf{z})-\mathbf{Y}_1(X\mathbf{z})J_0(\mathbf{z})\Big)
  \frac{F(\mathbf{z})}{J_0(\mathbf{z})},\varepsilon i\Big\}
  = \Big(J_1(X\varepsilon i)\mathbf{Y}_0(\varepsilon i)-\mathbf{Y}_1(X\varepsilon i)J_0(\varepsilon i)\Big)
  \frac{J_1(Y\varepsilon i)}{2J_0(\varepsilon i)},
\end{align*}
and
\begin{align*}
  \text{Res}\Big\{\Big(J_1(Xz)\mathbf{Y}_0(\mathbf{z})-&\mathbf{Y}_1(X\mathbf{z})J_0(\mathbf{z})\Big)
  \frac{F(\mathbf{z})}{J_0(\mathbf{z})},- \varepsilon i\Big\}\\
  &= \Big(J_1(-X\varepsilon i)\mathbf{Y}_0(-\varepsilon i)-\mathbf{Y}_1(-X\varepsilon i)J_0(-\varepsilon i)\Big)
  \frac{J_1(-Y\varepsilon i)}{2J_0(-\varepsilon i)}.
\end{align*}
Based on following identities that (e.g. see \cite[Sec. 3.1]{MOS66})
\begin{align*}
  \mathbf{I}_0(x)=J_0(ix),\quad i\mathbf{I}_1(x)=J_1(ix), \quad J_1(-\mathbf{z})=-J_1(\mathbf{z}),\quad J_0(-\mathbf{z})=J_0(\mathbf{z}),
\end{align*}
and
\begin{align*}
  \mathbf{Y}_0(ix)=i\mathbf{I}_0(x)-\frac{2}{\pi}\mathbf{K}_0(x),\quad \mathbf{Y}_1(ix)=-\mathbf{I}_1(x)+\frac{2i}{\pi}\mathbf{K}_1(x),
\end{align*}
and
\begin{align*}
  \mathbf{K}_0(-x)=\mathbf{K}_0(x) + i\pi \mathbf{I}_0(x),\quad \mathbf{K}_1(-x)=-\mathbf{K}_1(x)+i\pi \mathbf{I}_1(x),
\end{align*}
we infer that
\begin{align*}
  & \sum_{k\geqslant 1}\frac{1}{x^2_{0,k}+\varepsilon^2}\frac{J_1(Xx_{0,k})J_1(Yx_{0,k})}{J^2_1(x_{0,k})} \\
  & \qquad = \frac{\pi }{4}\frac{i\mathbf{I}_1(Y\varepsilon)}{2\mathbf{I}_0(\varepsilon)}
  \Big(i\mathbf{I}_1(X\varepsilon)(\mathbf{Y}_0(\varepsilon i)
  + \mathbf{Y}_0(-\varepsilon i))
  + \mathbf{I}_0(\varepsilon)(\mathbf{Y}_1(-X\varepsilon i)-\mathbf{Y}_1(X\varepsilon i))\Big) \\
  & \qquad = \frac{\pi i}{8}\frac{\mathbf{I}_1(Y\varepsilon)}{\mathbf{I}_0(\varepsilon)}\bigg(i\mathbf{I}_1(X\varepsilon)
  \Big(-\frac{4}{\pi}\Big)\mathbf{K}_0(\varepsilon)+\mathbf{I}_0(\varepsilon)
  \Big(-\frac{4i}{\pi}\Big) \mathbf{K}_1(X\varepsilon) \bigg) \\
  & \qquad = \frac{\mathbf{I}_1(Y\varepsilon)}{2\mathbf{I}_0(\varepsilon)}
  \Big(\mathbf{I}_1(X\varepsilon)\mathbf{K}_0(\varepsilon)+\mathbf{I}_0(\varepsilon)\mathbf{K}_1(X\varepsilon) \Big),
\end{align*}
as desired.
\end{proof}

The terminology of completely monotone functions, appearing in assumption \eqref{A1}, is important in our paper.
\begin{definition}\label{def:cmf}
A function $f:(0,\infty)\to \mathbb{R}$ is said to be \textit{completely monotone}
if it is of class $C^{\infty}$ and it satisfies
\begin{align*}
  (-1)^nf^{(n)}(t)\geqslant 0\qquad \forall t>0, \quad \forall  n\in \mathbb{N}.
\end{align*}
\end{definition}
\vskip0.5mm
\noindent The typical example is $f(t)= t^{-\alpha},$ with $\alpha\geqslant 0$.
One can refer, for instance, to \cite{SSV10} for various examples of completely monotone functions.
\\
The following result is fundamental in the theory of completely monotone functions.
For more details, e.g. see Theorem 1.4 in \cite{SSV10}.
\begin{lemma}[Bernstein's theorem]\label{lem:bernstein}
Let $f:(0,\infty)\mapsto \mathbb{R}$ be a completely monotone function. Then it is the Laplace transform of a unique nonnegative measure
$\mu$ on $[0,\infty)$, that is,
\begin{align*}
  \forall\, t>0,\quad f(t)=\int_{0}^{\infty}e^{-tx}\dd \mu(x)\triangleq \mathcal{L}(\mu)(t).
\end{align*}
Conversely, whenever $\mathcal{L}(\mu)(t)<\infty$ for every $t>0$, the function $t\mapsto \mathcal{L}(\mu)$ is a completely monotone function.
\end{lemma}

\noindent We also recall the following useful result on the propagation of higher regularity/integrability of
completely monotone functions (see \cite[Lemma 6.2]{HXX23b} for the proof).
\begin{lemma}\label{lem:int}
The following statements hold true.
\begin{enumerate}
\item Assume that $f$ is a completely monotone function satisfying
\begin{equation}\label{eq:f-cond1}
  \int_0^{t_0} |f(t)| t^\beta \dd t <\infty,\quad \textrm{for some}\; \beta\in (-1,\infty)\;\textrm{and }\; t_0>0,
\end{equation}
then we have
\begin{align*}
  \int_0^{t_0}|f^{(k)}(t)| t^{k+\beta}\dd t \leqslant C_{k,\beta}\int_0^{t_0}f(t)t^\beta\dd t,\quad \forall k\in \mathbb{N}.
\end{align*}
\item
Assume that $f(t)$ is a smooth function satisfying \eqref{eq:f-cond1} and $f'(t)$ is with constant sign, then we have
\begin{align}\label{es:f-weig}
  \int_0^{t_0}|f'(t)| t^{1+\beta}\dd t\leqslant
  (1+\beta) \int_0^{t_0}|f(t)| t^\beta \dd x + |f(t_0)|t_0^{1+\beta}.
\end{align}
\end{enumerate}
\end{lemma}

In the spectral study of the linearized operator, the following Mikhlin multiplier type theorem
for an operator defined on a periodic function will be used; see e.g. \cite[Theorem 4.5]{AB04} for the proof.
\begin{lemma}\label{lem:multiplier-lemma}
Given $\{a_n\}_{n\in \mathbb{Z}}$ and $h\in L^1(\TT)$, and define the operator
\begin{align*}
  T h(\theta)=\sum_{n\in \mathbb{Z}} a_n \widehat{h}_ne^{in\theta},
\end{align*}
where $\widehat{h}_n= \frac{1}{2\pi}\int_{\TT} h(\theta) e^{-in\theta} \dd \theta$ is the $n$-th Fourier coefficient of
the periodic function $h$.
Assume that
\begin{align*}
  \sup_{n\in \ZZ}|a_n|<\infty,\quad\textrm{and}\quad \sup_{n\in \ZZ}|n(a_{n+1}-a_n)|<\infty,
\end{align*}
then the operator $T$ is bounded in $C^{k+\alpha}(\TT)$, for any $k\in \NN$ and $\alpha\in (0,1)$.
\end{lemma}

\noindent
Finally, the Crandall-Rabinowitz theorem from the local bifurcation theory plays a fundamental role in our paper,
and for the proof we refer to \cite{C-R71}.
\begin{theorem}[Crandall-Rabinowitz's theorem]\label{thm:C-R}
Let $X$ and $Y$ be two Banach spaces, $V$ a neighborhood of $\mathbf{0}$ in $X$ and let
$F : \RR \times V \to Y$ be with the following  properties:
\begin{enumerate}
\item $F (\lambda, \mathbf{0}) = \mathbf{0}$ for any $\lambda\in \RR$.
\item The partial derivatives $\partial_\lambda F$, $\partial_x F$ and $\partial_\lambda \partial_{x}F$ exist and are continuous.
\item $N(\mathcal{L}_0)$ and $Y/R(\mathcal{L}_0)$ are one-dimensional.
\item {\it Transversality assumption}: $\partial_\lambda \partial_x F(0, \mathbf{0})x_0 \not\in R(\mathcal{L}_0)$, where
\begin{align*}
  N(\mathcal{L}_0) = \mathrm{span}\{x_0\}, \quad \mathcal{L}_0\triangleq \partial_x F(0,\mathbf{0}).
\end{align*}
\end{enumerate}
If $Z$ is any complement of $N(\mathcal{L}_0)$ in $X$, then there is a neighborhood $U$ of $(0,\mathbf{0})$ in $\RR \times X$,
an interval $(-a,a)$, and continuous functions $\varphi: (-a,a) \to \RR$,
$\psi: (-a,a) \to Z$ such that $\varphi(0) = 0$, $\psi(0) = \mathbf{0}$ and
\begin{align*}
  F^{-1}(\mathbf{0})\cap U=\Big\{\big(\varphi(\xi), \xi x_0+\xi\psi(\xi)\big)\,:\,\vert \xi\vert<a\Big\}
  \cup\Big\{(\lambda,\mathbf{0})\,:\, (\lambda,\mathbf{0})\in U\Big\}.
\end{align*}
\end{theorem}

\end{document}